\documentclass[a4paper]{amsart}
\usepackage{amsmath,amsthm,amssymb,latexsym,epic,bbm,comment,mathrsfs}
\usepackage{graphicx,enumerate,stmaryrd,xcolor,tikz}
\usepackage[all,2cell]{xy}
\xyoption{2cell}

\newtheorem{theorem}{Theorem}
\newtheorem{lemma}[theorem]{Lemma}
\newtheorem{corollary}[theorem]{Corollary}
\newtheorem{proposition}[theorem]{Proposition}

\newtheorem{remark}[theorem]{Remark}
\usepackage[all]{xy}
\usepackage[active]{srcltx}
\usepackage[parfill]{parskip}
\usepackage{enumerate}

\newcommand{\tto}{\twoheadrightarrow}

\begin{document}

\title[Some homological properties of category $\mathcal{O}$, VII]
{Some homological properties of category $\mathcal{O}$, VII}

\author[V.~Mazorchuk]
{Volodymyr Mazorchuk}

\begin{abstract}
We describe Calabi-Yau objects in the regular block of the 
(parabolic) BGG category $\mathcal{O}$ associated to a
semi-simple finite dimensional complex Lie algebra.
Each such object comes with a natural transformation from
the Serre functor to a shifted identity whose evaluation
at that object is an isomorphism.
\end{abstract}

\maketitle

{\em Dedicated to the memory of Andriy Vitaliyovych Olshanskyy.}

\section{Introduction}\label{s1}

\subsection{Motivation}\label{s1.1}

This paper is motivated by the ideas and the results of the two 
very interesting preprints \cite{EH1,EH2} which outline the theory of 
categorical diagonalization and study it in the context of
diagrammatic Soergel bimodules. 
The ideas of these preprints have 
shown to be very useful for the study of 
Soergel bimodules and their applications to low dimensional topology,
see \cite{GHW,GHMN,EH3,Ho,MMMTZ2,HRW} and references therein.
The present paper is an attempt to understand
these ideas and results in the context of the BGG category $\mathcal{O}$, cf.
\cite{BGG,Hu}.

Category $\mathcal{O}$ is a certain category of modules associated to a fixed
triangular decomposition of a semi-simple finite dimensional complex Lie algebra.
The principal block $\mathcal{O}_0$ of $\mathcal{O}$ is equivalent to the category
of finite dimensional modules over a certain finite dimensional, associative,
Koszul algebra $A$. This block is equipped with an action of the finitary
bicategory $\mathscr{P}$ of projective functors. There is a combinatorial model
for the action of $\mathscr{P}$ on $\mathcal{O}_0$, due to Soergel, which uses
the coinvariant algebra $\mathtt{C}$ of the Weyl group of the underlying Lie
algebra. This model represents the projective functors acting on $\mathcal{O}_0$
via the so-called {\em Soergel bimodules} (over $\mathtt{C}$). Grothendieck
decategorification of the graded version of this model gives the right regular
representation of the associated Hecke algebra, where the Kazhdan-Lusztig
basis of the latter appears naturally as the image of the indecomposable 
Soergel bimodules.

Lifting the action of $\mathscr{P}$ on $\mathcal{O}_0$ to the level of the derived
category (or, equivalently, to the level of the homotopy category of complexes),
gives, in particular, an action of Rouquier's $2$-braid group on
$\mathcal{D}^b(\mathcal{O}_0)$ by the so-called {\em shuffling functors}.
The canonical generator of the center of the braid group corresponds to
the square of the longest element of the Weyl group and is sometimes 
referred to as the {\em full twist}, see \cite{EH2}. At the level of 
the category $\mathcal{D}^b(\mathcal{O}_0)$, this full twist is 
a Serre functor, denoted $\mathbb{S}$, as shown in \cite{MS2}.
Understanding the homological properties of $\mathbb{S}$ is important
for the study of the category $\mathcal{O}_0$.

The paper \cite{EH2} discusses, in the setup of Soergel bimodules over the
polynomial algebra, and especially in type $A$, the so-called {\em categorical 
diagonalization} of the full twist. This is a special kind of structure given 
by certain  natural transformations from this full twist to the identity,
shifted both in the direction of homological position and in the 
direction of Koszul grading. 

The notion of a Serre functor $\mathbb{S}$ on a $\Bbbk$-linear category
$\mathcal{C}$ was defined in \cite{BoKa}. The principal defining
property of $\mathbb{S}$ are isomorphisms 
$\mathcal{C}(X,\mathbb{S}(Y))\cong \mathcal{C}(Y,X)^*$,
natural in both, $X$ and $Y$, for all $X,Y\in \mathcal{C}$ .
In \cite{CZ},  an object $Y\in \mathcal{C}$ is called a 
{\em Calabi-Yau object of dimension $i$}, where $i\in\mathbb{Z}$, provided that
$\mathcal{C}(X,Y[i])\cong \mathcal{C}(Y,X)^*$, or, equivalently,
$\mathbb{S}Y\cong Y[i]$. Understanding Calabi-Yau object 
provides important information about the structure and behavior of
the Serre functor. 

Combining the ideas of \cite{EH1,EH2} with those of  \cite{CZ},  
given a Calabi-Yau object $Y$ of dimension $i$, it seems reasonable to ask
whether there is a natural transformation from $\mathbb{S}$ to the
$i$-th shift of the identity in the homological position whose evaluation
at $Y$ is an isomorphism. That is exactly the main question we address in the
present paper in the setup of the category $\mathcal{O}_0$.

\subsection{Results}\label{s1.2}
To be able to formulate the main results of this paper, we need 
to introduce some notation. Let $W$ be the Weyl group of our
Lie algebra. Then the isomorphism classes of simple objects in
$\mathcal{O}_0$ are indexed by the elements of $W$. For
$w\in W$, we denote by $L_w$ the simple object corresponding to
$w$. Here $L_e$ is the trivial module. Also, the indecomposable
projective endofunctors of $\mathcal{O}_0$ are indexed 
by the elements of $W$ and denoted $\theta_w$. The Kazhdan-Lusztig
combinatorics splits $W$ into left, right and two-sided Kazhdan-Lusztig cells.
Each left and each right KL-cell contains a unique distinguished involution
called a {\em Duflo element}. We denote by $\mathbf{D}$ the set of
all Duflo elements in $W$. If $x,y\in W$ belong to the same right
cell, we write $x\sim_R y$. We denote by $\mathbf{a}$ Lusztig's
$\mathbf{a}$-function on $W$. As usual, we denote by $[i]$ the shift
in homological position and by $\langle i\rangle$ the shift in
the Koszul grading.

We can now state our first main result, Theorem~\ref{conj1} in the text,
as follows:
\vspace{2mm}

{\bf Theorem~A.}
{\em
For $d\in\mathbf{D}$ and $w\in W$ such that $w\sim_R d$, the object 
$\theta_w L_d$ is a Calabi-Yau object in $\mathcal{D}^b(\mathcal{O}_0)$
of dimension $2\mathbf{a}(w_0d)$. More precisely, in the graded picture, we have 
\begin{displaymath}
\mathbb{S}(\theta_w L_{d})\cong
\theta_w L_{d}
\langle 2\big(\mathbf{a}(d)-\mathbf{a}(w_0d)\big)\rangle
[2\mathbf{a}(w_0d)].
\end{displaymath}
}
\vspace{2mm}

Our second main result, which is 
Theorem~\ref{prop-s8.25-1} in the text, asserts the following:
\vspace{2mm}

{\bf Theorem~B.}
{\em
For any $d\in \mathbf{D}$, there exists a natural transformation
\begin{displaymath}
\alpha_d:\mathbb{S}\to 
\mathrm{Id}\langle 2(\mathbf{a}(d)-\mathbf{a}(w_0d))\rangle[2\mathbf{a}(w_0d)]
\end{displaymath}
such that, for any $w\sim_R d$, the evaluation of $\alpha_d$ at 
$\theta_{w}L_{d}$ is an isomorphism.
}
\vspace{2mm}

We note that the objects of the form $\theta_w L_d$ appeared previously in
\cite{MS5} in the context of categorification of 
Wedderburn's basis in type $A$. They also appear as projective-injective
objects for certain subquotients of $\mathcal{O}_0$, see \cite{MS4}.
The latter was probably the main motivation to investigate
these objects as potential Calabi-Yau object. Projective-injective
objects (with isomorphic top and socle) are obvious Calabi-Yau objects
of dimension $0$ in the very general case. Projective-injective
objects for subquotient categories are, of course, not 
Calabi-Yau objects in general. However, the category $\mathcal{O}_0$
and, especially, the action of projective functors on this category
makes the situation ``nice'' enough for the above results to
be possible.

We propose the following conjecture:

{\bf Conjecture~C.}
{\em
For a fixed $d\in\mathbf{D}$, any Calabi-Yau object in 
$\mathcal{D}^b(\mathcal{O}_0)$ of dimension
$2\mathbf{a}(w_0d)$ admits a filtration with subquotients
of the form $\theta_w L_{d'}$, where
$d'\in\mathbf{D}$ is such that $\mathbf{a}(d)=\mathbf{a}(d')$
and $w\in W$ such that $w\sim_R d'$.
}

\subsection{Methods}\label{s1.3}

Proofs in the paper involve a wide variety of various methods and techniques,
most of which were developed only very recently.  The general background of 
all arguments is, as usual, 
\begin{itemize}
\item the Kazhdan-Lusztig combinatorics from \cite{KL}
and its connection to $\mathcal{O}_0$ as established in \cite{BB,BK,EW};
\item Soergel's combinatorial description of 
both $\mathcal{O}_0$ and $\mathscr{P}$ from \cite{So,So3};
\item Koszul duality from \cite{BGS}.
\end{itemize}

Further, we extensively use the ideas and results from several of the
previous papers in the ``Some homological properties of category $\mathcal{O}$''
series, namely, from \cite{Maz1,Maz2,CM5,KMM,KMM2}. This includes:
\begin{itemize}
\item explicit information on the homological invariants
(for instance, on projective dimension) of the structural
objects in $\mathcal{O}_0$;
\item explicit information on the structure of resolutions
of structural modules in $\mathcal{O}_0$;
\item Auslander regularity of $\mathcal{O}_0$;
\item latest advances in our understanding of Kostant's problem
for simple highest weight modules.
\end{itemize}
The present paper would hardly be possible without the insight and the 
intuition about some very tiny special features of various complexes
of modules in the category $\mathcal{O}_0$ which was developed during
the work on all previous papers in this series.

Finally, we crucially employ the methods and techniques from
$2$-representation theory which allow us to study the
category $\mathcal{O}_0$ from the point of view of the 
bicategory $\mathscr{P}$ acting on $\mathcal{O}_0$.
Here, in particular, we heavily rely on
\begin{itemize}
\item properties of categories with full projective functors
as developed in \cite{Kh};
\item abstract theory of birepresentations of finitary
bicategories as developed in \cite{MM1,MMMTZ1} and its concrete
realization via the action of projective functors on category 
$\mathcal{O}$ as studied in \cite{MMMTZ2};
\item $2$-categorical approach to Kostant's problem
developed in \cite{KMM2}.
\end{itemize}

\subsection{Structure of the paper}\label{s1.4}

The paper is organized as follows:
In Section~\ref{s2} we collected all necessary preliminaries 
on category $\mathcal{O}$ and various technical tools that 
are used in the paper. Section~\ref{s3} is devoted to the first 
main result, Theorem~\ref{conj1}, which describes a family
of Calabi-Yau objects in the principal block of $\mathcal{O}$.
We start by proving the result is some special cases.
This eventually directs us towards the idea how to prove 
Theorem~\ref{conj1} in the general case.
Section~\ref{s8} is dedicated to the second 
main result, Theorem~\ref{prop-s8.25-1}, which describes
certain natural transformations from the Serre functor to
the (shifted) identity, whose evaluations at the
corresponding Calabi-Yau objects are isomorphisms.
Section~\ref{s6} contains some bonus homological properties
of category $\mathcal{O}$ which are inspired by the proofs
of Theorem~\ref{conj1} and Theorem~\ref{prop-s8.25-1}.
Section~\ref{s9} provides a generalizations of the two
main results to the case of the parabolic category 
$\mathcal{O}$. Finally, in Section~\ref{s7} we collected
a number of explicit examples, both low rank and general.
These examples were quite helpful to the author from the
point of view of explaining, better understanding and 
checking the main results of the paper.

We are very far away from solving all interesting problems
in the area. A number of guesses, expectations and conjectures are 
formulated throughout the paper.
\vspace{5mm}

\textbf{Acknowledgements:} This research was partially supported by
the Swedish Research Council. The author thanks Chih-Whi Chen for
stimulating discussions at the very early stages of this  project.
The author thanks the referee  for very helpful comments.
\vspace{5mm}

\section{Preliminaries on category $\mathcal{O}$}\label{s2}

\subsection{Conventions}\label{s2.1}

In this paper we work over the field $\mathbb{C}$ of complex numbers.
In particular, all algebras are assumed to be $\mathbb{C}$-algebras
and all categories and functors are assumed
to be $\mathbb{C}$-linear. We denote by $*$ the usual contravariant
$\mathbb{C}$-duality $\mathrm{Hom}_{\mathbb{C}}({}_-,\mathbb{C})$.

For a Lie algebra $\mathfrak{a}$, we denote by $U(\mathfrak{a})$
its universal enveloping algebra. We have the canonical
isomorphism between the category $\mathfrak{a}$-Mod of all $\mathfrak{a}$-modules
and the category $U(\mathfrak{a})$-Mod of all $U(\mathfrak{a})$-modules.
It restricts to an isomorphism between the category $\mathfrak{a}$-mod 
of all finitely generated $\mathfrak{a}$-modules
and the category $U(\mathfrak{a})$-mod of all finitely generated 
$U(\mathfrak{a})$-modules.
We denote by $Z(\mathfrak{a})$ the center of $U(\mathfrak{a})$.

It what follows, {\em graded} means $\mathbb{Z}$-graded.

\subsection{Definition and the principal block}\label{s2.2}

Let $\mathfrak{g}$ be a semi-simple, finite dimensional, complex Lie 
algebra with a fixed triangular decomposition 
$\mathfrak{g}=\mathfrak{n}_-\oplus \mathfrak{h}\oplus\mathfrak{n}_+$,
where $\mathfrak{h}$ is a Cartan subalgebra. Associated to this datum,
we have the corresponding BGG category $\mathcal{O}$ defined as the
full subcategory of $U(\mathfrak{g})$-mod, consisting of all objects, the action of
$\mathfrak{h}$ on which is diagonalizable and the action of $U(\mathfrak{n}_+)$
on which is locally finite, see \cite{BGG,Hu}.

For $\lambda\in \mathfrak{h}^*$, we denote by $\Delta(\lambda)$ the corresponding
Verma module with highest weight $\lambda$ and by $L(\lambda)$ the simple
top of $\Delta(\lambda)$, see \cite[Chapter~7]{Di}.

Let $\mathbf{R}$ denote the root system of the pair $(\mathfrak{g},\mathfrak{h})$.
Our fixed triangular decomposition of $\mathfrak{g}$ induces a decomposition of $\mathbf{R}$
into a disjoint union of the set $\mathbf{R}_+$ of all positive roots and
the set $\mathbf{R}_-$ of all negative roots. Let $\boldsymbol{\pi}$ be the
corresponding basis of $\mathbf{R}$. Denote by $\rho$ the half of the sum of
all positive roots.

We also have the Weyl group $W$ of the root system $\mathbf{R}$ which acts on $\mathfrak{h}^*$
via the defining action, written $(w,\lambda)\mapsto w(\lambda)$,
for $w\in W$ and $\lambda\in \mathfrak{h}^*$. We also have the dot-action of 
$W$ on $\mathfrak{h}^*$ given by $w\cdot\lambda:=w(\lambda+\rho)-\rho$.

Category $\mathcal{O}$ admits a decomposition
\begin{displaymath}
\mathcal{O}\cong \bigoplus_{\chi:Z(\mathfrak{g})\to\mathbb{C}} \mathcal{O}_\chi 
\end{displaymath}
with respect to the action of the center $Z(\mathfrak{g})$ defined in the terms
of a central character $\chi$, where $\mathcal{O}_\chi$ consist of 
all modules $M\in \mathcal{O}$ such that, for every $m\in M$ and 
$z\in Z(\mathfrak{g})$, we have $(z-\chi(z))^km=0$, for $k\gg 0$.

Let $\chi_{{}_0}$ denote the central character of the trivial $\mathfrak{g}$-module
$L(0)$. In this paper we will mostly study the principal block 
$\mathcal{O}_0:=\mathcal{O}_{\chi_{{}_0}}$ of $\mathcal{O}$. This block coincides
with the Serre subcategory of $\mathcal{O}$ generated by $L_w:=L(w\cdot 0)$, where
$w\in W$. The category $\mathcal{O}_0$ is equivalent to the category $A$-mod of 
finite dimensional modules over some basic, finite dimensional,
associative algebra $A$ of finite global dimension.

For $w\in W$, we set $\Delta_w:=\Delta(w\cdot 0)$ and also denote by
$P_w$ and $I_w$ the indecomposable projective cover and injective
envelope of $L_w$, respectively. 

We denote by $\star$ the usual simple preserving duality on $\mathcal{O}$ and set
$\nabla_w:=\Delta_w^\star$, for $w\in W$. The latter modules are usually called
the {\em dual Verma modules}.

We have the full subcategories $\mathcal{F}(\Delta)$ and $\mathcal{F}(\nabla)$ 
of $\mathcal{O}_0$ consisting of all objects that admit a filtration with Verma
or dual Verma subquotients, respectively. For $w\in W$, we denote by $T_w$ the 
corresponding tilting module, that is the unique, up to isomorphism, indecomposable
object in $\mathcal{F}(\Delta)\cap\mathcal{F}(\nabla)$ for which there is 
an embedding $\Delta_w\subset T_w$ such that the corresponding cokernel
is in $\mathcal{F}(\Delta)$, see \cite{CI,Ri}. We have $T_w^\star\cong T_w$.

For $M\in \mathcal{O}_0$, we denote by
\begin{itemize}
\item $\mathcal{P}_\bullet(M)$ a minimal projective resolution of $M$;
\item $\mathcal{I}_\bullet(M)$ a minimal injective coresolution of $M$;
\item $\mathcal{T}_\bullet(M)$ a minimal complex of tilting modules
that represents $M$ in $\mathcal{D}^b(\mathcal{O}_0)$.
\end{itemize}

\subsection{Grading and combinatorics}\label{s2.3}

The algebra $A$ is Koszul, see \cite{So,BGS}, in particular, 
it has the corresponding positive Koszul grading 
$\displaystyle A\cong \bigoplus_{i\geq 0}A_0$ with semi-simple $A_0$.
The corresponding category $A$-fgmod of finite dimensional
graded $A$-modules and homogeneous module homomorphisms of degree zero
is called the {\em graded lift of $\mathcal{O}_0$}
and denoted ${}^{\mathbb{Z}}{\mathcal{O}}_0$. We have the obvious canonical functor 
$\mathrm{Forget}$ from 
${}^{\mathbb{Z}}{\mathcal{O}}_0$ to $A$-mod$\,\cong \mathcal{O}_0$ which
forgets the grading. For $n\in\mathbb{Z}$, we denote by $ \langle n\rangle$
the usual degree shift endofunctor of ${}^{\mathbb{Z}}{\mathcal{O}}_0$
which maps elements of degree $m$
to elements of degree $m-n$, for $m\in\mathbb{Z}$.
As usual, we use the lowercase notation $\mathrm{hom}$ and 
$\mathrm{ext}$ for homogeneous homomorphisms and extensions of degree zero
as compared to the ungraded notation $\mathrm{Hom}$ and 
$\mathrm{Ext}$.

All structural objects in $\mathcal{O}_0$ admit graded lifts, in the sense
that they belong to the image of $\mathrm{Forget}$. Moreover, for indecomposable
structural objects, the corresponding graded lift (i.e. the preimage under
$\mathrm{Forget}$) is unique up to isomorphism and graded shift, see \cite{St}.
We fix the following graded lifts of the structural modules
which we will denote in the same way as ungraded modules, for $w\in W$:
\begin{itemize}
\item $L_w$ is concentrated in degree zero;
\item $P_w$ and $\Delta_w$ have tops concentrated in degree zero;
\item $I_w$ and $\nabla_w$ have socles concentrated in degree zero;
\item $T_w$ has the unique $L_w$-subquotient in degree zero.
\end{itemize}
The duality $\star$ also admits a graded lift, denoted by the same symbol.
It satisfies $\star\circ \langle n\rangle=\langle -n\rangle\circ\star$, 
for $n\in\mathbb{Z}$.

Let $\ell$ be the length function on $W$ and $w_0$ the longest element of $W$.
Let $\mathbf{H}$ be the Hecke algebra of $W$. It is an algebra 
over $\mathbb{Z}[v,v^{-1}]$. It has the {\em standard basis}
$\{H_w\,:\,w\in W\}$ and the {\em Kazhdan-Lusztig (KL) basis}
$\{\underline{H}_w\,:\,w\in W\}$, see \cite{KL}.
We use the normalization of \cite{So2}. The entries of the 
transformation matrix between the two bases are called 
{\em KL-polynomials} and denoted $p_{x,y}$, where $x,y\in W$.
We have $\underline{H}_y=\sum_{x}p_{x,y}H_x$.

The Grothendieck group of ${}^{\mathbb{Z}}{\mathcal{O}}_0$
is isomorphic to $\mathbf{H}$ by sending $[\Delta_w]$
to $H_w$, for $w\in W$, and letting $v$ act as 
$\langle -1\rangle$. By the Kazhdan-Lusztig conjecture,
proved in \cite{BK,BB}, see also \cite{EW},
this isomorphism maps $[P_w]$ to $\underline{H}_w$, for $w\in W$.

As usual, we have the theory of left, right and two-sided
KL-orders and cells. For  $x,y\in W$, we have $x\leq_L y$
provided that there is $w\in W$ such that $\underline{H}_y$
appears with a non-zero coefficient in the decomposition
of $\underline{H}_w\underline{H}_x$ as a linear combination
of elements in the KL-basis. Then $\leq_L$ is a partial pre-order,
called the {\em left KL-pre-order} and the equivalence classes
with respect to it are called the {\em left KL-cells}.
We write $x\sim_L y$ if $x\leq_L y$ and $y\leq_L x$.
The {\em right} and the {\em two-sided} pre-orders $\leq_R$ and $\leq_J$ 
and the corresponding cells $\sim_R$ and $\sim_J$ are defined
similarly using multiplication on the right or from both sides,
respectively.  We denote by $\sim_H$ the equivalence relation
given by the intersection of $\sim_R$ and $\sim_L$.

Each left (and each right) KL-cell contains a unique element 
called {\em Duflo element} or {\em Duflo involution} 
(this element is an involution). We denote by $\mathbf{D}$
the set of all Duflo involutions. In type $A$, all involutions
are Duflo involutions. We refer to \cite{MM1} for details.

We also have {\em Lusztig's $\mathbf{a}$-function}
$\mathbf{a}:W\to\mathbb{Z}_{\geq 0}$ introduced in \cite{Lu}.
For $w\in W$, the value $\mathbf{a}(w)$ is defined as the 
maximum $v$-degree, taken over all $x,y\in W$, with which 
$\underline{H}_w$ appears in the decomposition of 
$\underline{H}_x\underline{H}_y$ into a linear combination
of the elements of the KL basis. Note that, to attain 
this maximal degree requires both $x\sim_J w$
and $y\sim_J w$.
The function $\mathbf{a}$ is constant on the two-sided KL-cells
of $W$, moreover, for $x\in W$, we have $\mathbf{a}(x)=\ell(x)$
provided that $x$ is the longest element of some parabolic
subgroup of $W$. These two properties identify $\mathbf{a}$
uniquely in type $A$. In all types, the function $\mathbf{a}$ is
strictly monotone with respect to all  KL orders.

\subsection{Functors}\label{s2.4}

For each $w\in W$, we denote by $\theta_w$ the indecomposable 
projective endofunctor of ${\mathcal{O}}_0$ such that 
$\theta_w\Delta_e\cong P_w$, see \cite{BG}. Each $\theta_w$
is exact, biadjoint to $\theta_{w^{-1}}$ and admits a
graded lift which we will denote by the same symbol.
We denote by $\mathscr{P}$ the bicategory of projective
endofunctors of ${\mathcal{O}}_0$. It is finitary 
in the sense of \cite{MM1,MMMTZ1}. We also denote by 
${}^{\mathbb{Z}}{\mathscr{P}}$ the bicategory of projective
endofunctors of ${}^{\mathbb{Z}}{\mathcal{O}}_0$. It is 
locally finitary in the sense of \cite{Mac1,Mac2}.
The split Grothendieck group of ${}^{\mathbb{Z}}{\mathscr{P}}$ 
is isomorphic to $\mathbf{H}$ by sending $[\theta_w]$
to $\underline{H}_w$ and with the same convention for 
the action of $v$ as above. The defining action of 
${}^{\mathbb{Z}}{\mathscr{P}}$ on ${}^{\mathbb{Z}}{\mathcal{O}}_0$
decategorifies, by taking the corresponding Grothendieck groups,
to a right regular $\mathbf{H}$-module.

For each simple reflection $s\in W$, we have the corresponding 
shuffling functor $C_s$ defined in \cite{Ca} as the cokernel of the 
adjunction morphism $\theta_e\to\theta_s$ and the coshuffling 
functor $K_s$ defined in \cite{Ca} as the kernel of the 
adjunction morphism $\theta_s\to\theta_e$. The functor 
$C_s$ is left adjoint to $K_s$, in particular, the functor 
$C_s$ is right exact while $K_s$ is left exact. Furthermore, the left derived functor
$\mathcal{L}C_s$ is an auto-equivalence of the bounded derived
category $\mathcal{D}^b({\mathcal{O}}_0)$ whose inverse is
the right derived functor $\mathcal{R}K_s$. 
The functor $\mathcal{L}C_s$ admits a realization 
via tensoring with the complex  $0\to\theta_e\to\theta_s\to 0$
of functors, in which the non-zero morphism is the adjunction morphism,
followed by taking the total complex. The complexes of the  form
$0\to\theta_e\to\theta_s\to 0$ are known as {\em Rickard or 
Rouquier complexes},  see \cite{Rickard,Rouquier}.
Similarly one can describe the coshuffling functors.
Further, the functors
$\mathcal{L}C_s$, where $s$ runs through the set $S$ of all simple
reflections in $W$, satisfy the same braid relations as the 
corresponding simple reflections of $W$, see \cite{Ca,MS1}.
In particular, given a reduced decomposition
$w=s_1s_2\dots s_k$ of some $w\in W$, we can consider
the corresponding shuffling functor
$C_w:=C_{s_k}\circ \dots \circ C_{s_2}\circ C_{s_1}$
(here, reversing the order reflects the fact that the action of 
projective functors on ${\mathcal{O}}_0$ is the right action) 
and let $K_{w^{-1}}$ be its adjoint. We have
$\mathcal{L}C_w:=\mathcal{L}C_{s_k}\circ \dots \circ \mathcal{L}C_{s_2}\circ \mathcal{L}C_{s_1}$
and a similar decomposition for the corresponding coshuffling functors.
As both adjunctions morphisms $\theta_e\to\theta_s$ and 
$\theta_s\to\theta_e$ are gradeable, both shuffling and coshuffling
functors admit the obvious graded lifts.

Finally, for each $w\in W$ we have the corresponding twisting functor $\top_w$ and 
its right adjoint $E_{w^{-1}}$, the Enright completion functor, see \cite{AS,KM,Jo}.
The functor $\top_w$ is right exact while $E_w$ is left exact.
The left derived functor $\mathcal{L}\top_w$ is an auto-equivalence of
$\mathcal{D}^b({\mathcal{O}}_0)$ whose inverse is
the right derived functor $\mathcal{R}E_w$, see \cite{AS}. Further, the functors
$\mathcal{L}\top_s$, where $s\in S$, satisfy the same braid relations as the 
corresponding simple reflections of $W$, see \cite{KM}. Both 
$\top_w$ and $E_w$ are gradeable, see the appendix of \cite{MO}.
Last, but not least, is that both twisting and Enright completion functors 
functorially commute with projective functors, see \cite{AS}.
Consequently, both twisting and Enright completion functors commute
with shuffling and coshuffling functors.

\subsection{Soergel's combinatorial description}\label{s2.5}

Let $\mathtt{C}$ denote the coinvariant algebra of $W$. As was shown in
\cite{So}, the commutative algebra $\mathtt{C}$ is isomorphic to 
$\mathrm{End}_\mathfrak{g}(P_{w_0})$, moreover, the 
so-called {\em Soergel's combinatorial functor}
\begin{displaymath}
\mathbb{V}:=\mathrm{Hom}_{\mathfrak{g}}(P_{w_0},{}_-):
\mathcal{O}_0\to \mathtt{C}\text{-mod}
\end{displaymath}
is full and faithful on projective objects. Consequently, 
the algebra $A$ is isomorphic to the (opposite of the) endomorphism 
algebra of the $\mathtt{C}$-module given by the direct sum of 
all $\mathbb{V}(P_w)$, where $w\in W$. 

For a simple reflection $s$, let $\mathtt{C}^s$ be the algebra
of all $s$-invariants in $\mathtt{C}$. We can now define the
corresponding $\mathtt{C}$-$\mathtt{C}$-bimodule
$\mathtt{B}_s:=\mathtt{C}\otimes_{\mathtt{C}^s}\mathtt{C}$.
Given a reduced expression $w=s_1s_2\dots s_m$ of some $e\neq w\in W$,
consider the $\mathtt{C}$-$\mathtt{C}$-bimodule $\mathtt{B}_w$
defined as 
\begin{displaymath}
\mathtt{B}_w:= \mathtt{B}_{s_1}\otimes_{\mathtt{C}}
\mathtt{B}_{s_2}\otimes_{\mathtt{C}}\dots\otimes_{\mathtt{C}}\mathtt{B}_{s_m}.
\end{displaymath}
Set $\mathtt{B}_e:=\mathtt{C}$.
The additive closure $\mathscr{B}$ inside $\mathtt{C}$-mod-$\mathtt{C}$
of all $\mathtt{B}_w$, where $w\in W$, is a monoidal subcategory of
$\mathtt{C}$-mod-$\mathtt{C}$. Moreover, the category $\mathscr{B}$
is monoidally equivalent to $\mathscr{P}$, see \cite{So3}. The bimodule
$\mathtt{B}_e$ is, clearly, indecomposable. Each $\mathtt{B}_w$
has a unique indecomposable direct summand, denoted $\underline{\mathtt{B}}_w$,
which is not isomorphic to any $\underline{\mathtt{B}}_x$, where 
$x< w$ with respect to the Bruhat order on $W$.

This allows us to construct the $\mathtt{C}$-modules $\mathbb{V}(P_w)$
inductively. To start with, the $\mathtt{C}$-module $\mathbb{V}(P_e)$ is simple.
Furthermore, for $w\in W$, the $\mathtt{C}$-module $\mathbb{V}(P_w)$
is isomorphic to $\underline{\mathtt{B}}_w\otimes_{\mathtt{C}} \mathbb{V}(P_e)$.

The algebra $\mathtt{C}$ is graded, with generators in degree $2$,
and all the story above admits a natural graded lift.

The bimodules in the additive closure of $\underline{\mathtt{B}}_w$
are usually called {\em Soergel bimodules} over the coinvariant
algebra. A similar definition works over the polynomial algebra,
that is the symmetric algebra of $\mathfrak{h}$, and provides a 
combinatorial description for projective functors on the so-called
{\em thick category $\mathcal{O}$}, where the usual category 
$\mathcal{O}$ is a subcategory, see \cite{So3}. Projective functors
on the thick category $\mathcal{O}$ act on the usual 
category $\mathcal{O}$ by restriction, which defines a monoidal functor
from the monoidal category of Soergel bimodules over the polynomial algebra
to the monoidal category of Soergel bimodules over the coinvariant
algebra. This monoidal functor maps the full twist of the former 
category to the full twist of the latter.

\subsection{Harish-Chandra bimodules}\label{s2.6}

Yet another way to look at the category $\mathscr{P}$ is through the prism
of Harish Chandra bimodules. Consider the category $\mathscr{H}$ of
Harish Chandra $U(\mathfrak{g})$-$U(\mathfrak{g})$-bimodules, that is 
finitely generated $U(\mathfrak{g})$-$U(\mathfrak{g})$-bimodules, the adjoint
action of $\mathfrak{g}$ on which is locally finite and has finite multiplicities
for all simple subquotients.
A typical example of such a bimodule would be the quotient of $U(\mathfrak{g})$
by the two-sided ideal generated be some ideal in $Z(\mathfrak{g})$ of finite codimension.

Let $\mathbf{m}$ be the maximal ideal of $Z(\mathfrak{g})$ given by the kernel of 
$\chi_{{}_0}$. Consider the full subcategory ${}^{\infty}_{\,0}\hspace{-1mm}\mathscr{H}_0^1$
of $\mathscr{H}$ consisting of all object on which the right action of $\mathbf{m}$
is zero while the left action of $\mathbf{m}$ is locally nilpotent. Then
${}^{\infty}_{\,0}\hspace{-1mm}\mathscr{H}_0^1$ is equivalent to $\mathcal{O}_0$ by \cite{BG}.

Let $\mathbf{n}$ be the  ideal of $Z(\mathfrak{g})$ given by the kernel of 
the projection of $Z(\mathfrak{g})$ onto the endomorphism algebra of $P_{w_0}$,
see \cite{So}. In particular, we have $Z(\mathfrak{g})/\mathbf{n}\cong\mathtt{C}$.
Note that the latter algebra is isomorphic to the center of $\mathcal{O}_0$
and hence also of the algebra $A$. In particular, $\mathbf{n}$ has finite codimension, 
in fact, it contains $\mathbf{m}^{\ell(w_0)}$.
Also, the identity functor on $\mathcal{O}_0$ can be realized as tensoring with 
the Harish Chandra bimodule $U(\mathfrak{g})/U(\mathfrak{g})\mathbf{n}$.
Consider the full subcategory ${}^{\infty}_{\,0}\hspace{-1mm}\mathscr{H}_0^{\mathtt{C}}$
of $\mathscr{H}$ consisting of all object on which the right action of $\mathbf{n}$
is zero while the left action of $\mathbf{m}$ is locally nilpotent. Then
${}^{\infty}_{\,0}\hspace{-1mm}\mathscr{H}_0^{\mathtt{C}}$ is a monoidal 
category (in which the monoidal structure is given by tensoring over 
$U(\mathfrak{g})$) and it is monoidally equivalent to  $\mathscr{P}$. 

\subsection{Serre functor}\label{s2.7}

Given a $\mathbb{C}$-linear additive category $\mathcal{C}$ with finite 
dimensional morphism spaces, a {\em right Serre functor} on $\mathcal{C}$
is an endofunctor $\mathbb{S}$ of $\mathscr{C}$, for which there are
isomorphisms $\mathcal{C}(i,\mathbb{S}(j))\cong  \mathcal{C}(j,i)^*$,
for all $i,j\in \mathbb{C}$, natural in both $i$ and $j$.
If exists, a right Serre functor is unique, up to isomorphism, and
it commutes, up to isomorphism, with any auto-equivalence of $\mathcal{C}$.
A right Serre functor which is itself an equivalence is called a
{\em Serre functor}. We refer to \cite{BoKa} for details.

If $B$ is a finite dimensional associative algebra of finite global
dimension, then the corresponding bounded derived category $\mathcal{D}^b(B)$
of $B$-mod has a Serre functor given by the left derived of the Nakayama 
endofunctor $B^*\otimes_B{}_-$ of $B$-mod, see \cite{Ha}. In particular, 
in the setup of this paper, both
$\mathcal{D}^b(A)$ and $\mathcal{D}^b(\mathcal{O}_0)$ have a Serre functor.

In \cite{MS2} is was shown that the Serre functor $\mathbb{S}$ on 
$\mathcal{D}^b(\mathcal{O}_0)$ can be alternatively described both as
$(\mathcal{L}\top_{w_0})^2$ and as $(\mathcal{L}C_{w_0})^2$.

\subsection{Auslander regularity}\label{s2.8}

Recall that a finite dimensional associative algebra $B$ of finite global dimension 
is called {\em Auslander regular} provided that the regular $B$-module ${}_BB$
admits a minimal injective coresolution
\begin{displaymath}
0\to{}_BB\to Q_0\to Q_1\to\dots\to Q_m\to 0 
\end{displaymath}
such that the projective dimension of each component $Q_i$ of this 
resolution is at most $i$, where $i=0,1,2,\dots,m$. In \cite{KMM} 
it was shown that the algebra $A$ is Auslander regular.

More detailed information about the injective dimension of the indecomposable projective modules in 
$\mathcal{O}_0$ was obtained in \cite{Maz1,Maz2}. For $w\in W$, the injective dimension
of $P_w$ equals $2\mathbf{a}(w_0w)$.

As explained in \cite[Subsection~4.3]{MT}, the notion of Auslander regularity is closely
connected to certain homological properties of the Serre functor. Indeed, the fact that 
some $I_x$ appears as a summand of some $Q_i$ in the above resolution is equivalent to
$\mathrm{Ext}^i(L_x,B)\neq  0$. Applying the Serre functor $\mathbb{S}$, we get
$\mathrm{Ext}^i(\mathbb{S}L_x,B^*)\neq  0$  which, in turn, is equivalent to
the evaluation of $\mathcal{L}_i\mathbb{S}$ at $L_x$ being non-zero. Therefore,
Auslander regularity of $B$ can be reformulated via a certain vanishing property
for the cohomology of the Serre functor on simple objects.

\subsection{Koszul and Ringel self-dualities}\label{s2.9}

The category $\mathcal{O}_0$ is Ringel self-dual, in particular,
the functor $\top_{w_0}$ that maps $P_w$ to $T_{w_0w}$, for $w\in W$,
induces an equivalence between the categories of projective and 
tilting objects in $\mathcal{O}_0$. We refer to \cite{So4} for details.

The category $\mathcal{O}_0$ is also Koszul self-dual, see \cite{So,BGS}.
This means the following, see \cite{Maz3}: Denote by $\mathcal{LP}(\mathcal{O}_0)$
the category of linear complexes of projective objects in $\mathcal{O}_0$,
that is, those complexes $\mathcal{X}_\bullet$ of projective objects,
for which each $\mathcal{X}_i$ is generated in degree $-i$, for $i\in\mathbb{Z}$.
Then this category is equivalent to the category of graded $A$-modules.
This equivalence maps $\mathcal{P}_\bullet(L_w)$ to $I_{w^{-1}w_0}$, for $w\in W$.

Similarly, denote by $\mathcal{LI}(\mathcal{O}_0)$
the category of linear complexes of injective objects in $\mathcal{O}_0$,
that is, those complexes $\mathcal{Y}_\bullet$ of injective objects,
for which the socle of each $\mathcal{Y}_i$ is concentrated 
in degree $-i$, for $i\in\mathbb{Z}$.
Then this category is equivalent to the category of graded $A$-modules.
This equivalence maps $\mathcal{I}_\bullet(L_w)$ to $P_{w^{-1}w_0}$, for $w\in W$.

The Koszul self-duality of $\mathcal{O}_0$ swaps 
the derived twisting and the derived shuffling functors, see  \cite[Subsection~6.5]{MOS}.

Combining the Ringel and Koszul self-dualities, we get the following:
Denote by $\mathcal{LT}(\mathcal{O}_0)$
the category of linear complexes of tilting objects in $\mathcal{O}_0$,
that is, those complexes $\mathcal{Z}_\bullet$ of tilting objects,
for which the middle of each $\mathcal{Z}_i$ is concentrated 
in degree $-i$, for $i\in\mathbb{Z}$.
Then this category is equivalent to the category of graded $A$-modules.
This equivalence maps $\mathcal{T}_\bullet(L_w)$ to $T_{w_0w^{-1}w_0}$, for $w\in W$.

\section{Calabi-Yau objects in the principal block}\label{s3}

\subsection{Calabi-Yau objects in triangulated categories}\label{s3.1}

Let $\mathcal{C}$  be a  triangulated category with a shift 
functor denoted by $[1]$  and a Serre functor $\mathbb{S}$.
Following \cite{CZ},  an object $M\in \mathcal{C}$ is called a 
{\em Calabi-Yau object of dimension $i$}, where $i\in\mathbb{Z}$, provided that
$\mathbb{S}M\cong M[i]$.

For example,  if $B$ is a finite dimensional associative algebra of 
finite global dimension and $P$ is a projective-injective $B$-module
whose top is isomorphic to the socle, then $P$ is a 
Calabi-Yau object of dimension $0$ in $\mathcal{D}^b(B)$. This is because $\mathbb{S}P$
is isomorphic, by the definition of the Nakayama functor, to the
injective envelope  of the top of $P$. In the setup of the category
$\mathcal{O}_0$,  it follows that $P_{w_0}$ is a 
Calabi-Yau object of dimension $0$ in $\mathcal{D}^b(\mathcal{O}_0)$.

Another fairly straightforward example for the category $\mathcal{O}_0$
is that $L_e$  is a  Calabi-Yau object of dimension $2\ell(w_0)$. Indeed,
the Serre functor $\mathbb{S}$ is the composition of $2\ell(w_0)$ functors
of the form $\mathcal{L}\top_s$, where $s$ is a simple reflection.
And it is well-known that 
$\mathcal{L}\top_sL_e\cong L_e\langle -1\rangle [1]$, see \cite{AS}.

\subsection{Calabi-Yau objects in $\mathcal{D}^b(\mathcal{O}_0)$}\label{s3.2}

Our main result in this section is the following theorem.

\begin{theorem}\label{conj1}
For $d\in\mathbf{D}$ and $w\in W$ such that $w\sim_R d$, the object 
$\theta_w L_d$ is a Calabi-Yau object in $\mathcal{D}^b(\mathcal{O}_0)$
of dimension $2\mathbf{a}(w_0d)$.

More precisely, in the graded picture, we have 
\begin{equation}\label{eq-1}
\mathbb{S}(\theta_w L_{d})\cong
\theta_w L_{d}
\langle 2\big(\mathbf{a}(d)-\mathbf{a}(w_0d)\big)\rangle
[2\mathbf{a}(w_0d)].
\end{equation}
\end{theorem}

Since Soergel  bimodules over the polynomial algebra act on the
usual category $\mathcal{O}$ via restriction, the second claim of 
Theorem~\ref{conj1} remains true if one replaces $\mathbb{S}$ with the
full twist for Soergel  bimodules over the polynomial algebra.

The remainder of this section is devoted to the proof of this
theorem, first in some special cases, and then in general. 
We will also derive some homological consequences from this
result at the end of the section.

Note that, in type $A$, the modules
$\{\theta_w L_d\,:\,d\in\mathbf{D},w\in W,w\sim_R d\}$ described in Theorem~\ref{conj1}
categorify Wedderburn's basis of the complex group algebra of 
the symmetric group, see \cite[Lemma~11]{MS5}.

\subsection{Parabolic subcategories}\label{s3.3}

Let $\mathfrak{p}$ be a parabolic subalgebra of $\mathfrak{g}$ containing
the Borel subalgebra $\mathfrak{b}:=\mathfrak{h}\oplus\mathfrak{n}_+$. Associated to 
$\mathfrak{p}$, we have the corresponding parabolic category $\mathcal{O}^\mathfrak{p}$
defined as the full subcategory of $\mathcal{O}$ consisting of all objects,
the action of $U(\mathfrak{p})$ on which is locally finite, see \cite{RC}.
The category $\mathcal{O}^\mathfrak{p}$ is a Serre subcategory of $\mathcal{O}$
and it inherits from $\mathcal{O}$ a block decomposition. We denote by
$A^\mathfrak{p}$ the quotient of $A$ corresponding to $\mathcal{O}^\mathfrak{p}_0$.
Then $A^\mathfrak{p}$ is a graded quotient of $A$ and, moreover, it turns
out to be Koszul, see \cite{BGS}. In particular, we have the graded 
version ${}^\mathbb{Z}\mathcal{O}^\mathfrak{p}_0$ of $\mathcal{O}^\mathfrak{p}_0$
defined as $A^\mathfrak{p}$-fgmod.

Let $W^\mathfrak{p}$ be the parabolic subgroup of $W$ corresponding to
$\mathfrak{p}$, that is, the subgroup generated by all simple reflections for which 
the corresponding negative root subspace belongs to $\mathfrak{p}$.
Consider the set $(W^\mathfrak{p}\setminus W)_{\mathrm{short}}$
of all shortest coset representatives in $W^\mathfrak{p}\setminus W$.
Then the block $\mathcal{O}_0^\mathfrak{p}$ is the Serre subcategory
of $\mathcal{O}_0$ generated by all simples $L_w$, where
$w\in (W^\mathfrak{p}\setminus W)_{\mathrm{short}}$.

The exact embedding $\iota_{\mathfrak{p}}:\mathcal{O}_0^\mathfrak{p}\subset
\mathcal{O}_0$ has both left and right adjoints. The left adjoint
$Z^{\mathfrak{p}}:\mathcal{O}\to \mathcal{O}_0^\mathfrak{p}$,
also called {\em Zuckerman functor}, see \cite{MS3}, takes a module from
$\mathcal{O}$ to its maximal quotient that belongs to 
$\mathcal{O}_0^\mathfrak{p}$. The right adjoint
$Z_{\mathfrak{p}}:\mathcal{O}\to \mathcal{O}_0^\mathfrak{p}$,
also called {\em dual Zuckerman functor}, takes a module from
$\mathcal{O}$ to its maximal submodule that belongs to 
$\mathcal{O}_0^\mathfrak{p}$. Note that the duality $\star$ restricts to
$\mathcal{O}_0^\mathfrak{p}$ and, moreover, we have $Z^{\mathfrak{p}}\cong
\star\circ Z_{\mathfrak{p}}\circ \star$. The action of $\mathscr{P}$
restricts to $\mathcal{O}_0^\mathfrak{p}$, in particular, both 
$Z^{\mathfrak{p}}$ and $Z_{\mathfrak{p}}$ functorially commute with projective functors.

For $w\in (W^\mathfrak{p}\setminus W)_{\mathrm{short}}$, we have the 
corresponding 
\begin{itemize}
\item  {\em parabolic Verma module}
$\Delta^\mathfrak{p}_w:=Z^{\mathfrak{p}}(\Delta_w)$;
\item  indecomposable projective cover
$P^\mathfrak{p}_w:=Z^{\mathfrak{p}}(P_w)$
of $L_w$ in $\mathcal{O}_0^\mathfrak{p}$;
\item  indecomposable injective envelope
$I^\mathfrak{p}_w:=Z_{\mathfrak{p}}(I_w)$
of $L_w$ in $\mathcal{O}_0^\mathfrak{p}$;
\item  {\em parabolic dual Verma module}
$\nabla^\mathfrak{p}_w:=(\Delta^\mathfrak{p}_w)^\star=Z_{\mathfrak{p}}(\nabla_w)$.
\end{itemize}
We will use the usual conventions for the graded version of these
modules.

Let $w_0^\mathfrak{p}$ denote the longest element in $W^\mathfrak{p}$.
Consider the element $w_0^\mathfrak{p}w_0$ and let $\mathcal{R}^\mathfrak{p}$
be the KL-right cell of this element. Let $d^\mathfrak{p}$ be the Duflo element
in $\mathcal{R}^\mathfrak{p}$. The module $\Delta^{\mathfrak{p}}_e$ has simple
socle which is isomorphic to $L_{d^\mathfrak{p}}$
(in the graded version, to $L_{d^\mathfrak{p}}\langle -\mathbf{a}(d^\mathfrak{p})\rangle$).
Moreover, all other composition subquotients of $\Delta^{\mathfrak{p}}_e$
are of the form $L_x$, where $x<_J d^\mathfrak{p}$ and hence
$\theta_w L_x=0$, for every $w\in \mathcal{R}^\mathfrak{p}$, see \cite[Lemma~12]{MM1}.
For $w\in \mathcal{R}^\mathfrak{p}$, the  projective module 
$P^\mathfrak{p}_w\cong\theta_w \Delta^{\mathfrak{p}}_e\cong \theta_w L_{d^\mathfrak{p}}$
is isomorphic to $I^\mathfrak{p}_w$ and any indecomposable projective-injective module in 
$\mathcal{O}_0^\mathfrak{p}$ is of such form, see \cite{Ir}.

\subsection{Special cases of Theorem~\ref{conj1} related 
to parabolic subcategories}\label{s3.4}

\begin{proposition}\label{prop2}
Let $\mathfrak{p}$ be a parabolic subalgebra of $\mathfrak{g}$ containing
$\mathfrak{b}$.
Then, for any $w\in \mathcal{R}^\mathfrak{p}$, the module
$\theta_w L_{d^\mathfrak{p}}$ is a Calabi-Yau object in $\mathcal{D}^b(\mathcal{O}_0)$
of dimension $2\mathbf{a}(w_0d^\mathfrak{p})$.
\end{proposition}

\begin{proof}
The proof is based on an extension of the computation in the proof of 
\cite[Proposition~4.4]{MS2}. The parabolic Verma module
$\Delta_e^\mathfrak{p}$  has a resolution $\mathcal{X}_\bullet$ by Verma modules
which is obtained by parabolically inducing the BGG resolution, see \cite{BGG2}, of the
trivial module over the semi-simple subalgebra of the Levi factor of
$\mathfrak{p}$. In more detail, for $i\in\mathbb{Z}$, we have 
\begin{displaymath}
\mathcal{X}_i\cong
\begin{cases}
\displaystyle
\bigoplus_{w\in W^\mathfrak{p},\,\ell(w)=|i|}\Delta_w,
& \text{if }\,\,0\leq -i\leq \ell(w_0^\mathfrak{p});\\
0,& \text{else}.
\end{cases}
\end{displaymath}
Applying $\mathcal{L}\top_{w_0}$, maps each $\Delta_w$ appearing in this 
resolution to $\nabla_{w_0w}$ which results in a coresolution of
$L_{w_0w_0^\mathfrak{p}}[\ell(w_0^\mathfrak{p})]\cong 
\nabla_{w_0w_0^\mathfrak{p}}^{w_0(\mathfrak{p})}[\ell(w_0^\mathfrak{p})]$
by dual Verma modules. 
Here the fact that we hit the correct 
coresolution follows,  for example, from the uniqueness of the latter,
see, for example, \cite[Theorem~33]{MaMr}. Note that here we change the parabolic
subalgebra $\mathfrak{p}$ to another parabolic subalgebra which we denote
by $w_0(\mathfrak{p})$.
The simple reflections defining $w_0(\mathfrak{p})$
are exactly the $w_0$-conjugates of the simple reflections that define
$\mathfrak{p}$. In particular, we have
\begin{displaymath}
w_0w_0^\mathfrak{p}=(w_0w_0^\mathfrak{p}w_0)w_0=w_0^{w_0(\mathfrak{p})}w_0.
\end{displaymath}
Note that here $w_0^{w_0(\mathfrak{p})}$ is the longest element
of the parabolic subgroup $W^{w_0(\mathfrak{p})}$ of $W$.
For example, in type $A$, the action of $w_0$ corresponds
to the unique non-trivial automorphism of the Dynkin diagram of the root system.

Now we observe that $L_{w_0w_0^\mathfrak{p}}[\ell(w_0^\mathfrak{p})]\cong 
\nabla_{w_0w_0^\mathfrak{p}}^{w_0(\mathfrak{p})}[\ell(w_0^\mathfrak{p})]\cong 
\Delta_{w_0w_0^\mathfrak{p}}^{w_0(\mathfrak{p})}[\ell(w_0^\mathfrak{p})]$ is exactly
the simple parabolic Verma module in $\mathcal{O}_0^{w_0(\mathfrak{p})}$.
Therefore, similarly to $\Delta_e^\mathfrak{p}$, 
the module $\Delta_{w_0w_0^\mathfrak{p}}^{w_0(\mathfrak{p})}[\ell(w_0^\mathfrak{p})]$ 
has a resolution,
which we denote by $\mathcal{Y}_\bullet$, by Verma modules. 
In more detail, for $i\in\mathbb{Z}$, we have 
\begin{displaymath}
\mathcal{Y}_i\cong
\begin{cases}
\displaystyle
\bigoplus_{w\in W^{w_0(\mathfrak{p})},\,\ell(w)=|i|-\ell(w_0^\mathfrak{p})}\Delta_{ww_0w_0^\mathfrak{p}},
& \text{if }\,\,\ell(w_0^\mathfrak{p})\leq -i\leq 2\ell(w_0^\mathfrak{p});\\
0,& \text{else}.
\end{cases}
\end{displaymath}
Applying $\mathcal{L}\top_{w_0}$, maps each $\Delta_w$ appearing in this 
resolution to $\nabla_{w_0w}$ which results in a coresolution of
$\nabla_{e}^{\mathfrak{p}}[2\ell(w_0^\mathfrak{p})]$
by dual Verma modules. Note that here we come back to the parabolic
subalgebra $\mathfrak{p}$ since $w_0^2=e$.

The above shows that $\mathbb{S}(\Delta_e^\mathfrak{p})\cong 
\nabla_{e}^{\mathfrak{p}}[2\ell(w_0^\mathfrak{p})]$. Recall that $\mathbb{S}$
is a composition of derived twisting functors and the latter commute with 
projective functors. Hence, $\mathbb{S}$ commutes with projective functors.
For $w\in \mathcal{R}^\mathfrak{p}$, we note that
\begin{displaymath}
\theta_w(\Delta_e^\mathfrak{p})\cong
\theta_w(\nabla_e^\mathfrak{p})\cong
\theta_w L_{d^\mathfrak{p}}\cong
I_w^\mathfrak{p}\cong P_w^\mathfrak{p}.
\end{displaymath}
Therefore, applying such $\theta_w$ to $\mathbb{S}(\Delta_e^\mathfrak{p})\cong 
\nabla_{e}^{\mathfrak{p}}[2\ell(w_0^\mathfrak{p})]$, we obtain
$\mathbb{S}(P_w^\mathfrak{p})\cong P_w^\mathfrak{p}[2\ell(w_0^\mathfrak{p})]$.

Finally, note that $\ell(w_0^\mathfrak{p})=\mathbf{a}(w_0^\mathfrak{p})$ as $w_0^\mathfrak{p}$
is the longest element of a parabolic subalgebra. Also, we have
$\mathbf{a}(w_0^\mathfrak{p})=\mathbf{a}(w_0w_0^\mathfrak{p}w_0)$ as conjugation 
by $w_0$ is an automorphism of the root system. We can write 
$w_0w_0^\mathfrak{p}w_0=w_0(w_0^\mathfrak{p}w_0)$ and note that 
$w_0^\mathfrak{p}w_0\sim_R d^\mathfrak{p}$ implies that
$w_0(w_0^\mathfrak{p}w_0)\sim_R w_0d^\mathfrak{p}$, see \cite[Proposition~6.2.9]{BjBr}.
Hence $\mathbf{a}(w_0^\mathfrak{p})=\mathbf{a}(w_0d^\mathfrak{p})$
and the claim follows.
\end{proof}

\begin{proposition}\label{prop3}
Let $\mathfrak{p}$ be a parabolic subalgebra of $\mathfrak{g}$ containing
$\mathfrak{b}$.
In the graded picture, we have
\begin{displaymath}
\mathbb{S}(\theta_w L_{d^\mathfrak{p}})\cong
\theta_w L_{d^\mathfrak{p}}
\langle 2\big(\mathbf{a}(d^\mathfrak{p})-\mathbf{a}(w_0d^\mathfrak{p})\big)\rangle
[2\mathbf{a}(w_0d^\mathfrak{p})].
\end{displaymath}
\end{proposition}

\begin{proof}
Recall that twisting functors are acyclic on Verma modules, see
\cite[Theorem~2.2]{AS}. Therefore application of derived twisting functors to 
complexes whose components have Verma filtration reduces to 
application of the ordinary twisting functors.
We will use this frequently in the following proof.

In $\Delta_e^\mathfrak{p}$, we have the subquotient 
$L_{d^\mathfrak{p}}\langle -\mathbf{a}(d^\mathfrak{p})\rangle$
of interest and the top of each 
$\mathcal{X}_i$ lives in degree $-i$. The application
of $\top_{w_0}$ to $\mathcal{X}_i$ results in a module that has
its socle in degree $-i$. In particular, we have that 
$\top_{w_0}\mathcal{X}_{-\ell(w_0^\mathfrak{p})}$, which is isomorphic to
$\nabla_{w_0w_0^\mathfrak{p}}$ as an ungraded module,
has its socle $L_{w_0w_0^\mathfrak{p}}$ in degree 
$\ell(w_0^\mathfrak{p})$, that is, this subquotient is
$L_{w_0w_0^\mathfrak{p}}\langle -\ell(w_0^\mathfrak{p})\rangle$.

If we now take $\mathcal{Y}_\bullet$ as the graded Verma resolution of
$L_{w_0w_0^\mathfrak{p}}\langle -\ell(w_0^\mathfrak{p})\rangle
[\ell(w_0^\mathfrak{p})]$, then, again, the top of each $\mathcal{Y}_i$ 
lives in degree $-i$. Again, the application of $\top_{w_0}$ to 
$\mathcal{Y}_i$ results in a module that has its socle in degree $-i$.
So, now we have $\top_{w_0}\mathcal{Y}_{-2\ell(w_0^\mathfrak{p})}$, 
which is isomorphic to $\nabla_{e}$ as an ungraded module,
has its socle $L_{e}$ in degree  $2\ell(w_0^\mathfrak{p})$,
that is, comes with a $-2\ell(w_0^\mathfrak{p})$ shift.
Taking this shift into account, the subquotient 
$L_{d^\mathfrak{p}}$
of interest now lives in degree $2\ell(w_0^\mathfrak{p})-\mathbf{a}(d^\mathfrak{p})$,
in other words, this subquotient is 
$L_{d^\mathfrak{p}}\langle \mathbf{a}(d^\mathfrak{p})-2\ell(w_0^\mathfrak{p})\rangle$.
Adding all this up, we get the total shift of 
$ 2\big(\mathbf{a}(d^\mathfrak{p})-\ell(w_0^\mathfrak{p})\big)$
for our $L_{d^\mathfrak{p}}$ subquotient.

During the proof of Proposition~\ref{prop2} we established that 
$\ell(w_0^\mathfrak{p})=\mathbf{a}(w_0d^\mathfrak{p})$.
Hence the claim of our proposition follows.
%
\end{proof}

\subsection{Serre functor on simple modules}\label{s4.2}

We make the following general observation on
the value of the Serre functor $\mathbb{S}$ on simple modules.

\begin{lemma}\label{lem6}
Let $x,w\in W$ and $i\in\mathbb{Z}$ be such that 
$\mathcal{D}^b(\mathcal{O}_0)(P_x[i],\mathbb{S}(L_w))\neq 0$.
Then $x\sim_H w$ or $x<_J w$.
\end{lemma}

\begin{proof}
From \cite[Section~6]{AS} it follows that, applying 
$\mathcal{L}\top_s$, where $s\in S$, to a simple module
$L_a$, for some $a\in W$, results in a complex whose homology 
only contains simple subquotients of the form $L_b$, where $b\leq _L a$.
Consequently, any finite composition of functors of 
the form $\mathcal{L}\top_s$, for various $s\in S$,
has a similar property: applying such composition  to $L_a$
results in a complex whose homology 
only contains simple subquotients of the form $L_b$, where $b\leq _L a$.
As we know that $\mathbb{S}$ is a composition of  functors
of the form $\mathcal{L}\top_s$, where $s\in S$, the property
$x\leq_L w$ follows.

On the other hand, we also know that 
$\mathbb{S}$ is a composition of the functors
of the form $\mathcal{L} C_s$, where $s\in S$. As, for any 
projective functor $\theta$, any simple 
subquotient of $\theta L_w$ is of the form $L_y$, for some
$y\leq_R w$, see \cite[Lemma~13]{MM1}, it follows that 
$x\leq_R w$. 

Combining $x\leq_L w$ and $x\leq_R w$, we get the claim of the lemma.
\end{proof}

\subsection{Kazhdan-Lusztig  combinatorics of type $A$}\label{s4.3}

Kazhdan-Lusztig combinatorics in type $A$ is especially nice, see \cite{KL,Ge}.
Let $\mathfrak{g}=\mathfrak{sl}_n$, then $W=S_n$. The Robinson-Schensted
correspondence, see \cite{Sa}, gives rise to a bijection between the elements of $S_n$ and pairs
of standard Young tableaux of the same shape, where the shape is a partition of $n$:
\begin{displaymath}
\mathbf{RS}:S_n\to\coprod_{\lambda\vdash n}
\mathbf{SYT}_\lambda\times  \mathbf{SYT}_\lambda,
\end{displaymath}
where $\mathbf{SYT}_\lambda$ denotes the set of all standard Young
tableaux of shape $\lambda$.  Let us
write $\mathbf{RS}(\pi)=(\mathbf{p}_\pi,\mathbf{q}_\pi)$. Then $\pi\sim_L\sigma$
if and only if $\mathbf{q}_\pi=\mathbf{q}_\sigma$ and $\pi\sim_R\sigma$
if and only if $\mathbf{p}_\pi=\mathbf{p}_\sigma$. Consequently, $\sim_H$
is the equality relation. Furthermore, $\pi\sim_J\sigma$ if and only if
the shapes of $\mathbf{p}_\pi$ and $\mathbf{p}_\sigma$ coincide. This means that
the two-sided cells of $S_n$ are in bijection with partitions of $n$.
Additionally, the two-sided order $\leq_J$ is given by the dominance order 
on partitions.

As a consequence of the above description, any two-sided KL-cell in type $A$
contains both the longest element $w_0^\mathfrak{p}$
of some parabolic subalgebra $\mathfrak{p}$;
an element of the form $w_0w_0^\mathfrak{q}$, for some 
some parabolic subalgebra $\mathfrak{q}$; and an element of the form 
$w_0^\mathfrak{r}w_0$, for some parabolic subalgebra $\mathfrak{r}$.

Another consequence is that the Duflo elements in type $A$ are exactly the
involutions.

Let $\mathcal{R}$ be a right KL-cell and $d$ a Duflo element in 
$\mathcal{R}$. Then the additive closure $\mathbf{C}_\mathcal{R}$
of all modules in $\mathcal{O}$ of the form $\theta_w L_d$, where $w\in \mathcal{R}$,
is stable under the action of $\mathscr{P}$. It is called the 
{\em cell birepresentation} of $\mathscr{P}$ corresponding to $\mathcal{R}$, 
see \cite{MM1} for detail. A very special feature of type $A$ is the 
following: given two right KL-cells $\mathcal{R}$ and $\mathcal{R}'$
inside the same two-sided KL-cell, the corresponding cell birepresentations
$\mathbf{C}_\mathcal{R}$ and $\mathbf{C}_{\mathcal{R}'}$ are
biequivalent, see \cite{MS4,MM1}. The biequivalence can be 
constructed fairly explicitly using recursive application of the derived twisting functors
$\mathcal{L}\top_s$, where $s\in S$, which functorially commute with
projective functors, followed by projections onto
the appropriate cell birepresentations, see \cite{MS4} for details.

As a consequence of this biequivalence, we have the following in type $A$:
given a two-sided cell $\mathcal{J}$ and two elements $x,y\in \mathcal{J}$,
the module $\theta_x L_y$ is either zero or indecomposable. Moreover,
in the latter case, we have $\theta_x L_y\cong \theta_w L_d$, where
$d$ is the Duflo element in the right KL-cell $\mathcal{R}$ of $y$
and $w$ is the unique element in the intersection of $\mathcal{R}$
with the left KL-cell of $x$. This gives us substantial flexibility 
allowing to rewrite the module $\theta_w L_d$, which appears in the
formulation of Theorem~\ref{conj1}, in a different 
form $\theta_x L_y$ that will make the proof in the next subsection work.

\subsection{Proof of Theorem~\ref{conj1} in type $A$}\label{s4.4}

In this subsection, we will prove Theorem~\ref{conj1} in type $A$.

Let $d$ be an involution in $S_n$ (i.e. a Duflo element in $S_n$). 
Consider its left KL-cell $\mathcal{L}$,
its right  KL-cell $\mathcal{R}$ and its two-sided KL-cell $\mathcal{J}$.
We want to prove Formula~\eqref{eq-1}, for any $w\in \mathcal{R}$.

There is a parabolic subalgebra $\mathfrak{p}$ in $\mathfrak{g}$
such that $w_0^\mathfrak{p}w_0$ belongs to $\mathcal{J}$. Let
$\mathcal{R}'$ be the right KL-cell of this element 
$w_0^\mathfrak{p}w_0$ and let $d'$ be the Duflo element in $\mathcal{R}'$.
From Proposition~\ref{prop3} we know that Formula~\eqref{eq-1} is true
for $d=d'$ and for any $w\in \mathcal{R}'$.

As mentioned in the previous paragraph, the cell birepresentations
$\mathbf{C}_\mathcal{R}$ and $\mathbf{C}_{\mathcal{R}'}$ are biequivalent.
Let $\Phi:\mathbf{C}_\mathcal{R}\to \mathbf{C}_{\mathcal{R}'}$ be such 
a biequivalence and $\Phi^{-1}$ be its inverse. Then both $\Phi$ and
$\Phi^{-1}$ functorially commute with all projective functors. 

Now recall that our Serre functor $\mathbb{S}$ is a composition of the
derived shuffling functors $\mathcal{L}C_s$, where $s\in S$. The functor
$\mathcal{L}C_s$ is represented by the complex $0\to \theta_e\to \theta_s\to 0$,
where the map $\theta_e\to \theta_s$ is the adjunction morphism.
In other words, this is a complex consisting of projective functors 
and morphisms between them. Consequently, both $\Phi$ and $\Phi^{-1}$ 
commute with the functor given by this complex. This means that
both $\Phi$ and $\Phi^{-1}$ commute with $\mathbb{S}$.
Here we note that $\Phi$ and $\Phi^{-1}$ are only equivalences 
between certain subquotients of $\mathcal{O}_0$, but not of the whole
$\mathcal{O}_0$, so we cannot directly use that Serre functors commute
with any auto-equivalence of $\mathcal{O}_0$.

Since $\Phi(\theta_w L_d)\cong \theta_x L_{d'}$, for some
$x\in \mathcal{R}'$ and $\Phi^{-1}$ commutes with $\mathbb{S}$, 
using Proposition~\ref{prop3} we have:
\begin{displaymath}
\begin{array}{rcl}
\mathbb{S}(\theta_w L_d)&\cong &\mathbb{S}(\Phi^{-1}(\theta_x L_{d'}))\\
&\cong &\Phi^{-1}(\mathbb{S}(\theta_x L_{d'}))\\
&\cong &\Phi^{-1}(\theta_x L_{d'}
\langle 2\big(\mathbf{a}(d')-\mathbf{a}(w_0d')\big)\rangle
[2\mathbf{a}(w_0d')] )\\
&\cong &\theta_w L_d
\langle 2\big(\mathbf{a}(d')-\mathbf{a}(w_0d')\big)\rangle
[2\mathbf{a}(w_0d')].
\end{array}
\end{displaymath}
Note that $\mathbf{a}(d)=\mathbf{a}(d')$ and
$\mathbf{a}(w_0d)=\mathbf{a}(w_0d')$ since $d\sim_J d'$.
This completes the proof of Theorem~\ref{conj1} in type $A$.

\subsection{Subcategories $\mathcal{O}^{\hat{\mathcal{R}}}$}\label{s5.2}

For a  right KL-cell $\mathcal{R}$ in $W$, denote by $\hat{\mathcal{R}}$
the set of all elements $w\in W$ such that $w\leq_R \mathcal{R}$.
Denote by $\mathcal{O}^{\hat{\mathcal{R}}}$  the Serre subcategory 
of $\mathcal{O}_0$ generated by all simples $L_w$, where $w\in \hat{\mathcal{R}}$.
The category $\mathcal{O}^{\hat{\mathcal{R}}}$ appears in \cite{MS3,Maz2}.
This category is stable under the action of $\mathscr{P}$.
The inclusion $\mathcal{O}^{\hat{\mathcal{R}}}\subset 
\mathcal{O}_0$ is exact and hence has two adjoints:
the left adjoint $Z^{\hat{\mathcal{R}}}$ given by taking the maximal
quotient in $\mathcal{O}^{\hat{\mathcal{R}}}$ and the right 
adjoint $Z_{\hat{\mathcal{R}}}$ given by taking the maximal
submodule in $\mathcal{O}^{\hat{\mathcal{R}}}$. 

As the natural inclusion $\mathcal{O}^{\hat{\mathcal{R}}}\subset 
\mathcal{O}_0$ is a morphism of birepresentations of $\mathscr{P}$,
both $Z^{\hat{\mathcal{R}}}$ and $Z_{\hat{\mathcal{R}}}$ are also
morphisms of birepresentations of $\mathscr{P}$, that is, they both  
functorially commute with projective functors.

\subsection{Further special cases of  Theorem~\ref{conj1}}\label{s5.3}

The proof in Subsection~\ref{s4.4} can be (substantially) refined to
establish the following.

\begin{proposition}\label{thm7}
Assume that $\mathfrak{g}$ is any semi-simple finite
dimensional complex  Lie algebra and $d\in\mathbf{D}$ is such that
there exists a parabolic subalgebra $\mathfrak{p}$ of $\mathfrak{g}$ with the property
$d\sim_J w_0^\mathfrak{p}w_0$. Then, for  any $w\in W$ such that $w\sim_R d$, the object 
$\theta_w L_d$ is a Calabi-Yau object in $\mathcal{D}^b(\mathcal{O}_0)$
of dimension $2\mathbf{a}(w_0d)$. Moreover, in the graded picture, we have 
\begin{displaymath}
\mathbb{S}(\theta_w L_{d})\cong
\theta_w L_{d}
\langle 2\big(\mathbf{a}(d)-\mathbf{a}(w_0d)\big)\rangle
[2\mathbf{a}(w_0d)].
\end{displaymath}
\end{proposition}

\subsection{Proof of Proposition~\ref{thm7}}\label{s5.3-5}

Let $\mathcal{R}$ be the right KL-cell of $d$
and $\mathcal{J}$ the two-sided KL-cell  of $d$.
Let $d'$ be a Duflo element such that $d\sim_J d'$ 
and $d'\sim_R w_0^\mathfrak{p}w_0$, as assumed in the formulation. 
Let $\mathcal{R}'$ be the right KL-cell of $d'$. Then, from  
Proposition~\ref{prop3} we know that Formula~\eqref{eq-1} is true
for $d=d'$ and for any $w\in \mathcal{R}'$.
Assume now that $d\neq d'$.

Since we are outside type $A$, the birepresentations 
$\mathbf{C}_{\mathcal{R}}$ and $\mathbf{C}_{\mathcal{R}'}$
are not biequivalent, in general. However, we still do have 
homomorphisms between these birepresentations, that is 
functors between the underlying additive categories that
commute with the action of projective functors. The problem
is that these homomorphism are no longer equivalences, for
example, they might send indecomposable objects to 
decomposable objects etc.

Note that both birepresentations are {\em transitive} in the
sense that, for any non-zero object  $X$ of this birepresentation,
the additive closure of $\mathscr{P}X$ coincides with the whole
birepresentation. Consequently, any non-zero homomorphism from such
a birepresentation must be non-zero on any non-zero object.
Also, the additive closure of the image of any non-zero 
homomorphism to a transitive birepresentation coincides with the
whole birepresentation. 

\begin{lemma}\label{lem8}
In the situation described above,
there exist a non-zero  homomorphisms $\Phi$
of birepresentations of $\mathscr{P}$ from 
$\mathbf{C}_{\mathcal{R}'}$ and $\mathbf{C}_{\mathcal{R}}$.
\end{lemma}

\begin{proof}
Such homomorphism can be constructed using twisting functors
similar to the corresponding equivalences in type $A$ that are constructed
in the proof of \cite[Theorem~5.4(ii)]{MS2}. Let us now describe the
details of how this works. 

Let $\mathcal{R}_1$ and $\mathcal{R}_2$ be two different 
right KL-cells inside the same two-sided KL-cell $\mathcal{J}$.
For a simple reflection $s$, consider the corresponding functor
$\top_s$. This functor commutes with the action of $\mathscr{P}$
and hence the composition $Z_{\hat{\mathcal{R}_2}}\circ \top_s$ 
commutes with the action of $\mathscr{P}$ as well.

Now, choose $s$ such that $sw<w$, for $w\in \mathcal{R}_1$
(we note that all elements in $\mathcal{R}_1$ have the same left descent).
Then, for $w\in \mathcal{R}_1$, the module $\top_s L_w$ has top
$L_w$ and a semi-simple radical described in \cite[Sections~6 and 7]{AS}.
As $\mathcal{R}_1\neq \mathcal{R}_2$, we have
$L_w\not\in \mathcal{O}^{\hat{\mathcal{R}_2}}$
since any two right KL-cells inside the
same two-sided KL-cell are not comparable with respect to the right KL-order.
This implies that $Z_{\hat{\mathcal{R}_2}}\circ \top_s(L_w)=
Z_{\hat{\mathcal{R}_2}}(\mathrm{Rad}(\top_s(L_w)))$.

Assume that $L_x$ is a summand of $\mathrm{Rad}(\top_s(L_w))$.
If $x\not\in \mathcal{J}$, then $\theta_y L_x=0$, for any
$y\in \mathcal{J}$. This implies that, for $y\in \mathcal{J}$,
the module $Z_{\hat{\mathcal{R}_2}}\circ \top_s(\theta_y L_w)$
is a direct sum of modules of the form $\theta_y L_x$, where
$x\in \mathcal{R}_2$ is such that $L_x$ appears in $\mathrm{Rad}(\top_s(L_w))$.
In other words, the functor $Z_{\hat{\mathcal{R}_2}}\circ \top_s$
maps objects of $\mathbf{C}_{\mathcal{R}_1}$ to objects of
$\mathbf{C}_{\mathcal{R}_2}$.

It remains to recall that the left KL order is generated by the 
relation connecting $w$ to all $x$ such that $L_x$ appears in the
radical of $\top_s L_w$, for some $s\in S$, cf. \cite{KL}
and \cite[Sections~7]{AS}. Therefore, starting from 
$\mathcal{R}_1$ we can find some $s$ and some
$\mathcal{R}_2\neq \mathcal{R}_1$ such that the above gives a
non-zero functor from $\mathbf{C}_{\mathcal{R}_1}$ to 
$\mathbf{C}_{\mathcal{R}_2}$. Then, repeating this finitely many times,
we eventually construct a non-zero functor from any 
$\mathbf{C}_{\mathcal{R}_1}$ to any other $\mathbf{C}_{\mathcal{R}_2}$,
for $\mathcal{R}_1$ and $\mathcal{R}_2$ in $\mathcal{J}$.
\end{proof}

Let $\Phi:\mathbf{C}_{\mathcal{R}'}\to\mathbf{C}_{\mathcal{R}}$ 
be a functor given by Lemma~\ref{lem6}. If we apply it to 
the direct sum $X$ of all $\theta_u L_{d'}$, where $u\in \mathcal{R}'$,
we will get a direct sum of all $\theta_w L_{d}$, where
$w\in \mathcal{R}$, each appearing with a non-zero multiplicity.
Applying $\Phi$ (which commutes with $\mathbb{S}$) to
\begin{displaymath}
\mathbb{S}(X)\cong 
X\langle 2\big(\mathbf{a}(d)-\mathbf{a}(w_0d)\big)\rangle
[2\mathbf{a}(w_0d)],
\end{displaymath}
and using the additivity of $\mathbb{S}$, we obtain that 
$\mathbb{S}(\theta_w L_{d})$ is a direct sum the objects
of the form
\begin{displaymath}
\theta_v L_d\langle 2\big(\mathbf{a}(d)-\mathbf{a}(w_0d)\big)\rangle
[2\mathbf{a}(w_0d)],
\end{displaymath}
where $v\in\mathcal{R}$, with some multiplicities.

To complete the proof, we only need to argue that 
all of these have multiplicity $0$ apart from
$v=w$, which has multiplicity $1$. Recall that 
$\mathbb{S}$ is an auto-equivalence. This implies
that $\mathbb{S}(\theta_w L_{d})$ is indecomposable,
since $\theta_w L_{d}$ is indecomposable. In other
words, 
\begin{displaymath}
\mathbb{S}(\theta_w L_{d})\cong
\theta_v L_d\langle 2\big(\mathbf{a}(d)-\mathbf{a}(w_0d)\big)\rangle
[2\mathbf{a}(w_0d)],
\end{displaymath}
for some $v\in \mathcal{R}$, and it remains to 
show that $v=w$.

Next we observe that
the Serre functor $\mathbb{S}$ defines the identity map 
at the level of the ungraded Grothendieck group of
$\mathcal{D}^b(\mathcal{O}_0)$. This is because
$\mathbb{S}$ maps each $P_w$ to $I_w$ and these two
modules have the same character as they are connected
by the simple preserving duality $\star$. Moreover,
the representatives of the modules $P_w$, for $w\in W$, form
a basis of the ungraded Grothendieck group as $\mathcal{O}_0$
has finite global dimension.

Therefore the representatives of the objects
$\mathbb{S}(\theta_w L_d)$ and $\theta_w L_d$
in the ungraded Gro\-then\-dieck group of $\mathcal{D}^b(\mathcal{O}_0)$
must coincide. Next we note that 
$\theta_d L_d$ is the only one of the modules $\theta_w L_d$,
where $w\in \mathcal{R}$, which has
$L_e$ as a simple subquotient. Consequently,
$\mathbb{S}(\theta_d L_d)$ is isomorphic to
$\theta_d L_d\langle 2\big(\mathbf{a}(d)-\mathbf{a}(w_0d)\big)\rangle
[2\mathbf{a}(w_0d)]$.

\begin{lemma}\label{lem9}
Assume that we are under the 
same assumptions as Proposition~\ref{thm7}, that is,
there exists a parabolic subalgebra $\mathfrak{p}$ such that
$d\sim_J w_0^\mathfrak{p}w_0$. Then,
for $w\in W$ such that $w\sim_H d$ and $i,j\in\mathbb{Z}$, we have
\begin{displaymath}
\mathcal{D}^b(\mathcal{O}_0)
\big(P_w\langle i\rangle[j],\mathbb{S}(L_d)\big)=
\begin{cases}
\mathbb{C},&  w=d, i=2\big(\mathbf{a}(d)-\mathbf{a}(w_0d)\big), j=2\mathbf{a}(w_0d);\\
0,&\text{else}. 
\end{cases}
\end{displaymath}
\end{lemma}

\begin{proof}
Assume that some $L_w\langle i\rangle[j]$ appears as a subquotient of
a homology of $\mathbb{S}(L_d)$. By Lemma~\ref{lem6}, we either have
$w\sim_H d$ or $w<_J d$. In the case $w<_J d$, we have that $\theta_d L_w=0$.
In the case $w\sim_H d$, we have
$\theta_d L_w\cong \theta_w L_d$, combining 
\cite[Proposition~4.15]{MMMTZ2} with \cite[Formulae (4.11) and (4.12)]{MMMTZ2}.
Now the claim of the lemma follows from the fact that
$\mathbb{S}(\theta_d L_d)$ is isomorphic to
$\theta_d L_d\langle 2\big(\mathbf{a}(d)-\mathbf{a}(w_0d)\big)\rangle
[2\mathbf{a}(w_0d)]$ which we established above.
\end{proof}

Lemma~\ref{lem9} says  that,  modulo simple modules indexed by 
cells that are strictly smaller than $\mathcal{J}$ with respect to the
two-sided KL-order, the only simple constituent of the homology of the complex 
$\mathbb{S}(L_d)$ is $L_d\langle 2\big(\mathbf{a}(d)-\mathbf{a}(w_0d)\big)\rangle
[2\mathbf{a}(w_0d)]$. Now we can complete the proof of Theorem~\ref{thm7}
as, for $w\in \mathcal{R}$, we have $\theta_w L_x=0$, for  all $x\leq_J w$.

Indeed,  for $w\in \mathcal{R}$,  applying $\theta_w$ to $\mathbb{S}(L_d)$,
we obtain that $\mathbb{S}(\theta_w L_d)\cong \theta_w \mathbb{S}(L_d)$ is isomorphic to
$\theta_w L_d\langle 2\big(\mathbf{a}(d)-\mathbf{a}(w_0d)\big)\rangle
[2\mathbf{a}(w_0d)]$. This proves Proposition~\ref{thm7}.

\subsection{Proof of Theorem~\ref{conj1} in the general case}\label{s5.7}

Our proof of Proposition~\ref{thm7} leads to a strategy to prove
Theorem~\ref{conj1} in the general case. We start by the
following generalization of Lemma~\ref{lem9}.

\begin{lemma}\label{lem11}
Under the assumptions of Theorem~\ref{conj1},
the claim of Lemma~\ref{lem9} is true for any $d\in\mathbf{D}$. 
Namely, for any $d\in\mathbf{D}$ and
$w\in W$ such that $w\sim_H d$ and any $i,j\in\mathbb{Z}$, we have
\begin{displaymath}
\mathcal{D}^b(\mathcal{O}_0)
\big(P_w\langle i\rangle[j],\mathbb{S}(L_d)\big)=
\begin{cases}
\mathbb{C},&  w=d, i=2\big(\mathbf{a}(d)-\mathbf{a}(w_0d)\big), j=2\mathbf{a}(w_0d);\\
0,&\text{else}. 
\end{cases}
\end{displaymath}
\end{lemma}

\begin{proof}
We start by noticing that
\begin{displaymath}
\mathcal{D}^b(\mathcal{O}_0)
\big(P_w\langle i\rangle[j],\mathbb{S}(L_d)\big)\cong
\mathcal{D}^b(\mathcal{O}_0)
\big(\mathbb{S}(L_d),I_w\langle i\rangle[j]\big).
\end{displaymath}
Recall that $\mathbb{S}$ is an equivalence 
and let $\mathbb{S}^{-1}$ be the inverse 
equivalence. Then we have 
\begin{displaymath}
\mathcal{D}^b(\mathcal{O}_0)
\big(\mathbb{S}(L_d),I_w\langle i\rangle[j]\big)\cong
\mathcal{D}^b(\mathcal{O}_0)
\big(L_d,\mathbb{S}^{-1}(I_w\langle i\rangle[j])\big).
\end{displaymath}
Also, recall that $\mathbb{S}$, being the derived of the
Nakayama functor, maps $P_w$ to $I_w$. Hence
\begin{displaymath}
\mathcal{D}^b(\mathcal{O}_0)
\big(\mathbb{S}(L_d),I_w\langle i\rangle[j]\big)\cong
\mathcal{D}^b(\mathcal{O}_0)
\big(L_d\langle -i\rangle[-j],P_w\big).
\end{displaymath}
Now we recall from \cite{Maz2} that the injective dimension of
$P_w$ equals $2\mathbf{a}(w_0w)$. At the same time, $A$ is 
Auslander regular by \cite{KMM}. Therefore, assuming
\begin{equation}\label{eq-5}
\mathcal{D}^b(\mathcal{O}_0)
\big(L_d\langle -i\rangle[-j],P_w\big)\neq 0
\end{equation}
implies $j\geq 2\mathbf{a}(w_0w)$ since $d$ and $w$ belong to the 
same two-sided cell and hence the projective dimension of $I_d$
equals $2\mathbf{a}(w_0w)$ as well. Consequently, 
Formula~\eqref{eq-5} implies $j=2\mathbf{a}(w_0w)$.

Next we look at the complex $\theta_d \mathbb{S}(L_d)\cong
\mathbb{S}(\theta_d L_d)$. The application of $\theta_d$ will
kill all simple constituents of the homology of
$\mathbb{S}(L_d)$ which come from the cell that are two-sided
smaller than the two-sided cell $\mathcal{J}$ containing $d$
(note that $\mathcal{J}$ contains $w$ as well). 
Therefore the previous paragraph implies that
all homology of $\theta_d \mathbb{S}(L_d)$ is concentrated at
the same homological position, namely at the position $-2\mathbf{a}(w_0w)$.
In other words, $\theta_d \mathbb{S}(L_d)$ is just a module shifted
by $2\mathbf{a}(w_0w)$ in the homological position.

As $\theta_d L_d$ is indecomposable and $\mathbb{S}$ is an equivalence, 
it follows that the homology of $\theta_d \mathbb{S}(L_d)$ must be
indecomposable. At the level of the ungraded Grothendieck group,
the application of $\mathbb{S}$ gives the identity map
(since $w_0^2=e$ in $W$). Therefore the ungraded homology of
$\theta_d \mathbb{S}(L_d)$ at position $-2\mathbf{a}(w_0w)$
is isomorphic to $\theta_d L_d$. As the latter is not
isomorphic to any $\theta_d L_u$, where $u\sim_H d$ and $d\neq u$
(for example, since $[\theta_d L_d:L_e]\neq 0$ while $[\theta_d L_u:L_e]=0$),
it follows that Formula~\eqref{eq-5} implies $w=d$.

It remains to show that $i=2\big(\mathbf{a}(d)-\mathbf{a}(w_0d)\big)$.
For this, we recall from \cite[Theorem~16]{Maz2} that the collection
$\{\theta_x L_y\,:\,x,y\in W\}$ of modules is Koszul-Ringel 
self-dual and this Koszul-Ringel self-duality swaps $\theta_x L_y$ with
$\theta_{y^{-1}w_0}L_{w_0x^{-1}}$. This Koszul-Ringel self-duality,
denoted $\mathrm{K}$, provides two
different realizations of ${}^\mathbb{Z}\mathcal{O}_0$ inside
$\mathcal{D}^b({}^\mathbb{Z}\mathcal{O}_0)$ and $\mathbb{S}$ is a
Serre functor on the latter category. The functor $\mathrm{K}$
is a self-equivalence of $\mathcal{D}^b({}^\mathbb{Z}\mathcal{O}_0)$
and hence commutes with $\mathbb{S}$. Moreover, the functor $\mathrm{K}$
commutes with the shifts in the following way: 
$\mathrm{K}\langle a\rangle[b]=\langle -a\rangle[a+b]\mathrm{K}$,
see \cite[Proposition~20]{MOS}.
For $w\in W$ and $d_w$ the Duflo element in the 
right KL-cell of $w$, denote by $i_w\in\mathbb{Z}$ the shift such that 
\begin{displaymath}
\mathbb{S}(\theta_w L_{d_w})\cong  
\theta_w L_{d_w}\langle i_w\rangle [2\mathbf{a}(w_0d_w)].
\end{displaymath}
Applying $\mathrm{K}$, we get
\begin{displaymath}
\mathbb{S}(\theta_{d_ww_0} L_{w_0w^{-1}})\cong  
\theta_{d_ww_0} L_{w_0w^{-1}}\langle -i_w\rangle [i_w+2\mathbf{a}(w_0d_w)].
\end{displaymath}
As we already know from the above  that the homological shift
$i_w+2\mathbf{a}(w_0d_w)$ on the right hand side of the latter
isomorphism should be equal
to $2\mathbf{a}(w_0(w_0w^{-1}))=2\mathbf{a}(w)$, we deduce that
$i_w=2\mathbf{a}(w)- 2\mathbf{a}(w_0d_w)$. 
As we always have $\mathbf{a}(w)=\mathbf{a}(d_w)$, the claim 
that $i=2\big(\mathbf{a}(d)-\mathbf{a}(w_0d)\big)$ follows. 
This completes the proof of our lemma.
\end{proof}

To prove Theorem~\ref{conj1}, we apply $\theta_w$ 
to  $\mathbb{S}(L_d)$. Lemma~\ref{lem11} tells us
that $\mathbb{S}(L_d)$ has a unique simple subquotient 
$L_d$ in homology in the correct homological position with 
the correct grading shift, while all other simple
subquotients of the homology of $\mathbb{S}(L_d)$ are killed by $\theta_w$.
Also, $\mathbb{S}$ functorially commutes with $\theta_w$
and hence $\theta_w\mathbb{S}(L_d)\cong \mathbb{S}(\theta_w L_d)$.
Therefore the claim of Theorem~\ref{conj1} follows.

\subsection{Homological consequences of Theorem~\ref{conj1}}\label{s5.8}

The following corollary is inspired by \cite[Theorem~A]{KMM}.

\begin{corollary}\label{cor-s5.8-1}
Let $x,y\in W$ and $i\in\mathbb{Z}$. Then we have:

\begin{enumerate}[$($a$)$]
\item\label{cor-s5.8-1.1}
$\mathrm{Ext}^i(I_x,L_y)=0$ unless $i\leq 2\mathbf{a}(w_0x)$
and $y\geq_L x$.
\item\label{cor-s5.8-1.2}
$\mathrm{Ext}^i(I_x,L_y)=0$ if $i<2\mathbf{a}(w_0x)$
and $y\sim_L x$.
\item\label{cor-s5.8-1.3}
$\mathrm{Ext}^{2\mathbf{a}(w_0x)}(I_x,L_y)=0$ if 
$y\sim_L x$  and $y\neq x$.
\item\label{cor-s5.8-1.4}
$\mathrm{Ext}^{2\mathbf{a}(w_0x)}(I_x,L_x)=\mathbb{C}$, more precisely,
in the graded picture, we have
\begin{displaymath}
\mathrm{ext}^{2\mathbf{a}(w_0x)}
(I_x,L_x\langle 2(\mathbf{a}(x)-\mathbf{a}(w_0x))\rangle)=\mathbb{C}.
\end{displaymath}
\item\label{cor-s5.8-1.5}
$\mathrm{Ext}^{2\mathbf{a}(w_0x)}(I_x,L_y)=0$ provided that
$y>_L x$.
\end{enumerate}
\end{corollary}

\begin{proof}
In Claim~\eqref{cor-s5.8-1.1}, the restriction $i\leq 2\mathbf{a}(w_0x)$
follows from \cite[Theorem~16]{Maz1} and \cite[Theorem~20]{Maz2}.
To prove the second restriction,  let $\mathcal{X}_\bullet=
\mathcal{P}_\bullet(I_e)$. Then $\theta_x I_e\cong I_x$
and hence $\theta_x\mathcal{X}_\bullet$ is a projective resolution
(not necessarily minimal)  of  $I_x$. For any $w\in W$, all summands
of $\theta_xI_w$ are of the form  $I_z$, where $z\geq_L x$. This means that
only such summands appear in $\theta_x\mathcal{X}_\bullet$ and implies
the second restriction in Claim~\eqref{cor-s5.8-1.1}.
 
Claim~\eqref{cor-s5.8-1.2} follows from \cite[Theorem~A]{KMM}.

To  prove the rest, for $w\in W$, consider the Duflo element $d$ in the
left KL-cell of $w$  and the Duflo element  $d'$ in the right
KL-cell of $w$.  Then $\theta_d L_w\cong \theta_w L_{d'}$  by
\cite[Proposition~4.15]{MMMTZ2}.

Consider the complex  $\mathbb{S}(L_w)$. Let $L_z$, for $z\in W$, be a simple
constituent of some homology of $\mathbb{S}(L_w)$. Then either
$z\sim_H w$ or $z<_J w$ by Lemma~\ref{lem6}. In the case $z<_J w$,
we have $\theta_d L_z=0$ by \cite[Lemma~12]{MM1}. If $z\sim_H w$,
then, by the previous paragraph,  $\theta_d L_z\cong \theta_z L_{d'}$.

Therefore, from Theorem~\ref{conj1} it follows that,  among all  homologies
of $\mathbb{S}(L_w)$, only the homology at the homological  position 
$-2\mathbf{a}(w_0w)$ contains a simple  constituent $L_z$ such that
$z\sim_H w$, moreover, such $z$ is unique and, in fact, for this unique $z$
we have $z=w$. 

Since this constituent can be detected by taking homomorphism from
$\mathbb{S}(L_w)$ to the complex $I_w[2\mathbf{a}(w_0w)]$, applying the
inverse of $\mathbb{S}$ to this hom-space, we get the dual of Claim~\eqref{cor-s5.8-1.3}
and the dual of the ungraded formula  in Claim~\eqref{cor-s5.8-1.4}. The
graded formula in Claim~\eqref{cor-s5.8-1.4} follows from 
Lemma~\ref{lem9}.

Finally, to prove Claim~\eqref{cor-s5.8-1.5}, let us assume 
$\mathrm{Ext}^{2\mathbf{a}(w_0x)}(I_x,L_y)\neq 0$ for some $y>_L x$.
Since the projective dimension of $I_x$ equals $2\mathbf{a}(w_0x)$,
this non-zero extension induces a non-zero extension
$\mathrm{Ext}^{2\mathbf{a}(w_0x)}(I_x,P_y)\neq 0$.
But this, in turn, implies that the injective dimension
$2\mathbf{a}(w_0y)$ of $P_y$ is at least $2\mathbf{a}(w_0x)$.
This, however, contradicts the strict monotonicity of the 
$\mathbf{a}$ function as we assumed $y>_L x$.
\end{proof}

\subsection{Degrees}\label{s5.85}

Denote by $\mathbf{S}$ the set of all integers of the form
$2\mathbf{a}(w_0d)$, where $d\in\mathbf{D}$. For $i\in\mathbf{S}$,
consider the following subcategories of $\mathcal{O}_0$:
\begin{itemize}
\item  the full subcategory $\mathbf{CY}_i$ of $\mathcal{O}_0$
consisting of all objects $M\in \mathcal{O}_0$ such that $\mathbb{S}(M)\cong M[i]$;
\item  the full subcategory $\mathbf{PCY}_i$ of $\mathcal{O}_0$
consisting of all objects $M\in \mathcal{O}_0$, for which there exists an object 
$M'\in \mathcal{O}_0$ such that $\mathbb{S}(M)\cong M'[i]$;
\item  the full subcategory $\mathbf{F}_i$ of $\mathcal{O}_0$
consisting of all objects $M\in \mathcal{O}_0$ which admit a filtration 
\begin{equation}\label{eq-abs1}
0=M_0\subset M_1\subset \dots \subset M_{k-1}\subset M_k=M 
\end{equation}
such that each subquotient $M_j/M_{j-1}$ is isomorphic to 
$\theta_w L_d$, for some $d\in\mathbf{D}$ such that 
$\mathbf{a}(w_0d)=i$ and some $w\in W$ such that $w\sim_R d$
(both $w$ and $d$ might depend on $j$).
\end{itemize}
The notation $\mathbf{PCY}$ abbreviates {\em provisionally Calabi-Yau}.
We record the following observations.

\begin{proposition}\label{prop-s5.85-1}
For $i\in\mathbf{S}$, we have:

\begin{enumerate}[$($a$)$]
\item\label{prop-s5.85-1.1}
both $\mathbf{CY}_i\subset \mathbf{PCY}_i$ and $\mathbf{F}_i\subset \mathbf{PCY}_i$;
\item\label{prop-s5.85-1.2}
both $\mathbf{CY}_i$, $\mathbf{PCY}_i$ and $\mathbf{F}_i$
are closed with respect to finite direct sums;
\item\label{prop-s5.85-1.3}
both $\mathbf{CY}_i$, $\mathbf{PCY}_i$ and $\mathbf{F}_i$ are closed with 
respect to the action of $\mathscr{P}$;
\item\label{prop-s5.85-1.4}
$\mathbf{PCY}_i$ is closed with respect to extensions, 
kernels of epimorphisms and cokernels of monomorphisms; 
\item\label{prop-s5.85-1.5}
both $\mathbf{PCY}_i$ and $\mathbf{F}_i$ are exact categories in which the exact structure
is defined by the usual short exact sequences.
\end{enumerate}
\end{proposition}

\begin{proof}
To start with, $\mathbf{CY}_i\subset \mathbf{PCY}_i$ follows from the definitions
combined with the fact that $\theta_w L_d\in \mathbf{CY}_i$, for $d\in\mathbf{D}$ such that 
$\mathbf{a}(w_0d)=i$ and $w\in W$ such that $w\sim_R d$, given by Theorem~\ref{conj1}.
Also, Claim~\eqref{prop-s5.85-1.2} follows from the definitions and the additivity 
of the Serre functor $\mathbb{S}$.

Assume $M,N\in \mathbf{PCY}_i$ and 
\begin{equation}\label{eq-s5.85-1-7}
0\to M\to X\to N\to 0
\end{equation}
be a short exact sequence.
This sequence corresponds to a distinguished triangle $N[-1]\to M\to X\to N$ in
$\mathcal{D}^b(\mathcal{O}_0)$. As $\mathbb{S}$ is triangulated, we have a
distinguished triangle 
\begin{equation}\label{eq-s5.85-1-5}
\mathbb{S}(N)[-1]\to \mathbb{S}(M)\to \mathbb{S}(X)\to \mathbb{S}(N). 
\end{equation}
We have $\mathbb{S}(M)\cong M'[i]$ and $\mathbb{S}(N)\cong N'[i]$, for some
$M',N'\in \mathcal{O}_0$, due to our assumption that $M,N\in \mathbf{PCY}_i$.
The distinguished triangle given by \eqref{eq-s5.85-1-5} implies existence of 
some $X'\in \mathcal{O}_0$ such that we both have $\mathbb{S}(X)\cong X'[i]$ and 
also an exact sequence 
\begin{equation}\label{eq-s5.85-1-9}
0\to M'\to X'\to N'\to 0.
\end{equation}
Consequently,
$X\in \mathbf{PCY}_i$. This proves that $\mathbf{PCY}_i$ is extension closed.
In particular, it follows that $\mathbf{F}_i$, being the extension closure of
the modules of the form $\theta_w L_d\in \mathbf{CY}_i$, for $d\in\mathbf{D}$ such that 
$\mathbf{a}(w_0d)=i$ and $w\in W$ such that $w\sim_R d$, by definition,
is a subcategory of $\mathbf{PCY}_i$. This proves Claim~\eqref{prop-s5.85-1.1}.

The above proves the first assertion of Claim~\eqref{prop-s5.85-1.4}. The two other
assertions are proved by the same argument.

Next, Claim~\eqref{prop-s5.85-1.5} for $\mathbf{F}_i$ follows from the
definition of $\mathbf{F}_i$, while Claim~\eqref{prop-s5.85-1.5} for 
$\mathbf{PCY}_i$ follows from Claim~\eqref{prop-s5.85-1.4}.

Finally, let us prove Claim~\eqref{prop-s5.85-1.3}. If $\mathbb{S}(M)\cong M'[i]$
and $\theta$ is a projective functor, then 
$\theta \mathbb{S}(M)\cong\mathbb{S}(\theta M)\cong
\theta M'[i]$ as $\theta$ commutes with $\mathbb{S}$. This proves Claim~\eqref{prop-s5.85-1.3}
for both $\mathbf{PCY}_i$ and $\mathbf{CY}_i$. To prove Claim~\eqref{prop-s5.85-1.3}
for $\mathbf{F}_i$, we just need to show that $\theta \theta_w L_d\in \mathbf{F}_i$,
for any $d\in\mathbf{D}$ such that $\mathbf{a}(w_0d)=i$ and $w\in W$ such that $w\sim_R d$. 
For this, we note that $\theta \theta_w$ is a direct sum of functors of the form 
$\theta_{w'}$, where $w'\in W$ such that $w'\sim_R d$, and projective functors from 
strictly greater two-sided KL-cells. The latter projective functors annihilate $L_d$,
see \cite[Lemma~12]{MM1}. This completes the proof.
\end{proof}

Now we can describe all possible degrees for Calabi-Yau objects in $\mathcal{O}_0$.

\begin{corollary}\label{cor-s5.85-2}
Let $i\in\mathbb{Z}$ and $0\neq M\in\mathcal{O}_0$ be such that 
$\mathbb{S}(M)\cong M[i]$. Then there exists $d\in\mathbf{D}$ such that
$i=2\mathbf{a}(w_0d)$.
\end{corollary}

\begin{proof}
Let $x\in W$ be such that $L_x$ is a subquotient of $M$ and, for any $y\in W$
such that $L_y$ is a subquotient of $M$, we have $y\not\geq_J x$.
Let  $d$ be the Duflo element in the left KL-cell of $x$.
Then, for any $y\in W$ such that $L_y$ is a subquotient of $M$, we
either have $\theta_d L_y=0$ or $y\sim_L x$ by \cite[Lemma~12]{MM1}.

For $y\sim_L x$, let $d'$ be the Duflo element in the right KL-cell of $y$.
Then $\theta_d L_y$ belongs to the additive closure of the modules of the form $\theta_w L_{d'}$,
for $w\sim_R d'$, see \cite[Proposition~4.15]{MMMTZ2}. This implies that,
for $i=2\mathbf{a}(w_0d')$, we have $\theta_{d}M\in\mathbf{F}_i$.
Now the claim of our corollary follows from Proposition~\ref{prop-s5.85-1}.
\end{proof}

\begin{corollary}\label{cor-s5.85-3}
Let $d\in\mathbf{D}$ and $0\neq M\in\mathcal{O}_0$ be such that 
$\mathbb{S}(M)\cong M[2\mathbf{a}(w_0d)]$. Then, if $M$ is graded,
we have $\mathbb{S}(M)\cong 
M\langle 2(\mathbf{a}(d)-\mathbf{a}(w_0d))\rangle[2\mathbf{a}(w_0d)]$.
\end{corollary}

\begin{proof}
This is proved similarly to the proof of Corollary~\ref{cor-s5.85-2}.
\end{proof}

\begin{remark}
{\em 
Some further information about the categories 
$\mathbf{F}_i$, $\mathbf{CY}_i$ and $\mathbf{PCY}_i$
will be obtained latter in Subsection~\ref{s8.25775}.
}
\end{remark}

\section{Homomorphisms from $\mathbb{S}$ to shifts of the identity}\label{s8}

\subsection{Motivation from categorical diagonalization}\label{s8.1}

One of the main ideas of categorical diagonalization described in
the papers \cite{EH1,EH2} is that 
isomorphisms similar to $\mathbb{S}(M)\cong 
M\langle 2(\mathbf{a}(x)-\mathbf{a}(w_0x))\rangle[2\mathbf{a}(w_0d)]$
should not only exist as abstract isomorphisms, but they should come
as evaluations at $M$ of natural transformations
from $\mathbb{S}$ to 
$\mathrm{Id}\langle 2(\mathbf{a}(x)-\mathbf{a}(w_0x))\rangle[2\mathbf{a}(w_0d)]$.
Here the collection $\{\mathrm{Id}\langle t\rangle[s]\,:\,s,t\in\mathbb{Z}\}$
is referred to as a {\em choice of scalars} for the setup in question.

In this section we will take a closer look at the possibility of such
phenomena in the context of category $\mathcal{O}$.

\subsection{Two extreme cases}\label{s8.2}

\begin{proposition}\label{prop-s8.2-1}
There is a unique, up to scalar,  non-zero natural transformation 
$\alpha_{w_0}:\mathbb{S}\to \mathrm{Id}\langle 2(\mathbf{a}(w_0))\rangle$.
Moreover, the evaluation of $\alpha_{w_0}$ at $P_{w_0}\cong \theta_{w_0}L_{w_0}$
is an isomorphism.
\end{proposition}

\begin{proof}
We will use that $\mathbb{S}=(\mathcal{L}\top_{w_0})^2$, see 
Subsection~\ref{s2.7}. For a simple reflection $s$, consider
$\mathcal{L}\top_s$. By definition, the $0$-th homology of this
functor is $\top_s$. By \cite[Theorem~4]{KM}, see also
\cite[Proposition~2.3]{MS3}, there is a natural transformation
$\top_s\to \mathrm{Id}\langle 1\rangle$ whose evaluation at 
$P_{w_0}$ is an isomorphism. Composing these natural transformations
along some reduced expression of $w_0$, we get $\alpha_{w_0}$.

Uniqueness of $\alpha_{w_0}$, up to scalar, follows from 
\cite[Theorem~6]{MS3}.
\end{proof}

\begin{proposition}\label{prop-s8.2-2}
There exists a unique, up to scalar,  non-zero natural transformation 
$\alpha_{e}:\mathbb{S}\to 
\mathrm{Id}\langle -2(\mathbf{a}(w_0))\rangle[2(\mathbf{a}(w_0)]$.
Moreover, the evaluation of $\alpha_{e}$ at 
$L_{e}\cong \theta_{e}L_{e}$
is an isomorphism.
\end{proposition}

\begin{proof}
We will use that $\mathbb{S}=(\mathcal{L}\top_{w_0})^2$, see 
Subsection~\ref{s2.7}. For a simple reflection $s$, 
let $\mathfrak{p}=\mathfrak{p}_s$ be the (minimal) parabolic subcategory of 
$\mathfrak{g}$ corresponding to $s$. Consider the functor 
$\mathcal{L}\top_s$.

By \cite[Theorem~2.2]{AS}, only the first two homologies of
$\mathcal{L}\top_s$ are non-zero.
By \cite[Theorem~2]{MS3}, the $1$-st homology of 
$\mathcal{L}\top_s$ is isomorphic to the dual 
Zuckerman functor $Z_{\mathfrak{p}}$. 
The embedding $\iota_{\mathfrak{p}}$
of $\mathcal{O}_0^\mathfrak{p}$ into $\mathcal{O}_0$
gives rise to a natural transformation from 
$\iota_{\mathfrak{p}} Z_{\mathfrak{p}}$ to 
$\mathrm{Id}\langle -1\rangle$ whose evaluation at 
$L_{e}$ is an isomorphism. Since $\mathcal{L}\top_s$ has no higher homologies,
this induces a natural transformation $\beta_s$ from
$\mathcal{L}\top_s$ to $\mathrm{Id}\langle -1\rangle[1]$.
Composing these $\beta_s$
along some reduced expression of $w_0$, we get $\alpha_{e}$.

By adjunction, in the category of functors, we have  
$\mathrm{Hom}(\iota_{\mathfrak{p}} Z_{\mathfrak{p}},\mathrm{Id})\cong
\mathrm{Hom}(Z_{\mathfrak{p}},Z_{\mathfrak{p}})$.
The algebra $\mathrm{Hom}(Z_{\mathfrak{p}},Z_{\mathfrak{p}})$ 
is isomorphic to the endomorphism algebra of the identity endofunctor of 
$\mathcal{O}_0^\mathfrak{p}$, that is, to the center of $A^\mathfrak{p}$, 
by \cite[Theorem~6]{MS3}. The algebra $A^\mathfrak{p}$ is
positively graded and hence its center is positively graded as well.
This center is local as $\mathcal{O}_0^\mathfrak{p}$ is connected.
This means that the homogeneous part of degree zero in this center
is one-dimensional (as $\mathbb{C}$ is algebraically closed).
This implies uniqueness of $\alpha_{e}$, up to scalar.
\end{proof}

\subsection{General case}\label{s8.25}

\begin{theorem}\label{prop-s8.25-1}
For any $d\in \mathbf{D}$ and $w\in W$ such that 
$w\sim_R d$, there exists a natural transformation
\begin{displaymath}
\alpha_d:\mathbb{S}\to 
\mathrm{Id}\langle 2(\mathbf{a}(d)-\mathbf{a}(w_0d))\rangle[2\mathbf{a}(w_0d)]
\end{displaymath}
such that the evaluation of $\alpha_d$ at $\theta_{w}L_{d}$
is an isomorphism.
\end{theorem}

We expect $\alpha_d$ to be a scalar multiple of $\alpha_{d'}$
if (and, probably, only if) $d\sim_J d'$.

\begin{remark}\label{rem-s8.25-15}
{\rm 
The function $\mathbf{a}$ is constant on two-sided cells and strictly
monotone with respect to the two-sided order. However, the two-sided
order is not linear in general. It is easy to find examples of
$W$ and $x,y\in W$ such that $\mathbf{a}(x)=\mathbf{a}(y)$
but $\mathbf{a}(w_0x)\neq \mathbf{a}(w_0y)$. For instance,
we can take $W=S_7$ with $x$ in the two-sided cell corresponding to
the partition $(5,1,1)$ and $y$ in the two-sided cell corresponding to
the partition $(4,3)$. Then $\mathbf{a}(x)=\mathbf{a}(y)=3$ while
$\mathbf{a}(w_0x)=10\neq 9=\mathbf{a}(w_0y)$. Thus we might have 
homomorphisms from $\mathbb{S}$ to $\mathrm{Id}$ corresponding 
to different two-sided cells which live in the same
homological position but in different degrees.
More details on this example can be found in Subsection~\ref{s7.3}.
}
\end{remark}

I do not know how to extend 
Theorem~\ref{prop-s8.25-1} to Soergel bimodules  over
the polynomial algebra.

\subsection{Preparation}\label{s8.26}

To prove Proposition~\ref{prop-s8.25-1}, we will need some
preparation. Essentially, we need to go over the proof of 
\cite[Theorem~3]{KMM} with a fine-tooth comb.

For any $w\in W$, denote by $\widehat{w}$
the Duflo element in the left KL-cell of $w$.

\begin{lemma}\label{lem-s8.25-255}
For any $x,w\in W$, $d\in\mathbf{D}$,
$k\in\mathbb{Z}_{\geq 0}$ and $i\in\mathbb{Z}$, 
we have:
\begin{enumerate}[$($i$)$]
\item\label{lem-s8.25-255-000} 
$\mathrm{Ext}^k(T_w,L_x)\neq 0$ implies $k\leq\mathbf{a}(w)$.
\item\label{lem-s8.25-255-0000} 
$\mathrm{ext}^k(T_w,L_x\langle i\rangle)\neq 0$ implies 
$-\mathbf{a}(x)-k\leq i\leq \mathbf{a}(x)-k$.
\item\label{lem-s8.25-255-00} 
$\mathrm{Ext}^{\mathbf{a}(w)}(T_w,L_x)\neq 0$ implies 
$x\sim_L {w_0w}$.
\item\label{lem-s8.25-255-0} 
$\mathrm{Ext}^{k}(T_w,L_x)\neq 0$ and $k<\mathbf{a}(w)$ implies 
$\mathbf{a}(x)>\mathbf{a}(w_0w)$.
\item\label{lem-s8.25-255-1} 
$\mathrm{ext}^{\mathbf{a}(d)}(T_d,L_{\widehat{w_0d}}
\langle\mathbf{a}(w_0d)-\mathbf{a}(d)\rangle)\cong\mathbb{C}$.
\item\label{lem-s8.25-255-5} 
If $x\sim_L w_0d$,
the inequality
$\mathrm{ext}^{\mathbf{a}(d)}(T_d,L_{x}
\langle i\rangle)\neq 0$ implies both
$i=\mathbf{a}(w_0d)-\mathbf{a}(d)$
and $x=\widehat{w_0d}$.
\end{enumerate}
\end{lemma}

\begin{proof}
Claim~\eqref{lem-s8.25-255-000} follows from \cite[Theorem~17]{Maz2}.
As $T_w=\theta_{w_0w}T_{w_0}$, 
Claim~\eqref{lem-s8.25-255-00} follows from \cite[Theorem~A]{KMM2}.

To prove Claim~\eqref{lem-s8.25-255-0}, consider 
$\mathcal{P}_\bullet(T_{w_0})$ and note that $T_{w_0}=L_{w_0}$.
Therefore the complex $\mathcal{P}_\bullet(T_{w_0})\in\mathcal{LP}(\mathcal{O}_0)$
is sent to to $\nabla_e$ via the Koszul duality. By the
KL-combinatorics, if some $P_y$ appears as a summand of 
$\mathcal{P}_{-k}(T_{w_0})$, for $k<\mathbf{a}(w)$, then 
$\mathbf{a}(w_0y)<\mathbf{a}(w)$, that is $\mathbf{a}(y)>\mathbf{a}(w_0w)$.
Applying $\theta_{w_0w}$ to $\mathcal{P}_\bullet(T_{w_0})$, we get some,
not necessarily minimal, projective resolution of $T_w$, as projective
functors are exact and send projective objects to projective objects.
Now Claim~\eqref{lem-s8.25-255-0000} follows from the definition of 
the $\mathbf{a}$-function, see Subsection~\ref{s2.3}.
Further, any summand $P_x$ of $\mathcal{P}_{-k}(T_{w})$ is a summand of 
$\theta_{w_0w}\mathcal{P}_{-k}(T_{w_0})$, that is of some 
$\theta_{w_0w}P_y$, for $y$ as above. Therefore we have 
$\mathbf{a}(x)\geq \mathbf{a}(y)$, implying Claim~\eqref{lem-s8.25-255-0}.

As $T_d\cong \theta_{w_0d}T_{w_0}$, by adjunction, we have
\begin{displaymath}
\mathrm{ext}^{\mathbf{a}(d)}(T_{d},L_{\widehat{w_0d}}\langle i\rangle)\cong 
\mathrm{ext}^{\mathbf{a}(d)}(T_{w_0},\theta_{dw_0}L_{\widehat{w_0d}}\langle i\rangle).
\end{displaymath}
Since all Duflo elements are involutions, the element $\widehat{w_0d}$, 
being the Duflo element in the left KL-cell of $w_0d$, is also the 
Duflo element in the  right KL-cell of $dw_0=(w_0d)^{-1}$.
The module $\theta_{dw_0}L_{\widehat{w_0d}}$ is Koszul-Ringel dual
to the module $\theta_{\widehat{w_0d}w_0}L_{d}$ by
\cite[Theorem~16]{Maz2}. 
As $\widehat{w_0d}\sim_R dw_0 $,
we have $\widehat{w_0d}w_0\sim_R d$
by \cite[Page~179]{BjBr}. It follows that
$\theta_{\widehat{w_0d}w_0}L_{d}$ has simple top 
$L_{\widehat{w_0d}w_0}$. The latter corresponds to
$\theta_{w_0\widehat{w_0d}w_0}L_{w_0}\cong T_{\widehat{w_0d}w_0}$
under the Koszul-Ringel duality. This means that the
$\mathbf{a}(d)$-position of the complex
$\mathcal{T}_\bullet(\theta_{dw_0}L_{\widehat{w_0d}})\in\mathcal{LT}(\mathcal{O}_0)$
is isomorphic to $T_{\widehat{w_0d}w_0}\langle \mathbf{a}(d)\rangle$.
%

The module
$T_{\widehat{w_0d}w_0}$ is the image of $P_{w_0\widehat{w_0d}w_0}$ under $\top_{w_0}$,
moreover, any costandard filtration of $T_{\widehat{w_0d}w_0}$ is the image of
a standard filtration of $P_{w_0\widehat{w_0d}w_0}$ under $\top_{w_0}$.
Homomorphisms from $T_{w_0}$  to $T_{\widehat{w_0d}w_0}$ correspond to
costandard subquotients of the form $\nabla_{w_0}$ of $T_{\widehat{w_0d}w_0}$,
which, in turn, correspond to standard subquotients of the form
$\Delta_e$ in $P_{w_0\widehat{w_0d}w_0}$.

By the BGG reciprocity, 
the standard subquotients of  the form $\Delta_e$ in $P_{w_0\widehat{w_0d}w_0}$
correspond to simple subquotients of the form $L_{w_0\widehat{w_0d}w_0}$
in $\Delta_e$.
Since conjugation by $w_0$ corresponds to an automorphism of the Dynkin diagram,
$w_0\widehat{w_0d}w_0\in\mathbf{D}$. Clearly,
$\mathbf{a}(w_0\widehat{w_0d}w_0)=\mathbf{a}(dw_0)$. From the definition
of the $\mathbf{a}$-function, the simple subquotients 
$L_{w_0\widehat{w_0d}w_0}$ in $\Delta_e$ appear in degrees between
$\mathbf{a}(dw_0)$ and $\ell(\widehat{w_0d})$,
with $\mathbf{a}(dw_0)\leq \ell(\widehat{w_0d})$. 

To establish Claim~\eqref{lem-s8.25-255-1}, it remains to compare
with Claim~\eqref{lem-s8.25-255-0000}. The shift by 
$\mathbf{a}(dw_0)$ is the minimal possible by the previous paragraph,
but is the maximal one allowed by Claim~\eqref{lem-s8.25-255-0000}.
Hence Claim~\eqref{lem-s8.25-255-1} follows from
$[\Delta_e:L_{w_0\widehat{w_0d}w_0}\langle -\mathbf{a}(dw_0)\rangle]=1$
and the next paragraph, which, together with
Claim~\eqref{lem-s8.25-255-00}, implies that 
$\mathrm{Ext}^{\mathrm{a}(d)}(T_d,L_{\widehat{w_0d}})\neq 0$.

The same arguments applied to some $x\sim_L w_0d$,
different from $\widehat{w_0d}$, will give strictly larger
shifts as 
$[\Delta_e:L_{x}\langle -i\rangle]\neq 0$
implies $i>\mathbf{a}(dw_0)$ and this is not allowed
by Claim~\eqref{lem-s8.25-255-0000}. This proves 
Claim~\eqref{lem-s8.25-255-5} and completes the proof of the lemma.
\end{proof}

\begin{lemma}\label{lem-s8.25-26}
For any $w\in W$, $k\in\mathbb{Z}_{\geq 0}$
and $i\in\mathbb{Z}$, we have:
\begin{enumerate}[$($i$)$]
\item\label{lem-s8.25-26-1} 
$\mathrm{Ext}^{k}(\nabla_e,L_{w})\neq 0$ implies $k\geq 2\mathbf{a}(w_0w)$.
\item\label{lem-s8.25-26-15} 
$\mathrm{ext}^{k}(\nabla_e,L_{w}\langle i\rangle)\neq 0$ 
implies $-k\leq i\leq \mathbf{a}(x)-k$.
\item\label{lem-s8.25-26-3} 
$\mathrm{ext}^{2\mathbf{a}(w_0w)}(\nabla_e,L_{w}\langle i\rangle)\neq 0$ implies 
$i= \mathbf{a}(w)-2\mathbf{a}(w_0w)$.
\item\label{lem-s8.25-26-4} 
$\mathrm{ext}^{2\mathbf{a}(w_0w)}(\nabla_e,L_{w}\langle 
\mathbf{a}(w)-2\mathbf{a}(w_0w)\rangle)\neq 0$ implies
$w\in \mathbf{D}$.
\item\label{lem-s8.25-26-5} 
$\mathrm{ext}^{2\mathbf{a}(w_0w)}(\nabla_e,L_{w}\langle 
\mathbf{a}(w)-2\mathbf{a}(w_0w)\rangle)\cong\mathbb{C}$, if $w\in \mathbf{D}$.
\end{enumerate}
\end{lemma}

\begin{proof}
Claim~\eqref{lem-s8.25-26-1} follows from \cite[Theorem~3]{KMM}
and \cite[Theorem~20]{Maz2}. 

To prove the rest,
consider $\mathcal{T}_\bullet(\nabla_e)\in\mathcal{LT}(\mathcal{O}_0)$.
Under the Koszul-Ringel self-duality, this corresponds to 
$\nabla_e$. From the Kazhdan-Lusztig combinatorics, if some
$T_x\langle -k\rangle$ is a summand of some
$\mathcal{T}_{-k}(\nabla_e)$, then $\mathbf{a}(x)\leq k\leq\ell(x)$.
Moreover, if $k=\mathbf{a}(x)$, then $x\in\mathbf{D}$.
To construct a (not necessarily minimal) 
projective resolution of $\nabla_e$, we need to glue
minimal projective resolutions of all such $T_x\langle -k\rangle$.
In particular, the right inequality in Claim~\eqref{lem-s8.25-26-15} follows directly from
Lemma~\ref{lem-s8.25-255}\eqref{lem-s8.25-255-0000}.
The left inequality in Claim~\eqref{lem-s8.25-26-15} follows from the 
Koszulity of $\mathcal{O}_0$. Indeed, due to this Koszulity, extensions
between simple modules of $\mathcal{O}_0$ live on the main diagonal
(i.e. the diagonal defined by the ``homological position equals the graded shift''
condition).
Since the module $\nabla_e$ lives in non-positive degrees, it follows
that all extensions from it to simple modules live in non-negative 
shifts of the main diagonal.

A minimal projective resolution $\mathcal{P}_\bullet(T_x)$ 
of $T_x$ has length $\mathbf{a}(x)$. From 
Lemma~\ref{lem-s8.25-255}\eqref{lem-s8.25-255-0}, we see that,
if some $T_y$ appears as  a summand of some $\mathcal{P}_{m}(T_x)$,
for $m>-\mathbf{a}(x)$, then $\mathbf{a}(y)> \mathbf{a}(w_0x)$. At the same time,
if $x\in\mathbf{D}$ and
$T_y$ is a summand of $\mathcal{P}_{-\mathbf{a}(x)}(T_x)$,
then $y=\widehat{w_0x}$ by Lemma~\ref{lem-s8.25-255}\eqref{lem-s8.25-255-5}.
Now, Claims~\eqref{lem-s8.25-26-3}, \eqref{lem-s8.25-26-4}
and \eqref{lem-s8.25-26-5} follow from 
Lemma~\ref{lem-s8.25-255}\eqref{lem-s8.25-255-5} and 
and Lemma~\ref{lem-s8.25-255}\eqref{lem-s8.25-255-1}.
\end{proof}

\subsection{Proof of Theorem~\ref{prop-s8.25-1}}\label{s8.257}
For $d\in\mathbf{D}$, consider the minimal projective resolution 
$\mathcal{P}_\bullet(I_d)$ of $I_d$. We know that 
$\mathcal{P}_i(I_d)=0$, for $i<-2\mathbf{a}(w_0d)$,
by \cite[Theorem~20]{Maz2}; and also that
$\mathcal{P}_{-2\mathbf{a}(w_0d)}(I_d)\cong 
P_d\langle 2(\mathbf{a}(d)-\mathbf{a}(w_0d))\rangle$,
by Corollary~\ref{cor-s5.8-1}. We also know that 
any $P_x$ appearing as a summand of $\mathcal{P}_i(I_d)$,
for $i>-2\mathbf{a}(w_0d)$, satisfies $x>_L d$,
by Corollary~\ref{cor-s5.8-1}.

The natural projection $P_d\tto L_d$ gives rise to a non-zero
homomorphism from $\mathcal{P}_\bullet(I_d)$ to
$L_d\langle 2(\mathbf{a}(d)-\mathbf{a}(w_0d))\rangle[2\mathbf{a}(w_0d)]$
in the derived category. In fact, we have
\begin{displaymath}
\mathcal{D}^b({}^\mathbb{Z}\mathcal{O}_0)\big(
\mathcal{P}_\bullet(I_d),
L_d\langle 2(\mathbf{a}(d)-\mathbf{a}(w_0d))\rangle[2\mathbf{a}(w_0d)]
\big) \cong \mathbb{C}. 
\end{displaymath}
As $\mathcal{P}_\bullet(I_d)\cong \theta_d\mathcal{P}_\bullet(I_e)$
and $\theta_d$ is self-adjoint, by adjunction we have 
\begin{equation}\label{eq-123-nn}
\mathcal{D}^b({}^\mathbb{Z}\mathcal{O}_0)\big(
\mathcal{P}_\bullet(I_e),\theta_d
L_d\langle 2(\mathbf{a}(d)-\mathbf{a}(w_0d))\rangle[2\mathbf{a}(w_0d)]
\big) \cong \mathbb{C}. 
\end{equation}
Now, recall that, up to graded shift, all simple subquotients of 
$\theta_d L_d$ are of the form $L_x$, where
$x\leq_R d$, see \cite[Lemma~13]{MM1}. At the same time,
from Lemma~\ref{lem-s8.25-26} it follows that the only summand
of $\mathcal{P}_{-2\mathbf{a}(w_0d)}(I_e)$ that is an indecomposable
projective cover of a simple module of the latter form is the summand
$P_d\langle \mathbf{a}(d)-2\mathbf{a}(w_0d)\rangle$. This means that
the map which realizes a non-zero element in the space
\eqref{eq-123-nn} comes from a projection $\pi$ of 
$P_d\langle \mathbf{a}(d)-2\mathbf{a}(w_0d)\rangle$
onto $L_d\langle \mathbf{a}(d)-2\mathbf{a}(w_0d)\rangle$ 
which is the socle of $\theta_d
L_d\langle 2(\mathbf{a}(d)-\mathbf{a}(w_0d))\rangle[2\mathbf{a}(w_0d)]$,
see \cite[Corollary~3]{Maz2}.

Since $d\in\mathbf{D}$, there is a non-zero morphism $\varphi$ from
$\Delta_e\langle 2(\mathbf{a}(d)-\mathbf{a}(w_0d))\rangle[2\mathbf{a}(w_0d)]$
to $\theta_d
L_d\langle 2(\mathbf{a}(d)-\mathbf{a}(w_0d))\rangle[2\mathbf{a}(w_0d)]$
see \cite[Proposition~17]{MM1},
and so the image of $\pi$ belongs to the image of $\varphi$.

From the definition of the $\mathbf{a}$-function, we also have
a unique, up to scalar, non-zero homomorphism
$\psi:P_d\tto \Delta_e\langle\mathbf{a}(d)\rangle$. This homomorphism has
the property that any simple $L_w$ appearing in the cokernel is
killed by $\theta_d$ and hence satisfies $d\not\leq_L w$,
see \cite[Proposition~17]{MM1} and \cite[Lemma~12]{MM1}.
In particular, there are no non-zero homomorphism from 
$\mathcal{P}_{1-2\mathbf{a}(w_0d)}(I_d)$ to the cokernel of
$\psi$. This implies that
\begin{displaymath}
\mathcal{D}^b({}^\mathbb{Z}\mathcal{O}_0)\big(
\mathcal{P}_\bullet(I_d),
\Delta_e\langle 3\mathbf{a}(d)-2\mathbf{a}(w_0d)\rangle[2\mathbf{a}(w_0d)]
\big) \cong \mathbb{C}. 
\end{displaymath}
As $\mathcal{P}_\bullet(I_d)\cong \theta_d\mathcal{P}_\bullet(I_e)$,
$\theta_d P_e\cong P_d$
and $\theta_d$ is self-adjoint, by adjunction we have 
\begin{equation}\label{eq-123-nnn}
\mathcal{D}^b({}^\mathbb{Z}\mathcal{O}_0)\big(
\mathcal{P}_\bullet(I_e),
P_d\langle 3\mathbf{a}(d)-2\mathbf{a}(w_0d)\rangle[2\mathbf{a}(w_0d)]
\big) \cong \mathbb{C}. 
\end{equation}

The module $\theta_d
L_d\langle 2(\mathbf{a}(d)-\mathbf{a}(w_0d))\rangle$
is a quotient of 
$P_d\langle 3\mathbf{a}(d)-2\mathbf{a}(w_0d)\rangle$,
in fact, this quotient map is the image, under $\theta_d$,
of the projection from the image $\mathrm{Im}(\psi)$ of 
$\psi$ onto the simple top of $\mathrm{Im}(\psi)$, as the 
cokernel of $\psi$ is killed by the exact functor $\theta_d$.

The module $P_d$ has a Verma filtration. By the BGG-reciprocity
and the definition of $\mathbf{a}$, the module 
$\Delta_e\langle 2(\mathbf{a}(d)-\mathbf{a}(w_0d))\rangle$
appears exactly once as a subquotient of this filtration.
Using the usual long exact sequence arguments,
our observation in the previous paragraph and 
the fact that we already established that the 
image of $\pi$ belongs to the image of $\varphi$, we obtain that
any non-zero element in \eqref{eq-123-nnn} induces
a non-zero element $\eta$ in 
\begin{equation}\label{eq-123-nnnn}
\mathcal{D}^b({}^\mathbb{Z}\mathcal{O}_0)\big(
\mathcal{P}_\bullet(I_e),
\Delta_e\langle 2(\mathbf{a}(d)-\mathbf{a}(w_0d))\rangle[2\mathbf{a}(w_0d)]
\big). 
\end{equation}
In particular, it follows that
\begin{displaymath}
\mathrm{ext}^{2\mathbf{a}(w_0d)}
(\nabla_e,\Delta_e\langle 2(\mathbf{a}(d)-\mathbf{a}(w_0d))\rangle)
\neq 0. 
\end{displaymath}

Now let us recall that $\mathbb{S}$, being a composition of 
derived twisting functors, commutes with the action of
the bicategory of projective functors. 
The identity functor obviously 
commutes with the action of the bicategory of projective functors.  
When evaluated at the dominant object
$\Delta_e$, the functor $\mathbb{S}$ outputs $\nabla_e$,
while the identity outputs $\Delta_e$.
By \cite[Theorem~1]{Kh}, this determines out functors uniquely,
up to isomorphism.
Furthermore, by \cite[Theorem~2]{Kh}, our $\eta$ induces a non-zero 
natural transformation from $\mathbb{S}$ to  
$\mathrm{Id}\langle 2(\mathbf{a}(d)-\mathbf{a}(w_0d))\rangle[2\mathbf{a}(w_0d)]$. 
This is our $\alpha_d$.

It remains to show that $\alpha_d$ evaluates at $\theta_w L_d$ to
a non-zero map. Indeed, $\theta_w L_d$ is indecomposable as a
quotient of (shifted) $P_w$, moreover, the latter also gives that the 
endomorphism algebra of $\theta_w L_d$ is a quotient of the 
endomorphism algebra of $P_w$ and thus is positively graded.
Consequently, any non-zero endomorphism of $\theta_w L_d$
of degree zero is an automorphism.

By our above construction of $\alpha_d$, we see that it is defined
in terms of the evaluation of the canonical map 
$\theta_d\to\theta_e\langle \mathbf{a}(d)\rangle$ 
as constructed in \cite[Proposition~17]{MM1}
at $\theta_w L_d$. This canonical map is non-zero when 
evaluated at $L_d$ by construction. The fact that it is non-zero
when evaluated on $\theta_w L_d$ follows from the adjunction
axioms as $\theta_w L_d\neq 0$.

This completes the proof of Theorem~\ref{prop-s8.25-1}.

\subsection{Consequences}\label{s8.25775}

Using Theorem~\ref{prop-s8.25-1}, we can now strengthen the
assertion of Proposition~\ref{prop-s5.85-1}:

\begin{corollary}\label{cor-8.25775-1}
For $i\in\mathbf{S}$, we have $\mathbf{F}_i\subset \mathbf{CY}_i$.
\end{corollary}

We do expect that $\mathbf{F}_i=\mathbf{CY}_i$, for all $i\in\mathbf{S}$,
however, at the moment we can only prove this in the two extreme cases,
see Subsection~\ref{s6.55}.

To prove Corollary~\ref{cor-8.25775-1}, we would
need the following lemma.
Denote by $\alpha$ some linear combination of all $\alpha_d$, where
$d\in\mathbf{D}$ is such that $2\mathbf{a}(w_0d)=i$.

\begin{lemma}\label{lem-8.25775-2}
We can choose $\alpha$ such that,
for any $d\in\mathbf{D}$ with $2\mathbf{a}(w_0d)=i$
and for any $w\in W$ with $w\sim_R d$, the evaluation
of $\alpha$ at $\theta_w L_d$ is an isomorphism. 
\end{lemma}

\begin{proof}
Recall that $\theta_w L_d$ is an indecomposable module
whose endomorphism algebra is positively graded. In particular,
any homogeneous endomorphism of $\theta_w L_d$ 
of degree zero is a scalar
multiple of the identity. For $d'\in\mathbf{D}$ such that 
$2\mathbf{a}(w_0d')=i$, consider the corresponding $\alpha_{d'}$. 
Then we have the corresponding scalar $\lambda_{w,d,d'}$ 
with which the evaluation of $\alpha_{d'}$ acts on
$\theta_w L_d$. 

Write $\alpha=\sum_{d'} c_{d'}\alpha_{d'}$.
We know from Theorem~\ref{prop-s8.25-1} that $\lambda_{w,d,d}\neq 0$.
Therefore the set of all those coefficient vectors $(c_{d'})$ for which
the evaluation of $\alpha$ at $\theta_w L_d$  is zero
(i.e. $\sum_{d'} c_{d'} \lambda_{w,d,d'}=0$) is a proper
subspace of the space of all coefficient vectors. Note 
that the set of all coefficient vectors is just a non-zero
finite dimensional complex vector space.
Recall that
a non-zero finite dimensional complex vector space is not
the union of a finite set of its proper subspaces
(e.g. because any such union has measure zero). Therefore
there is a vector of coefficients for which the corresponding
$\alpha$, when evaluated at any $\theta_w L_d$, is non-zero and
hence an isomorphism, as asserted.
\end{proof}

\begin{proof}[Proof of Corollary~\ref{cor-8.25775-1}.]
For $M\in \mathbf{F}_i$, it is enough to show that the 
evaluation of the natural transformation
$\alpha$ given by Lemma~\ref{lem-8.25775-2} at $M$ is an isomorphism.
We proceed by induction on the length $k$
of the filtration \eqref{eq-abs1}. If $k=1$, then 
$M\in \mathbf{CY}_i$ by Theorem~\ref{conj1}.

Assume now that we have a short exact sequence 
\begin{displaymath}
\xymatrix{
0\ar[r]&N\ar[r]^f&M\ar[r]^g&\theta_w L_d\ar[r]& 0 
},
\end{displaymath}
for some $d\in \mathbf{D}$ and $w\in W$ such that $w\sim_R d$,
with $N\in \mathbf{F}_i$. By induction, we may assume that 
the evaluation of $\alpha$ at $N$ is an isomorphism.
This short exact sequence corresponds to
a distinguished triangle 
\begin{displaymath}
\xymatrix{
N\ar[r]^f&M\ar[r]^g&\theta_w L_d\ar[r]^h& N[1]
}
\end{displaymath}
in the derived category. Since $\alpha$ is a natural transformation,
we have a commutative diagram
\begin{displaymath}
\xymatrix{
\mathbb{S}N\ar[r]^{\mathbb{S}f}\ar[d]_{\alpha_{{}_N}}&
\mathbb{S}M\ar[r]^{\mathbb{S}g}\ar[d]_{\alpha_{{}_M}}&
\mathbb{S}\theta_w L_d\ar[r]^{\mathbb{S}h}\ar[d]_{\alpha_{{}_{\theta_w L_d}}}& 
\mathbb{S}N[1]\ar[d]_{\alpha_{{}_N}[1]}\\
N\ar[r]^f&M\ar[r]^g&\theta_w L_d\ar[r]^h& N[1]
}
\end{displaymath}
In this diagram, $\alpha_{{}_N}$,
and hence also $\alpha_{{}_N}[1]$, are isomorphisms by the 
inductive assumption and $\alpha_{{}_{\theta_w L_d}}$
is an isomorphism by the basis of the induction. 
Now the fact  that $\alpha_{{}_M}$ is an isomorphism
follows from the Five Lemma, completing the proof.
\end{proof}

\section{Various additional results}\label{s6}

\subsection{Some extensions in the singular case}\label{s6.1}

For appropriate singular blocks, our results on
extensions from the dominant dual Verma module to the dominant Verma
module can be made more precise as follows:

\begin{lemma}\label{lem-s8.25-20}
Let $\mathfrak{p}$ be a parabolic subalgebra of $\mathfrak{g}$ 
containing $\mathfrak{b}$. Let $\lambda$ be a dominant integral
weight whose dot-stabilizer is $W^\mathfrak{p}$.
Then
\begin{displaymath}
\mathrm{ext}^{2\mathbf{a}(w_0w_0^\mathfrak{p})}
\big(\nabla(\lambda),\Delta(\lambda)\langle i\rangle\big)\cong
\begin{cases}
\mathbb{C}, & i=2(\mathbf{a}(w_0^\mathfrak{p})-\mathbf{a}(w_0w_0^\mathfrak{p}));\\
0, & \text{otherwise}.
\end{cases}
\end{displaymath}
\end{lemma}

\begin{proof}
Let $\theta^\mathrm{on}$
be the projective functor of translation onto the $w_0^\mathfrak{p}$-wall from
$\mathcal{O}_0$ to $\mathcal{O}_\lambda$. Let $\theta^\mathrm{out}$
be the projective functor of translation out of the $w_0^\mathfrak{p}$-wall from
$\mathcal{O}_\lambda$ to $\mathcal{O}_0$. These two functors are 
biadjoint, the composition $\theta^\mathrm{out}\theta^\mathrm{on}$
is isomorphic to $\theta_{{}_{w_0^{\mathfrak{p}}}}$, and the composition
$\theta^\mathrm{on}\theta^\mathrm{out}$ is isomorphic to a direct sum of
$|W^\mathfrak{p}|$ copies of the identity functor on $\mathcal{O}_\lambda$,
see \cite[Proposition~4.1]{CM}. The more detailed graded decomposition is:
\begin{displaymath}
\theta^\mathrm{on}\theta^\mathrm{out}\cong
\bigoplus_{i\geq 0}
\mathrm{Id}_{{}_{\mathcal{O}_\lambda}}
\langle 2i-\ell(w_0^\mathfrak{p})
\rangle^{\oplus |\{w\in W^\mathfrak{p}\,:\,\ell(w)=i\}|}.
\end{displaymath}

We have $\theta^\mathrm{on}\nabla_e\cong \nabla(\lambda)
\langle \ell(w_0^\mathfrak{p})\rangle$, then
$\theta^\mathrm{out}\nabla(\lambda)\langle \ell(w_0^\mathfrak{p})\rangle
= I_{w_0^\mathfrak{p}}$ and, finally,
\begin{displaymath}
\theta^\mathrm{on}I_{w_0^\mathfrak{p}}\cong
\bigoplus_{i\geq 0}\nabla(\lambda)
\langle 2i\rangle^{\oplus |\{w\in W^\mathfrak{p}\,:\,\ell(w)=i\}|}.
\end{displaymath}
As projective functors are exact and map projectives to projectives,
it follows that the projective dimensions of $\nabla(\lambda)$
and $I_{w_0^\mathfrak{p}}$ coincide. The latter module has projective
dimension $2\mathbf{a}(w_0w_0^\mathfrak{p})$ by \cite[Theorem~20]{Maz2}.
Moreover, Corollary~\ref{cor-s5.8-1} describes explicitly the 
last non-zero term of a minimal projective resolution of $I_{w_0^\mathfrak{p}}$.
Namely, this term is 
$P_{w_0^\mathfrak{p}}
\langle 2(\mathbf{a}(w_0^\mathfrak{p})-\mathbf{a}(w_0w_0^\mathfrak{p}))\rangle$.
 
Note that we dually have  $\theta^\mathrm{on}\Delta_e\cong \Delta(\lambda)
\langle -\ell(w_0^\mathfrak{p})\rangle$, then
$\theta^\mathrm{out}\Delta(\lambda)\langle -\ell(w_0^\mathfrak{p})\rangle
= P_{w_0^\mathfrak{p}}$ and, finally,
\begin{displaymath}
\theta^\mathrm{on}P_{w_0^\mathfrak{p}}\cong
\bigoplus_{i\geq 0}\Delta(\lambda)
\langle -2i\rangle^{\oplus |\{w\in W^\mathfrak{p}\,:\,\ell(w)=i\}|}.
\end{displaymath}
By matching the summands, it follows that the  last non-zero term of a 
minimal projective resolution of $\nabla(\lambda)$ is
$\Delta(\lambda)
\langle 2(\mathbf{a}(w_0^\mathfrak{p})-\mathbf{a}(w_0w_0^\mathfrak{p}))\rangle$.
Since all Verma modules have trivial endomorphism algebras, 
the claim of the lemma follows.
\end{proof}

\subsection{Self-extensions of the simple $\mathtt{C}$-modules}\label{s6.3}

Consider the coinvariant algebra $\mathtt{C}$ with its natural positive
grading in which  the generators have degree $1$. Note that this differs 
from the grading associated to the interpretation of $\mathtt{C}$
as $\mathrm{End}(P_{w_0})$ as in the latter the generators have degree $2$.
So, we need to be careful with rescaling when interpreting the  
content of this subsection in the setup of category $\mathcal{O}$.

Let $\mathbb{C}$  be the unique simple $\mathtt{C}$-module which we consider
as the graded module  concentrated in degree $0$. Recall, see 
\cite[Subsection~1.1]{St2}, the description of 
$\mathrm{Ext}^*_\mathtt{C}(\mathbb{C},\mathbb{C})$.

Here is the list of the facts which we will need:
\begin{itemize}
\item as an algebra,  $\mathrm{Ext}^*_\mathtt{C}(\mathbb{C},\mathbb{C})$
is generated by $\mathrm{Ext}^1_\mathtt{C}(\mathbb{C},\mathbb{C})$
and $\mathrm{Ext}^2_\mathtt{C}(\mathbb{C},\mathbb{C})$;
\item as a vector space, $\mathrm{Ext}^*_\mathtt{C}(\mathbb{C},\mathbb{C})$ 
is isomorphic to the tensor product of the exterior algebra 
$\bigwedge \mathrm{Ext}^1_\mathtt{C}(\mathbb{C},\mathbb{C})$
and the symmetric algebra of a certain subspace $V$
of $\mathrm{Ext}^2_\mathtt{C}(\mathbb{C},\mathbb{C})$ of dimension $\dim(\mathfrak{h})$;
\item  $\mathrm{Ext}^1_\mathtt{C}(\mathbb{C},\mathbb{C})$ 
is a homogeneous space of dimension $\dim(\mathfrak{h})$ 
and is concentrated in degree $1$;
\item $V$ is the direct sum of homogeneous
subspaces whose degrees  are the degrees  of the algebraically independent
generators of the algebra of $W$-invariant polynomials
(note that these  degrees are  connected, using the $\pm 1$ shift, to 
the exponents of the root system, see \cite{Lee}),
moreover, those subspaces have dimension $1$, if the root system is irreducible.
\end{itemize}
This information allows  us to efficiently compute $\mathrm{Ext}^*_\mathtt{C}(\mathbb{C},\mathbb{C})$,
especially, for small values of $*$.
For example,  in the case of a root system of type $A_n$, the degrees 
in question are $2,3,\dots,n+1$. Consequently, in this case, we  have
\begin{equation}\label{eq-51}
\dim  \mathrm{ext}^2_\mathtt{C}(\mathbb{C},\mathbb{C}\langle -i\rangle)=
\begin{cases}
\binom{n}{2}+1,& i=2;\\
1,& i=3,4,\dots,n+1;\\
0, & \text{otherwise}.
\end{cases}
\end{equation}
Here we see the contribution $\binom{n}{2}$ from 
$\bigwedge \mathrm{Ext}^1_\mathtt{C}(\mathbb{C},\mathbb{C})$
and all the $1$'s, including that in degree  $2$, add up to the contribution of $V$.

\subsection{Some stability phenomena}\label{s6.4}

Relevance of $\mathrm{Ext}^*_\mathtt{C}(\mathbb{C},\mathbb{C})$
for the problems considered in the previous section is 
explained by the following proposition:

\begin{proposition}\label{prop-s6.4-1}
For all $k,m\in\mathbb{Z}_{\geq 0}$ with 
$m\leq k$, if $k=0$, and $m\leq k+1$, if $k>0$, 
and all $i\in\mathbb{Z}$, we have:
\begin{equation}\label{eq-a1b3}
\mathrm{ext}^m_{{}^\mathbb{Z}\mathcal{O}_0}
(\mathbb{S}^k\Delta_e,\Delta_e\langle i\rangle)
\cong
\begin{cases}
\mathrm{ext}^m_\mathtt{C}(\mathbb{C},\mathbb{C}
\langle \frac{i}{2}-k\ell(w_0)\rangle), & i\equiv 0\mod 2;\\
0, & i\equiv 1\mod 2.
\end{cases}
\end{equation}
\end{proposition}

\begin{proof}
For $k=0$, we have $m=0$ and 
\begin{displaymath}
\mathrm{ext}^0_{{}^\mathbb{Z}\mathcal{O}_0}
(\Delta_e,\Delta_e\langle i\rangle)
\cong\mathrm{ext}^0_\mathtt{C}(\mathbb{C},\mathbb{C}\langle i\rangle)\cong
\begin{cases}
\mathbb{C}, & i=0;\\
0, & i\neq 0;
\end{cases}
\end{displaymath}
is, obviously, true.

Consider
the linear tilting resolution $\mathcal{T}_\bullet(\nabla_e)$
of $\nabla_e$. To form a projective resolution of $\nabla_e$,
we need to glue projective resolutions of each tilting summand
$T_w$ appearing in $\mathcal{T}_\bullet(\nabla_e)$. 
Let $\mathcal{X}^1_\bullet$ be the resulting complex. 
Note that $\mathcal{P}_\bullet(\nabla_e)$ is a summand of  
$\mathcal{X}^1_\bullet$ and can be obtained from the latter
by removing subcomplexes homotopic to zero.

Now, for $k>1$, define $\mathcal{X}^k_\bullet$ recursively
from $\mathcal{X}^{k-1}_\bullet$ as follows: apply
$\mathbb{S}$ to $\mathcal{X}^{k-1}_\bullet$ to get a complex
of injective modules and then glue projective resolutions of 
all individual indecomposable summands of that complex.
Again,  $\mathcal{P}_\bullet(\mathbb{S}^k\Delta_e)$ is a summand of  
$\mathcal{X}^k_\bullet$ and can be obtained from the latter
by removing subcomplexes homotopic to zero.

Note that each tilting and each injective module has 
a projective cover by a sum of copies of $P_{w_0}$.
Also note that the latter module is unchanged under
$\mathbb{S}$ (up to isomorphism and 
apart from the shift in grading by $2\ell(w_0)$).
Consequently, all $\mathcal{X}^k_m$, for $m\leq k$,
are sums of copies of $P_{w_0}$. And all these components
and the differentials between them are unchanged, up to isomorphism
and shift of grading, when $m$ is fixed and $k$ increases.

Now we note that 
\begin{equation}\label{eq-a1b2}
\mathrm{ext}^m_{{}^\mathbb{Z}\mathcal{O}_0}
(\mathbb{S}^k\Delta_e,\Delta_e\langle i\rangle)=
\mathrm{hom}
(\top_{w_0}\mathcal{X}^k_{-m},\top_{w_0}\Delta_e\langle i\rangle)/N_m,
\end{equation}
where $N_m$ is the subspace generated by all maps that factor through
$\top_{w_0}\mathcal{X}^k_{1-m}$. Also note 
that $\top_{w_0}\Delta_e=L_{w_0}$. For a fixed $m\leq k$,
the previous paragraph implies that the right hand side
of \eqref{eq-a1b2}
does not depend on $k$. Since $\mathbb{V}$ is exact and sends 
$\Delta_e$ to the simple $\mathtt{C}$-module $\mathbb{C}$
and $P_{w_0}$ to the regular $\mathtt{C}$-module $\mathtt{C}$,
we obtain that the right hand side of \eqref{eq-a1b2}
is computed by the corresponding right hand side of 
\eqref{eq-a1b3}. 

It remains to consider the case $m=k+1$, for $k>0$, in which we need a bit more
care. The component $\mathcal{X}^k_{-k-1}$ might have, up to shift of
grading, indecomposable projective summands of the form
$P_w$, where $w\neq w_0$. Applying $\top_{w_0}$ maps such $P_w$
to $T_{w_0w}$ and the latter has a projective cover whose all summands
are of the form $P_{w_0}$, up to shift. This cover also covers
$I_{w}\cong \mathbb{S}P_w$. This implies that the above arguments
extend to the case $m=k+1$, but this is the absolute limit
of these arguments. The claim follows.
\end{proof}

\subsection{Homological position $-2$ in type $A$}\label{s6.5}

Consider the type  $A_n$, for $n\geq 3$. In this case we have the  following:
\begin{itemize}
\item the minimum partition with respect to  the two-sided order is $(n)$,
the $\mathbf{a}$-value of the corresponding two-sided KL-cell is $0$;
\item after removing $(n)$, the remaining minimum partition with respect 
to  the two-sided order is $(n-1,1)$, the $\mathbf{a}$-value of the 
corresponding two-sided KL-cell is $1$;
\item the maximum partition with respect to  the two-sided order is $(1^n)$,
the $\mathbf{a}$-value of the corresponding two-sided KL-cell is $\binom{n+1}{2}$;
\item after removing $(1^n)$, the remaining maximum partition with respect 
to  the two-sided order is $(2,1^{n-2})$, the $\mathbf{a}$-value of the 
corresponding two-sided KL-cell is $\binom{n}{2}$.
\end{itemize}
Note that, in types $A_1$ and $A_2$ some  of the partitions listed above coincide.

Recall the shifts $\langle 2\big(\mathbf{a}(d^\mathfrak{p})-\mathbf{a}(w_0d^\mathfrak{p})\big)\rangle
[2\mathbf{a}(w_0d^\mathfrak{p})]$ in \eqref{eq-1}.
If our Duflo element $d^\mathfrak{p}$ equals $w_0$
(i.,e. belongs to the two-sided cell  corresponding to $(1^n)$), this formula  outputs
$\langle n(n+1)\rangle [0]$.
If our Duflo element $d^\mathfrak{p}$ belongs to the two-sided cell corresponding
to $(2,1^{n-2})$, this formula  outputs
$\langle n(n-1)-2\rangle [2]$.

Now let us look at the degree shift between $n(n+1)$ and $n(n-1)-2$.
The difference is $2(n+1)$. At the same  time, Formula~\eqref{eq-51}
says that the maximal expected shift between a non-trivial 
$\mathrm{ext}^0_\mathtt{C}(\mathbb{C},\mathbb{C}\langle -i\rangle)$
and a non-trivial $\mathrm{ext}^2_\mathtt{C}(\mathbb{C},\mathbb{C}\langle -j\rangle)$ 
is $n+1$, which becomes $2(n+1)$ in the setup of category $\mathcal{O}$.
Therefore the extension given by Theorem~\ref{prop-s8.25-1}
corresponds exactly to this extreme shift.

\subsection{Classification of Calabi-Yau objects in the two extreme cases}\label{s6.55}

\begin{proposition}\label{prop-s6.55-1}
We have $\mathbf{F}_0=\mathbf{CY}_0=\mathrm{add}(P_{w_0})$. 
\end{proposition}

\begin{proof}
Note that $\mathbf{F}_0$
consists of those modules in $\mathcal{O}_0$ which have a filtration whose
subquotients are isomorphic to the projective-injective module $P_{w_0}=\theta_{w_0}L_{w_0}$.
This means that $\mathbf{F}_0=\mathrm{add}(P_{w_0})$, that is
the category of all projective-injective objects in $\mathcal{O}_0$.
So, due to Corollary~\ref{cor-8.25775-1}, we just need to prove
$\mathbf{CY}_0\subset \mathbf{F}_0$. 

Let $M\in \mathcal{O}_0$ be a non-zero object such that $\mathbb{S}M\cong M$.
As $\nabla_e$ is a quotient of $P_{w_0}$, the $0$-th component of
$\mathcal{P}_\bullet(\nabla_e)$ is $P_{w_0}$, which implies that 
$M$ is a quotients of $\theta_{w_0}M$. This means that each simple
constituent of the top  of $M$ is isomorphic to $L_{w_0}$.
In particular, for any simple reflection $s$, the evaluation of
the natural  transformation $\top_s\to \theta_e$ at $M$ is surjective.
Hence, by our construction of $\alpha_{w_0}$ in Proposition~\ref{prop-s8.2-1},
the evaluation of  $\alpha_{w_0}$ at $M$ is surjective, and hence an
isomorphism as $\mathbb{S}M\cong M$.

Let $P$ be a minimal projective cover of $M$. Then each summand of 
$P$  is isomorphic to $P_{w_0}$. In particular,   $P\in \mathbf{F}_0\subset \mathbf{CY}_0$.
Consider a short exact sequence
\begin{displaymath}
0\to N\to P \to M\to 0. 
\end{displaymath}
Here $N$ is isomorphic, in the derived category and up  to shift, to the  cone of the projection $P\to M$.
We know  that both $P$ and $M$  are in $\mathbf{CY}_0$ and that the evaluation of
$\alpha_{w_0}$ at both these modules is an isomorphism. Now the same  argument as in the
proof of Corollary~\ref{cor-8.25775-1} shows that the evaluation of
$\alpha_{w_0}$ at $N$ is an isomorphism. In particular, $N\in  \mathbf{CY}_0$.

We can now repeat the same argument for $N$ and continue recursively. This must,
however, stop after finitely many steps since $\mathcal{O}_0$ has finite
global  dimension. It follows that $M$ has a projective resolution consisting
of injective modules.  Therefore  $M$ itself  is both projective and injective,
that is belongs to $\mathbf{F}_0$.
The claim follows.
\end{proof}

\begin{proposition}\label{prop-s6.55-2}
We have $\mathbf{F}_{2\mathbf{a}(w_0)}=\mathbf{CY}_{2\mathbf{a}(w_0)}=
\mathrm{add}(L_e)$. 
\end{proposition}

\begin{proof}
By Corollary~\ref{cor-8.25775-1}, we know that 
$\mathbf{F}_{2\mathbf{a}(w_0)}\subset \mathbf{CY}_{2\mathbf{a}(w_0)}$. 
Let $M\in \mathbf{CY}_{2\mathbf{a}(w_0)}$ and $w\in W$ be such that $w\neq e$.
We claim that $[M:L_w]=0$.

Indeed, if $[M:L_w]\neq 0$, 
then $\mathcal{D}^b(\mathcal{O}_0)(\mathbb{S}M,I_w)\neq 0$.
Applying $\mathbb{S}^{-1}$, we get the inequality
$\mathrm{Ext}^{2\mathbf{a}(w_0)}(M,P_w)\neq 0$. This, however,
is not possible as the injective dimension of $P_w$ is $2\mathbf{a}(w_0w)<
2\mathbf{a}(w_0)$ by \cite[Theorem~20]{Maz2}.

Therefore the only composition subquotients of $M$ are $L_e$, in particular,
$M\in \mathbf{F}_{2\mathbf{a}(w_0)}$. Also,
by Weyl's theorem on complete reducibility, it follows that
$M\in \mathrm{add}(L_e)$. This completes the proof.
\end{proof}

\section{Parabolic category $\mathcal{O}$ and its generalizations}\label{s9}

\subsection{The category $\mathcal{O}^{\hat{\mathcal{R}}}_0$}\label{s9.1}

Fix a right KL-cell $\mathcal{R}$ in $W$ and denote by 
$\hat{\mathcal{R}}$ the ideal which $\mathcal{R}$  generates 
with respect to the right order, that is 
\begin{displaymath}
\hat{\mathcal{R}}=\{w\in W\,:\, w\leq_R x\text{ for some }x\in \mathcal{R}\}.
\end{displaymath}
Denote by $\mathcal{O}^{\hat{\mathcal{R}}}_0$ the Serre subcategory of
$\mathcal{O}_0$ generated by all $L_w$, where $w\in \hat{\mathcal{R}}$.
If $\mathcal{R}$ contains the element $w_0^\mathfrak{p}w_0$,
for some parabolic subalgebra  $\mathfrak{p}$ of $\mathfrak{g}$
containing $\mathfrak{b}$, then $\mathcal{O}^{\hat{\mathcal{R}}}_0=
\mathcal{O}^{\mathfrak{p}}_0$. Hence, the categories of the form 
$\mathcal{O}^{\hat{\mathcal{R}}}_0$ are natural generalizations of 
the parabolic category $\mathcal{O}$, see \cite{MS4}.

In the general case, the structure of $\mathcal{O}^{\hat{\mathcal{R}}}_0$
is not as nice as that of blocks of parabolic category 
$\mathcal{O}$, for example, $\mathcal{O}^{\hat{\mathcal{R}}}_0$
might fail to be a highest weight category, see \cite[Lemma~11]{MS5}.

The category $\mathcal{O}^{\hat{\mathcal{R}}}_0$ is stable with respect
to the action of projective functors and inherits from
$\mathcal{O}_0$ a natural $\mathbb{Z}$-grading. For $w\in \hat{\mathcal{R}}$,
we denote by $P^{\hat{\mathcal{R}}}_w$ and $I^{\hat{\mathcal{R}}}_w$
the indecomposable projective cover and injective envelope of 
$L_w$ in $\mathcal{O}^{\hat{\mathcal{R}}}_0$, respectively.
The module $P^{\hat{\mathcal{R}}}_w$ is injective if and only if 
$w\in \mathcal{R}$. In the latter case, 
$P^{\hat{\mathcal{R}}}_w\cong I^{\hat{\mathcal{R}}}_w$.

\subsection{Kostant's problem and dominant dimension}\label{s9.2}

For $w\in W$, we have two natural birepresentations of $\mathscr{P}$
associated with $L_x$. The first one is 
$\mathscr{P}/\mathrm{Ann}_\mathscr{P}(L_
w)$ and the other one
is $\displaystyle\mathrm{add}(\bigoplus_{x\in W}\theta_x L_w)$. There is the 
obvious morphism of birepresentations from the former to the latter,
given by sending $\theta$ to $\theta L_w$. Recall, see 
\cite[Corollary~7.6]{KMM2}, that Kostant's problem, in the sense of \cite{Jo2},
has positive solution for $L_w$ if and only if above morphism of
birepresentations is an equivalence.

For $d\in\mathbf{D}$, the Duflo element in a KL right cell $\mathcal{R}$, 
\cite[Theorem~5]{KaM10} says that Kostant's problem has positive solution 
for $L_d$ if and only if the cokernel of the natural embedding
$ P^{\hat{\mathcal{R}}}_e\hookrightarrow P^{\hat{\mathcal{R}}}_d$
embeds into a projective-injective object in $\mathcal{O}^{\hat{\mathcal{R}}}_0$.
This is equivalent to the property that $\mathcal{O}^{\hat{\mathcal{R}}}_0$
has dominant dimension at least two with respect to projective-injective objects,
see \cite{KSX} for the details on the latter notion.

One particular case is the following: Kostant's problem has positive solution 
for $L_d$, where $d\in\mathbf{D}$ is the Duflo element in the right KL-cell
$\mathcal{R}$ which contains $w_0^\mathfrak{p}w_0$, for some parabolic 
subalgebra $\mathfrak{p}$ of $\mathfrak{g}$ containing $\mathfrak{b}$,
see \cite[Corollary~18]{KaM10}.

\subsection{Applying $\mathbb{S}$ to 
$\mathcal{O}^{\hat{\mathcal{R}}}_0$}\label{s9.3}

\begin{proposition}\label{prop-s9.3-1}
Let $\mathcal{R}$ be a right KL-cell and
$d\in\mathbf{D}$ be the Duflo element in $\mathcal{R}$.
Then, for any $-2\mathbf{a}(w_0d)< i\leq 0$
and $M\in\mathcal{O}^{\hat{\mathcal{R}}}_0 $,
the $i$-th homology of $\mathbb{S}M$
is zero.
\end{proposition}

\begin{proof}
Each composition subquotient of $M$ is of the form $L_x$, for some 
$x\leq_R d$. Since $-2\mathbf{a}(w_0d)<i$, each summand $P_y$ of 
$\mathcal{P}_i(\nabla_e)$ satisfies $\mathbf{a}(y)>\mathbf{a}(d)$.
Therefore $\theta_y L_x=0$. The claim follows.
\end{proof}

\begin{proposition}\label{prop-s9.3-2}
Let $\mathcal{R}$ be a KL-right cell and
$d\in\mathbf{D}$ be the Duflo element in $\mathcal{R}$.
Then the condition that Kostant's problem has positive 
solution for $L_d$ is equivalent to the condition that 
the homology  of
$\mathbb{S}P^{\hat{\mathcal{R}}}_e$ at position $-2\mathbf{a}(w_0d)$
is isomorphic to $I^{\hat{\mathcal{R}}}_e$.
\end{proposition}

\begin{proof}
Let $M$ denote the homology  of
$\mathbb{S}P^{\hat{\mathcal{R}}}_e$ at position $-2\mathbf{a}(w_0d)$.

The only element $w\in W$ such that $P_w$ is a summand of
$\mathcal{P}_{-2\mathbf{a}(w_0d)}(\nabla_e)$ and
$\theta_w L_d\neq 0$ is $w=d$. Therefore $M$ is a quotient of
$\theta_d P^{\hat{\mathcal{R}}}_e\cong \theta_d L_d$.
At the same time, from $\theta_d M\cong \theta_d L_d$,
which is true by Theorem~\ref{conj1}, it follows that
$[M:L_d]=1$ and $[M:L_u]=0$, for any $u\sim_H d$
different from $d$. Moreover, if $d'\in \mathbf{D}$ is such that $d'\sim_J d$
and $d'\neq d$, then $\theta_{d'}P^{\hat{\mathcal{R}}}_e=0$
and thus $\theta_{d'} M=0$. This implies that $[M:L_x]=0$,
for any $x\sim_R d$ different from $d$.

This means that $M$ is a quotient of the module $N$ defined 
as the quotient of $P^{\hat{\mathcal{R}}}_d$ by the trace of 
all $P^{\hat{\mathcal{R}}}_x$, where $x\sim_R d$, in the
radical of $P^{\hat{\mathcal{R}}}_d$. 

The kernel of the map $N\tto M$ contains only simples of the form
$L_x$, where $x<_R d$. However, the corresponding projectives
$P_x$ cannot be summands of $\mathcal{P}_{-2\mathbf{a}(w_0d)-1}(\nabla_e)$
as $w_0x>_R w_0d$ and therefore $-2\mathbf{a}(w_0x)<-2\mathbf{a}(w_0d)-1$,
since $\mathbf{a}(w_0x)>\mathbf{a}(w_0d)$.
This means that $N$ is isomorphic to $M$.

The fact that $N\cong I^{\hat{\mathcal{R}}}_e$ is equivalent to 
the condition that Kostant's problem has positive 
solution for $L_d$ follows from 
\cite[Theorem~5]{KaM10}, see Subsection~\ref{s9.2}.
The claim follows.
\end{proof}

\begin{corollary}\label{cor-s9.3-4}
Let $\mathcal{R}$ be a KL-right cell and
$d\in\mathbf{D}$ be the Duflo element in $\mathcal{R}$.
Then $\mathbb{S}P^{\hat{\mathcal{R}}}_e[-2\mathbf{a}(w_0d)]
\cong I^{\hat{\mathcal{R}}}_e$ if and only if 
Kostant's problem has positive solution for $L_d$ 
and the projective dimension of $P^{\hat{\mathcal{R}}}_e$
in $\mathcal{O}$ equals $2\mathbf{a}(w_0d)$.
\end{corollary}

\begin{proof}
Assume that the projective dimension of $P^{\hat{\mathcal{R}}}_e$
in $\mathcal{O}$ equals $m$ and let $P_w$ be a summand of 
$\mathcal{P}_{-m}(P^{\hat{\mathcal{R}}}_e)$. Then the identity on
$P_w$ gives rise to a non-zero extension in $\mathcal{O}$, of 
degree $m$, from $P^{\hat{\mathcal{R}}}_e$ to $P_w$. Applying
$\mathbb{S}$, we get a non-zero homomorphism from 
$\mathbb{S}P^{\hat{\mathcal{R}}}_e$ to $I_w[m]$, that is,
the homology of $\mathbb{S}P^{\hat{\mathcal{R}}}_e$ at the
homological position $-m$ is non-zero.

We know that the homology of $\mathbb{S}P^{\hat{\mathcal{R}}}_e$ at the
homological position $-2\mathbf{a}(w_0d)$ is non-zero. Therefore
the projective dimension of $P^{\hat{\mathcal{R}}}_e$
in $\mathcal{O}$ is at least $2\mathbf{a}(w_0d)$. If the inequality 
is strict, then $\mathbb{S}P^{\hat{\mathcal{R}}}_e$ has a non-zero
homology at some position different from $-2\mathbf{a}(w_0d)$
and therefore $\mathbb{S}P^{\hat{\mathcal{R}}}_e[-2\mathbf{a}(w_0d)]
\cong I^{\hat{\mathcal{R}}}_e$ is not possible.

If the projective dimension of $P^{\hat{\mathcal{R}}}_e$
in $\mathcal{O}$ is exactly $2\mathbf{a}(w_0d)$, then 
the homology of $\mathbb{S}P^{\hat{\mathcal{R}}}_e$ 
is concentrated at the homological position $-2\mathbf{a}(w_0d)$.
Now the claim of the corollary follows from 
Proposition~\ref{prop-s9.3-2}.
\end{proof}

\begin{corollary}\label{cor-s9.3-3}
Let $\mathcal{R}$ be a KL-right cell and
$d\in\mathbf{D}$ be the Duflo element in $\mathcal{R}$.
Then $\mathbb{S}[-2\mathbf{a}(w_0d)]$ is an ungraded Serre functor
on the category of perfect complexes in 
$\mathcal{O}_0^{\hat{\mathcal{R}}}$
if and only if the following conditions are satisfied:
\begin{enumerate}[$($a$)$]
\item\label{cor-s9.3-3.1} 
Kostant's problem has positive solution for $L_d$;
\item\label{cor-s9.3-3.2}
the projective dimension of $P^{\hat{\mathcal{R}}}_e$
in $\mathcal{O}$ equals $2\mathbf{a}(w_0d)$;
\item\label{cor-s9.3-3.3}
the projective dimension of $I^{\hat{\mathcal{R}}}_e$
in $\mathcal{O}_0^{\hat{\mathcal{R}}}$ is finite.
\end{enumerate}
\end{corollary}

\begin{proof}
If $\mathbb{S}[-2\mathbf{a}(w_0d)]$ is a Serre functor
on the category of perfect complexes in 
$\mathcal{O}_0^{\hat{\mathcal{R}}}$, then 
$\mathbb{S}P^{\hat{\mathcal{R}}}_e[-2\mathbf{a}(w_0d)]
\cong I^{\hat{\mathcal{R}}}_e$ which implies both 
\eqref{cor-s9.3-3.1} and \eqref{cor-s9.3-3.2}
by Corollary~\ref{cor-s9.3-4}. The necessity of 
\eqref{cor-s9.3-3.3} follows from the assumption that 
category of perfect complexes in 
$\mathcal{O}_0^{\hat{\mathcal{R}}}$
has a Serre functor.

Conversely, if we assume \eqref{cor-s9.3-3.3}, we know that 
the category of perfect complexes in 
$\mathcal{O}_0^{\hat{\mathcal{R}}}$
has a Serre functor. Additionally, assuming \eqref{cor-s9.3-3.1},
\eqref{cor-s9.3-3.2}, we have that 
$\mathbb{S}P^{\hat{\mathcal{R}}}_e[-2\mathbf{a}(w_0d)]
\cong I^{\hat{\mathcal{R}}}_e$
by Corollary~\ref{cor-s9.3-4}. Note that the 
homology functor of $\mathbb{S}$ at position
$-2\mathbf{a}(w_0d)$ naturally commutes with 
projective functors by \cite[Lemma~8]{Kh}. Since this
homology functor maps $P^{\hat{\mathcal{R}}}_e$
to $I^{\hat{\mathcal{R}}}_e$, it thus must be isomorphic to 
the Nakayama functor on $\mathcal{O}_0^{\hat{\mathcal{R}}}$.

Therefore $\mathbb{S}[-2\mathbf{a}(w_0d)]$ is a triangulated
endofunctor of the category of perfect complexes in 
$\mathcal{O}_0^{\hat{\mathcal{R}}}$ which is isomorphic 
to the Serre functor when restricted to the category of 
projective modules. Then the two functors are isomorphic,
completing the proof.
\end{proof}

\begin{remark}\label{rem-s9.3-4}
{\rm
The conditions \eqref{cor-s9.3-3.1}, \eqref{cor-s9.3-3.2} and 
\eqref{cor-s9.3-3.3} of Corollary~\ref{cor-s9.3-3} are satisfied 
provided that $\mathcal{R}$ contains the element
$w_0^\mathfrak{p}w_0$, for some parabolic subalgebra
$\mathfrak{p}$ of $\mathfrak{g}$ containing $\mathfrak{b}$.
Indeed, in this case $\mathcal{O}_0^{\hat{\mathcal{R}}}\cong
\mathcal{O}_0^\mathfrak{p}$ has finite global dimension, 
which implies Corollary~\ref{cor-s9.3-3}\eqref{cor-s9.3-3.3}.
Condition Corollary~\ref{cor-s9.3-3}\eqref{cor-s9.3-3.1}
follows from \cite[Corollary~18]{KaM10}. Finally, 
condition Corollary~\ref{cor-s9.3-3}\eqref{cor-s9.3-3.2}
follows from \cite[Proposition~4.4]{MS2} using the arguments in the proof of
Corollary~\ref{cor-s9.3-3}.

It is possible that the categories $\mathcal{O}_0^\mathfrak{p}$
are the only categories of the form $\mathcal{O}_0^{\hat{\mathcal{R}}}$
that satisfy all the conditions \eqref{cor-s9.3-3.1}, \eqref{cor-s9.3-3.2} and 
\eqref{cor-s9.3-3.3} of Corollary~\ref{cor-s9.3-3}.
A detailed example in Subsection~\ref{s7.5} points in that 
direction. If this turns out to be the case, then the results of
this subsection are not stronger than the corresponding results
of \cite{MS2}, however, our proof that $\mathbb{S}[-2\mathbf{a}(w_0d)]$ 
is a Serre functor is different (our proof uses \cite{Kh} while the proof
of \cite{MS2} is a computation based on the self-duality of derived
Zuckerman functors).
}
\end{remark}

\subsection{Calabi-Yau objects for $\mathcal{O}^{\hat{\mathcal{R}}}_0$}\label{s9.8}

\begin{corollary}\label{cor-s9.8-1}
Assume that the conditions \eqref{cor-s9.3-3.1}, \eqref{cor-s9.3-3.2} and 
\eqref{cor-s9.3-3.3} of Corollary~\ref{cor-s9.3-3} are satisfied. Then we have:
\begin{enumerate}[$($a$)$]
\item \label{cor-s9.8-1.1}
$\mathbb{S}\langle 2\mathbf{a}(w_0d)\rangle[-2\mathbf{a}(w_0d)]$ is a graded Serre functor
for the category of perfect complexes over $\mathcal{O}_0^{\hat{\mathcal{R}}}$.
\item \label{cor-s9.8-1.2} For $d'\in\mathbf{D}$ and $w\in W$ such that $d'\leq_R d$
and $w\sim_R d'$, the object $\theta_w L_{d'}$ is a Calabi-Yau object in 
$\mathcal{O}_0^{\hat{\mathcal{R}}}$ of dimension $2(\mathbf{a}(w_0d')-\mathbf{a}(w_0d))$.
\item \label{cor-s9.8-1.3} The evaluation of
$\alpha_{d'}[-2\mathbf{a}(w_0d)]\langle 2\mathbf{a}(w_0d)\rangle$
at $\theta_w L_{d'}$ gives rise to an isomorphism
between $\mathbb{S}\langle 2\mathbf{a}(w_0d)\rangle[-2\mathbf{a}(w_0d)] \theta_w L_{d'}$
and
\begin{displaymath}
\theta_w L_{d'}\langle 2(\mathbf{a}(d')-\mathbf{a}(w_0d')+\mathbf{a}(w_0d))\rangle
[2(\mathbf{a}(w_0d')-\mathbf{a}(w_0d))].
\end{displaymath}
\end{enumerate}
\end{corollary}

\begin{proof}
The graded shift for the Serre functor on the category of perfect complexes 
over $\mathcal{O}_0^{\hat{\mathcal{R}}}$ can be determined by the condition
that it sends the projective object $\theta_d L_d\langle -\mathbf{a}(d)\rangle$
to the injective object $\theta_d L_d\langle \mathbf{a}(d)\rangle$.
Now the shift in \eqref{cor-s9.8-1.1} follows from the formulae
in Corollary~\ref{cor-s9.3-3} and Theorem~\ref{conj1}.
The remaining claims follow by combining Corollary~\ref{cor-s9.3-3} with
Theorem~\ref{conj1} and Theorem~\ref{prop-s8.25-1}.
\end{proof}

\section{Examples}\label{s7}

\subsection{Principal block of $\mathfrak{sl}_2$}\label{s7.1}

In the case  $\mathfrak{g}=\mathfrak{sl}_2$, the principal block 
$\mathcal{O}_0$ is equivalent to the category of modules over
the following quiver with relations, see \cite[Theorem~5.3.1]{Maz4}:
\begin{displaymath}
\xymatrix{
s\ar@/^2mm/[rr]^a&&e\ar@/^2mm/[ll]^b
},\qquad ab=0.
\end{displaymath}
Here are the graded diagrams of all structural modules in this case
(they are well-defined as all graded composition multiplicities
are either $0$ or $1$). For simplicity, the simple $L_s$ is 
displayed as $s$ and the simple $L_e$ as $e$:

\resizebox{\textwidth}{!}{
$
\xymatrix@C=3mm@R=3mm{
&L_s=T_s=\Delta_s=\nabla_s&L_e&P_s&P_e=\Delta_e&I_s&I_e=\nabla_e&T_e\\
-2&&&&&s\ar@{-}[d]^a&&\\
-1&&&&&e\ar@{-}[d]^b&s\ar@{-}[d]^a&s\ar@{-}[d]^a\\
0&s&e&s\ar@{-}[d]^a&e\ar@{-}[d]^b&s&e&e\ar@{-}[d]^b\\
1&&&e\ar@{-}[d]^b&s&&&s\\
2&&&s&&&&\\
}
$
}

Note that, up to a shift of grading,  all indecomposable object in $\mathcal{O}_0$
appear above.

A minimal graded projective resolution of {\color{magenta}$\nabla_e$} looks
as follows:
\begin{equation}\label{eq-ex1-1}
\xymatrix@C=7mm@R=3mm{
-1&&&&{\color{magenta}s}\ar@{-}@[magenta][d]\\
0&&&&{\color{magenta}e}\ar@{-}[d]\\
1&&&s\ar@{-}[d]\ar@{.>}[r]&s\\
2&&e\ar@{-}[d]\ar@{.>}[r]&e\ar@{-}[d]&\\
3&&s\ar@{.>}[r]&s&\\
}
\end{equation}
Here we directly see that
\begin{displaymath}
\mathrm{ext}^0(\nabla_e,\Delta_e\langle 2\rangle)=
\mathrm{ext}^1(\nabla_e,\Delta_e)=
\mathrm{ext}^2(\nabla_e,\Delta_e\langle -2\rangle)=\mathbb{C},
\end{displaymath}
where the natural transformation given by 
$\mathrm{ext}^0(\nabla_e,\Delta_e\langle 2\rangle)$
is constructed in Proposition~\ref{prop-s8.2-1} while
the natural transformation given by 
$\mathrm{ext}^2(\nabla_e,\Delta_e\langle -2\rangle)$
is constructed in Proposition~\ref{prop-s8.2-2}.

We have $\mathbf{D}=W=\{e,s\}$.
The object $L_e=\theta_e L_e$ is a Calabi-Yau object
of dimension $2$ and the object $T_e=\theta_s L_s$ 
is a Calabi-Yau object of dimension $0$. These exhaust
indecomposable Calabi-Yau objects.

\subsection{Principal block of $\mathfrak{sl}_3$}\label{s7.2}

Consider the principal block $\mathcal{O}_0$ for the Lie algebra 
$\mathfrak{sl}_3$. In this case $W=S_3=\{e,s,t,st,ts,sts=tst=w_0\}$.
The minimal graded tilting resolution of {\color{magenta}$\nabla_e$} looks as
follows:
\begin{displaymath}
0\to T_{w_0}\langle -3\rangle\to
T_{st}\langle -2\rangle\oplus T_{ts}\langle -2\rangle\to
T_s\langle -1\rangle\oplus T_t\langle -1\rangle\to 
T_e\to {\color{magenta}\nabla_e}\to 0.
\end{displaymath}
The minimal projective resolutions of the {\color{teal}indecomposable tilting
modules} look as follows:
\begin{gather*}
0\to P_{w_0}\langle 3\rangle \to {\color{teal}T_e}\to 0;\\
0\to P_{t} \to
P_{w_0}\langle 2\rangle \to {\color{teal}T_s}\to 0;\\
0\to P_{s} \to
P_{w_0}\langle 2\rangle \to {\color{teal}T_t}\to 0;\\
0\to P_{ts} \to
P_{w_0}\langle 1\rangle \to {\color{teal}T_{st}}\to 0;\\
0\to P_{st} \to
P_{w_0}\langle 1\rangle \to {\color{teal}T_{ts}}\to 0;\\
0\to P_{e}\langle -3\rangle \to
P_{s}\langle -2\rangle \oplus P_{t}\langle -2\rangle \to
P_{st}\langle -1\rangle \oplus P_{ts}\langle -1\rangle \to
P_{w_0} \to {\color{teal}T_{w_0}}\to 0.
\end{gather*}
Here $T_{w_0}$ is simple so the last linear complex is just 
computed using the Koszul duality. All other are obtained from the
last one applying indecomposable projective functors
(and removing summands that are homotopic to zero).
This implies that {\color{magenta}$\nabla_e$} has a projective resolution
of the form:
\begin{multline*}
0\to P_{e}\langle -6\rangle \to
P_{s}\langle -5\rangle \oplus P_{t}\langle -5\rangle \to \\ \to
P_{st}\langle -4\rangle \oplus P_{ts}\langle -4\rangle \to
P_{w_0}\langle -3\rangle\oplus P_{ts}\langle -2\rangle\oplus P_{st}\langle -2\rangle
\\
\to P_{w_0}\langle -1\rangle\oplus P_{w_0}\langle -1\rangle
\oplus P_{t}\langle -1\rangle\oplus P_{s}\langle -1\rangle
\to \\ \to 
P_{w_0}\langle 1\rangle\oplus P_{w_0}\langle 1\rangle
\to P_{w_0}\langle 3\rangle \to  {\color{magenta}\nabla_e}\to 0 .
\end{multline*}
As there are no potential cancellations, this resolution is minimal,
that is, coincides with $\mathcal{P}_\bullet(\nabla_e)$.
The Duflo elements are $\mathbf{D}=\{e,s,t,w_0\}$ an the corresponding
$\mathbf{a}$-values are:
\begin{displaymath}
\begin{array}{c||c|c|c|c}
d&e&s&t&w_0\\
\hline
\mathbf{a}(d)&0&1&1&3
\end{array}
\end{displaymath}
The summands of the $\mathcal{P}_\bullet(\nabla_e)$
which are relevant for Theorem~\ref{prop-s8.25-1} are:
\begin{itemize}
\item $P_e\langle -6\rangle[6]$ which corresponds to a non-zero
element in $\mathrm{ext}^6(\nabla_e,\Delta_e\langle -6\rangle)$;
\item $P_s\langle -1\rangle[2]$ which corresponds to a non-zero
element in $\mathrm{ext}^2(\nabla_e,\Delta_e)$;
\item $P_t\langle -1\rangle[2]$ which corresponds to a non-zero
element in $\mathrm{ext}^2(\nabla_e,\Delta_e)$;
\item $P_{w_0}\langle 3\rangle$ which corresponds to a non-zero
element in $\mathrm{ext}^0(\nabla_e,\Delta_e\langle 6\rangle)$ 
\end{itemize}
Note that $\mathrm{ext}^2(\nabla_e,\Delta_e)$ has dimension $1$
in this case, due to Proposition~\ref{prop-s6.4-1}
(in other words, the two non-zero elements of this space
corresponding to $P_s\langle -1\rangle[2]$ and $P_t\langle -1\rangle[2]$
are linearly dependent). To see  how this works,
we need to apply $\top_{w_0}$ to $\mathcal{P}_\bullet(\nabla_e)$
and then construct a projective resolution of the outcome.
In homological positions $-3$, $-2$, $-1$ and $0$, we get:
\begin{multline*}
P_{w_0}\oplus P_{w_0}\oplus P_{w_0}
\oplus P_{st}\langle -1\rangle \oplus P_{ts}\langle -1\rangle
\\
\to P_{w_0}\langle 2\rangle\oplus P_{w_0}\langle 2\rangle
\oplus P_{w_0}\oplus P_{w_0}
\to 
P_{w_0}\langle 4\rangle\oplus P_{w_0}\langle 4\rangle
\to P_{w_0}\langle 6\rangle \to  0
\end{multline*}
And here we see some potential for cancellation 
of summands $P_{w_0}$ in homological positions $-3$ and $-2$.
By Proposition~\ref{prop-s6.4-1}, one copy of $P_{w_0}$
must survive in homological position $-2$, so exactly one
such summand should be removed in each of these 
two positions to obtain a minimal projective resolution.
The remaining part
\begin{displaymath}
\to P_{w_0}\langle 2\rangle\oplus P_{w_0}\langle 2\rangle
\oplus P_{w_0}\to P_{w_0}\langle 4\rangle\oplus P_{w_0}\langle 4\rangle
\to P_{w_0}\langle 6\rangle \to  0
\end{displaymath}
mimics the beginning
\begin{displaymath}
\mathtt{C}\langle -2\rangle\oplus \mathtt{C}\langle -2\rangle\oplus 
\mathtt{C}\langle -3\rangle
\to
\mathtt{C}\langle -1\rangle\oplus \mathtt{C}\langle -1\rangle
\to\mathtt{C}\to {\color{cyan}\mathbb{C}}\to 0 
\end{displaymath}
of a minimal projective resolution of the simple
$\mathtt{C}$-module {\color{cyan}$\mathbb{C}$}
(note the overall grading shift and the fact that
the generators  of $\mathtt{C}$ have degree $2$
when related to  $\mathcal{O}_0$ via $\mathrm{End}(P_{w_0})$).

\subsection{Shifts in type $A_6$}\label{s7.3}

In type $A_6$, we have $W=S_7$. Two-sided cells of $S_7$
are in bijection with partitions of $7$ and have the 
following values of the $\mathbf{a}$-function:
\begin{displaymath}
\begin{array}{c||c|c|c|c|c|c|c|c|c}
\lambda&(7)&(6,1)&(5,2)&(5,1^2)&(4,3)&(4,2,1)&(4,1^3)&(3^2,1)&(3,2^2)\\
\hline
\mathbf{a}(\lambda)&0&1&2&3&3&4&6&5&6
\end{array}
\end{displaymath}
\begin{displaymath}
\begin{array}{c||c|c|c|c|c|c|c|c|c}
\lambda&(3,2,1^2)&(3,1^4)&(2^3,1)&(2^2,1^3)&(2,1^5)&(1^7)\\
\hline
\mathbf{a}(\lambda)&7&10&9&11&15&21
\end{array}
\end{displaymath}
Figure~\ref{fig1} represents graded shifts appearing in Theorem~\ref{conj1}
plotted in the coordinate plane. As usual, the horizontal axes depicts
homological position with positive shifts going to the left 
while the vertical axes depicts the grading with positive shifts going up.
Each magenta dot represents a two-sided cell and the partition corresponding
to that cell is written next to that dot. The position of the dot in the 
coordinate plane represents the shifts. One can note that there are 
two homological positions with two dots, representing two pairs of 
two-sided cells with the same $\mathbf{a}$-value. The area between the 
two dashed lines is the area of potential tops of $P_{w_0}$ summands as
predicted by $\mathrm{Ext}^*_\mathtt{C}(\mathbb{C},\mathbb{C})$
via Proposition~\ref{prop-s6.4-1}. Note that the magenta dot for
$(2,1^5)$ is on the lower dashed line.

\begin{figure}
\resizebox{!}{\textheight}{
\begin{tikzpicture}
\draw[gray, thin,  ->] (1,0) -- (-43,0) node[anchor=south west] {$\left[ a\right]$};
\draw[gray, thin,  ->] (0,-43) -- (0,43) node[anchor=north west] {$\langle b\rangle$};
\draw[gray,fill=gray] (0,0) circle (.3ex) node[anchor=north east] {\color{gray}$0$};
\draw[gray,fill=gray] (1,0) circle (.3ex) node[anchor=north east] {\color{gray}$-1$};
\draw[gray,fill=gray] (-1,0) circle (.3ex) node[anchor=north east] {\color{gray}$1$};
\draw[gray,fill=gray] (-2,0) circle (.3ex) node[anchor=north east] {\color{gray}$2$};
\draw[gray,fill=gray] (-3,0) circle (.3ex) node[anchor=north east] {\color{gray}$3$};
\draw[gray,fill=gray] (-4,0) circle (.3ex) node[anchor=north east] {\color{gray}$4$};
\draw[gray,fill=gray] (-5,0) circle (.3ex) node[anchor=north east] {\color{gray}$5$};
\draw[gray,fill=gray] (-6,0) circle (.3ex) node[anchor=north east] {\color{gray}$6$};
\draw[gray,fill=gray] (-7,0) circle (.3ex) node[anchor=north east] {\color{gray}$7$};
\draw[gray,fill=gray] (-8,0) circle (.3ex) node[anchor=north east] {\color{gray}$8$};
\draw[gray,fill=gray] (-9,0) circle (.3ex) node[anchor=north east] {\color{gray}$9$};
\draw[gray,fill=gray] (-10,0) circle (.3ex) node[anchor=north east] {\color{gray}$10$};
\draw[gray,fill=gray] (-11,0) circle (.3ex) node[anchor=north east] {\color{gray}$11$};
\draw[gray,fill=gray] (-12,0) circle (.3ex) node[anchor=north east] {\color{gray}$12$};
\draw[gray,fill=gray] (-13,0) circle (.3ex) node[anchor=north east] {\color{gray}$13$};
\draw[gray,fill=gray] (-14,0) circle (.3ex) node[anchor=north east] {\color{gray}$14$};
\draw[gray,fill=gray] (-15,0) circle (.3ex) node[anchor=north east] {\color{gray}$15$};
\draw[gray,fill=gray] (-16,0) circle (.3ex) node[anchor=north east] {\color{gray}$16$};
\draw[gray,fill=gray] (-17,0) circle (.3ex) node[anchor=north east] {\color{gray}$17$};
\draw[gray,fill=gray] (-18,0) circle (.3ex) node[anchor=north east] {\color{gray}$18$};
\draw[gray,fill=gray] (-19,0) circle (.3ex) node[anchor=north east] {\color{gray}$19$};
\draw[gray,fill=gray] (-20,0) circle (.3ex) node[anchor=north east] {\color{gray}$20$};
\draw[gray,fill=gray] (-21,0) circle (.3ex) node[anchor=north east] {\color{gray}$21$};
\draw[gray,fill=gray] (-22,0) circle (.3ex) node[anchor=north east] {\color{gray}$22$};
\draw[gray,fill=gray] (-23,0) circle (.3ex) node[anchor=north east] {\color{gray}$23$};
\draw[gray,fill=gray] (-24,0) circle (.3ex) node[anchor=north east] {\color{gray}$24$};
\draw[gray,fill=gray] (-25,0) circle (.3ex) node[anchor=north east] {\color{gray}$25$};
\draw[gray,fill=gray] (-26,0) circle (.3ex) node[anchor=north east] {\color{gray}$26$};
\draw[gray,fill=gray] (-27,0) circle (.3ex) node[anchor=north east] {\color{gray}$27$};
\draw[gray,fill=gray] (-28,0) circle (.3ex) node[anchor=north east] {\color{gray}$28$};
\draw[gray,fill=gray] (-29,0) circle (.3ex) node[anchor=north east] {\color{gray}$29$};
\draw[gray,fill=gray] (-30,0) circle (.3ex) node[anchor=north east] {\color{gray}$30$};
\draw[gray,fill=gray] (-31,0) circle (.3ex) node[anchor=north east] {\color{gray}$31$};
\draw[gray,fill=gray] (-32,0) circle (.3ex) node[anchor=north east] {\color{gray}$32$};
\draw[gray,fill=gray] (-33,0) circle (.3ex) node[anchor=north east] {\color{gray}$33$};
\draw[gray,fill=gray] (-34,0) circle (.3ex) node[anchor=north east] {\color{gray}$34$};
\draw[gray,fill=gray] (-35,0) circle (.3ex) node[anchor=north east] {\color{gray}$35$};
\draw[gray,fill=gray] (-36,0) circle (.3ex) node[anchor=north east] {\color{gray}$36$};
\draw[gray,fill=gray] (-37,0) circle (.3ex) node[anchor=north east] {\color{gray}$37$};
\draw[gray,fill=gray] (-38,0) circle (.3ex) node[anchor=north east] {\color{gray}$38$};
\draw[gray,fill=gray] (-39,0) circle (.3ex) node[anchor=north east] {\color{gray}$39$};
\draw[gray,fill=gray] (-40,0) circle (.3ex) node[anchor=north east] {\color{gray}$40$};
\draw[gray,fill=gray] (-41,0) circle (.3ex) node[anchor=north east] {\color{gray}$41$};
\draw[gray,fill=gray] (-42,0) circle (.3ex) node[anchor=north east] {\color{gray}$42$};
\draw[gray,fill=gray] (0,1) circle (.3ex) node[anchor=north east] {\color{gray}$1$};
\draw[gray,fill=gray] (0,2) circle (.3ex) node[anchor=north east] {\color{gray}$2$};
\draw[gray,fill=gray] (0,3) circle (.3ex) node[anchor=north east] {\color{gray}$3$};
\draw[gray,fill=gray] (0,4) circle (.3ex) node[anchor=north east] {\color{gray}$4$};
\draw[gray,fill=gray] (0,5) circle (.3ex) node[anchor=north east] {\color{gray}$5$};
\draw[gray,fill=gray] (0,6) circle (.3ex) node[anchor=north east] {\color{gray}$6$};
\draw[gray,fill=gray] (0,7) circle (.3ex) node[anchor=north east] {\color{gray}$7$};
\draw[gray,fill=gray] (0,8) circle (.3ex) node[anchor=north east] {\color{gray}$8$};
\draw[gray,fill=gray] (0,9) circle (.3ex) node[anchor=north east] {\color{gray}$9$};
\draw[gray,fill=gray] (0,10) circle (.3ex) node[anchor=north east] {\color{gray}$10$};
\draw[gray,fill=gray] (0,11) circle (.3ex) node[anchor=north east] {\color{gray}$11$};
\draw[gray,fill=gray] (0,12) circle (.3ex) node[anchor=north east] {\color{gray}$12$};
\draw[gray,fill=gray] (0,13) circle (.3ex) node[anchor=north east] {\color{gray}$13$};
\draw[gray,fill=gray] (0,14) circle (.3ex) node[anchor=north east] {\color{gray}$14$};
\draw[gray,fill=gray] (0,15) circle (.3ex) node[anchor=north east] {\color{gray}$15$};
\draw[gray,fill=gray] (0,16) circle (.3ex) node[anchor=north east] {\color{gray}$16$};
\draw[gray,fill=gray] (0,17) circle (.3ex) node[anchor=north east] {\color{gray}$17$};
\draw[gray,fill=gray] (0,18) circle (.3ex) node[anchor=north east] {\color{gray}$18$};
\draw[gray,fill=gray] (0,19) circle (.3ex) node[anchor=north east] {\color{gray}$19$};
\draw[gray,fill=gray] (0,20) circle (.3ex) node[anchor=north east] {\color{gray}$20$};
\draw[gray,fill=gray] (0,21) circle (.3ex) node[anchor=north east] {\color{gray}$21$};
\draw[gray,fill=gray] (0,22) circle (.3ex) node[anchor=north east] {\color{gray}$22$};
\draw[gray,fill=gray] (0,23) circle (.3ex) node[anchor=north east] {\color{gray}$23$};
\draw[gray,fill=gray] (0,24) circle (.3ex) node[anchor=north east] {\color{gray}$24$};
\draw[gray,fill=gray] (0,25) circle (.3ex) node[anchor=north east] {\color{gray}$25$};
\draw[gray,fill=gray] (0,26) circle (.3ex) node[anchor=north east] {\color{gray}$26$};
\draw[gray,fill=gray] (0,27) circle (.3ex) node[anchor=north east] {\color{gray}$27$};
\draw[gray,fill=gray] (0,28) circle (.3ex) node[anchor=north east] {\color{gray}$28$};
\draw[gray,fill=gray] (0,29) circle (.3ex) node[anchor=north east] {\color{gray}$29$};
\draw[gray,fill=gray] (0,30) circle (.3ex) node[anchor=north east] {\color{gray}$30$};
\draw[gray,fill=gray] (0,31) circle (.3ex) node[anchor=north east] {\color{gray}$31$};
\draw[gray,fill=gray] (0,32) circle (.3ex) node[anchor=north east] {\color{gray}$32$};
\draw[gray,fill=gray] (0,33) circle (.3ex) node[anchor=north east] {\color{gray}$33$};
\draw[gray,fill=gray] (0,34) circle (.3ex) node[anchor=north east] {\color{gray}$34$};
\draw[gray,fill=gray] (0,35) circle (.3ex) node[anchor=north east] {\color{gray}$35$};
\draw[gray,fill=gray] (0,36) circle (.3ex) node[anchor=north east] {\color{gray}$36$};
\draw[gray,fill=gray] (0,37) circle (.3ex) node[anchor=north east] {\color{gray}$37$};
\draw[gray,fill=gray] (0,38) circle (.3ex) node[anchor=north east] {\color{gray}$38$};
\draw[gray,fill=gray] (0,39) circle (.3ex) node[anchor=north east] {\color{gray}$39$};
\draw[gray,fill=gray] (0,40) circle (.3ex) node[anchor=north east] {\color{gray}$40$};
\draw[gray,fill=gray] (0,41) circle (.3ex) node[anchor=north east] {\color{gray}$41$};
\draw[gray,fill=gray] (0,42) circle (.3ex) node[anchor=north east] {\color{gray}$42$};
\draw[gray,fill=gray] (0,-1) circle (.3ex) node[anchor=north east] {\color{gray}$-1$};
\draw[gray,fill=gray] (0,-2) circle (.3ex) node[anchor=north east] {\color{gray}$-2$};
\draw[gray,fill=gray] (0,-3) circle (.3ex) node[anchor=north east] {\color{gray}$-3$};
\draw[gray,fill=gray] (0,-4) circle (.3ex) node[anchor=north east] {\color{gray}$-4$};
\draw[gray,fill=gray] (0,-5) circle (.3ex) node[anchor=north east] {\color{gray}$-5$};
\draw[gray,fill=gray] (0,-6) circle (.3ex) node[anchor=north east] {\color{gray}$-6$};
\draw[gray,fill=gray] (0,-7) circle (.3ex) node[anchor=north east] {\color{gray}$-7$};
\draw[gray,fill=gray] (0,-8) circle (.3ex) node[anchor=north east] {\color{gray}$-8$};
\draw[gray,fill=gray] (0,-9) circle (.3ex) node[anchor=north east] {\color{gray}$-9$};
\draw[gray,fill=gray] (0,-10) circle (.3ex) node[anchor=north east] {\color{gray}$-10$};
\draw[gray,fill=gray] (0,-11) circle (.3ex) node[anchor=north east] {\color{gray}$-11$};
\draw[gray,fill=gray] (0,-12) circle (.3ex) node[anchor=north east] {\color{gray}$-12$};
\draw[gray,fill=gray] (0,-13) circle (.3ex) node[anchor=north east] {\color{gray}$-13$};
\draw[gray,fill=gray] (0,-14) circle (.3ex) node[anchor=north east] {\color{gray}$-14$};
\draw[gray,fill=gray] (0,-15) circle (.3ex) node[anchor=north east] {\color{gray}$-15$};
\draw[gray,fill=gray] (0,-16) circle (.3ex) node[anchor=north east] {\color{gray}$-16$};
\draw[gray,fill=gray] (0,-17) circle (.3ex) node[anchor=north east] {\color{gray}$-17$};
\draw[gray,fill=gray] (0,-18) circle (.3ex) node[anchor=north east] {\color{gray}$-18$};
\draw[gray,fill=gray] (0,-19) circle (.3ex) node[anchor=north east] {\color{gray}$-19$};
\draw[gray,fill=gray] (0,-20) circle (.3ex) node[anchor=north east] {\color{gray}$-20$};
\draw[gray,fill=gray] (0,-21) circle (.3ex) node[anchor=north east] {\color{gray}$-21$};
\draw[gray,fill=gray] (0,-22) circle (.3ex) node[anchor=north east] {\color{gray}$-22$};
\draw[gray,fill=gray] (0,-23) circle (.3ex) node[anchor=north east] {\color{gray}$-23$};
\draw[gray,fill=gray] (0,-24) circle (.3ex) node[anchor=north east] {\color{gray}$-24$};
\draw[gray,fill=gray] (0,-25) circle (.3ex) node[anchor=north east] {\color{gray}$-25$};
\draw[gray,fill=gray] (0,-26) circle (.3ex) node[anchor=north east] {\color{gray}$-26$};
\draw[gray,fill=gray] (0,-27) circle (.3ex) node[anchor=north east] {\color{gray}$-27$};
\draw[gray,fill=gray] (0,-28) circle (.3ex) node[anchor=north east] {\color{gray}$-28$};
\draw[gray,fill=gray] (0,-29) circle (.3ex) node[anchor=north east] {\color{gray}$-29$};
\draw[gray,fill=gray] (0,-30) circle (.3ex) node[anchor=north east] {\color{gray}$-30$};
\draw[gray,fill=gray] (0,-31) circle (.3ex) node[anchor=north east] {\color{gray}$-31$};
\draw[gray,fill=gray] (0,-32) circle (.3ex) node[anchor=north east] {\color{gray}$-32$};
\draw[gray,fill=gray] (0,-33) circle (.3ex) node[anchor=north east] {\color{gray}$-33$};
\draw[gray,fill=gray] (0,-34) circle (.3ex) node[anchor=north east] {\color{gray}$-34$};
\draw[gray,fill=gray] (0,-35) circle (.3ex) node[anchor=north east] {\color{gray}$-35$};
\draw[gray,fill=gray] (0,-36) circle (.3ex) node[anchor=north east] {\color{gray}$-36$};
\draw[gray,fill=gray] (0,-37) circle (.3ex) node[anchor=north east] {\color{gray}$-37$};
\draw[gray,fill=gray] (0,-38) circle (.3ex) node[anchor=north east] {\color{gray}$-38$};
\draw[gray,fill=gray] (0,-39) circle (.3ex) node[anchor=north east] {\color{gray}$-39$};
\draw[gray,fill=gray] (0,-40) circle (.3ex) node[anchor=north east] {\color{gray}$-40$};
\draw[gray,fill=gray] (0,-41) circle (.3ex) node[anchor=north east] {\color{gray}$-41$};
\draw[gray,fill=gray] (0,-42) circle (.3ex) node[anchor=north east] {\color{gray}$-42$};
\draw[magenta,fill=magenta] (0,42) circle (.4ex) node[anchor=west] {\color{magenta}\large$(1^7)$};
\draw[magenta,fill=magenta] (-2,28) circle (.4ex) node[anchor=west] {\color{magenta}\large$(2,1^5)$};
\draw[magenta,fill=magenta] (-4,18) circle (.4ex) node[anchor=west] {\color{magenta}\large$(2^2,1^3)$};
\draw[magenta,fill=magenta] (-6,12) circle (.4ex) node[anchor=west] {\color{magenta}\large$(2^3,1)$};
\draw[magenta,fill=magenta] (-6,14) circle (.4ex) node[anchor=west] {\color{magenta}\large$(3,1^4)$};
\draw[magenta,fill=magenta] (-8,8) circle (.4ex) node[anchor=west] {\color{magenta}\large$(3,2,1^2)$};
\draw[magenta,fill=magenta] (-10,2) circle (.4ex) node[anchor=west] {\color{magenta}\large$(3,2^2)$};
\draw[magenta,fill=magenta] (-12,-2) circle (.4ex) node[anchor=west] {\color{magenta}\large$(3^2,1)$};
\draw[magenta,fill=magenta] (-12,-2) circle (.4ex) node[anchor=west] {\color{magenta}\large$(3^2,1)$};
\draw[magenta,fill=magenta] (-12,0) circle (.4ex) node[anchor=west] {\color{magenta}\large$(4,1^3)$};
\draw[magenta,fill=magenta] (-14,-6) circle (.4ex) node[anchor=west] {\color{magenta}\large$(4,2,1)$};
\draw[magenta,fill=magenta] (-18,-12) circle (.4ex) node[anchor=west] {\color{magenta}\large$(4,3)$};
\draw[magenta,fill=magenta] (-20,-14) circle (.4ex) node[anchor=west] {\color{magenta}\large$(5,1^2)$};
\draw[magenta,fill=magenta] (-22,-18) circle (.4ex) node[anchor=west] {\color{magenta}\large$(5,2)$};
\draw[magenta,fill=magenta] (-30,-28) circle (.4ex) node[anchor=west] {\color{magenta}\large$(6,1)$};
\draw[magenta,fill=magenta] (-42,-42) circle (.4ex) node[anchor=west] {\color{magenta}\large$(7)$};
\draw[gray, dashed] (0,42) -- (-42,-42);
\draw[gray, dashed] (0,42) -- (-12,-42);
\end{tikzpicture}
}
\caption{Shifts for $S_7$}\label{fig1}
\end{figure}
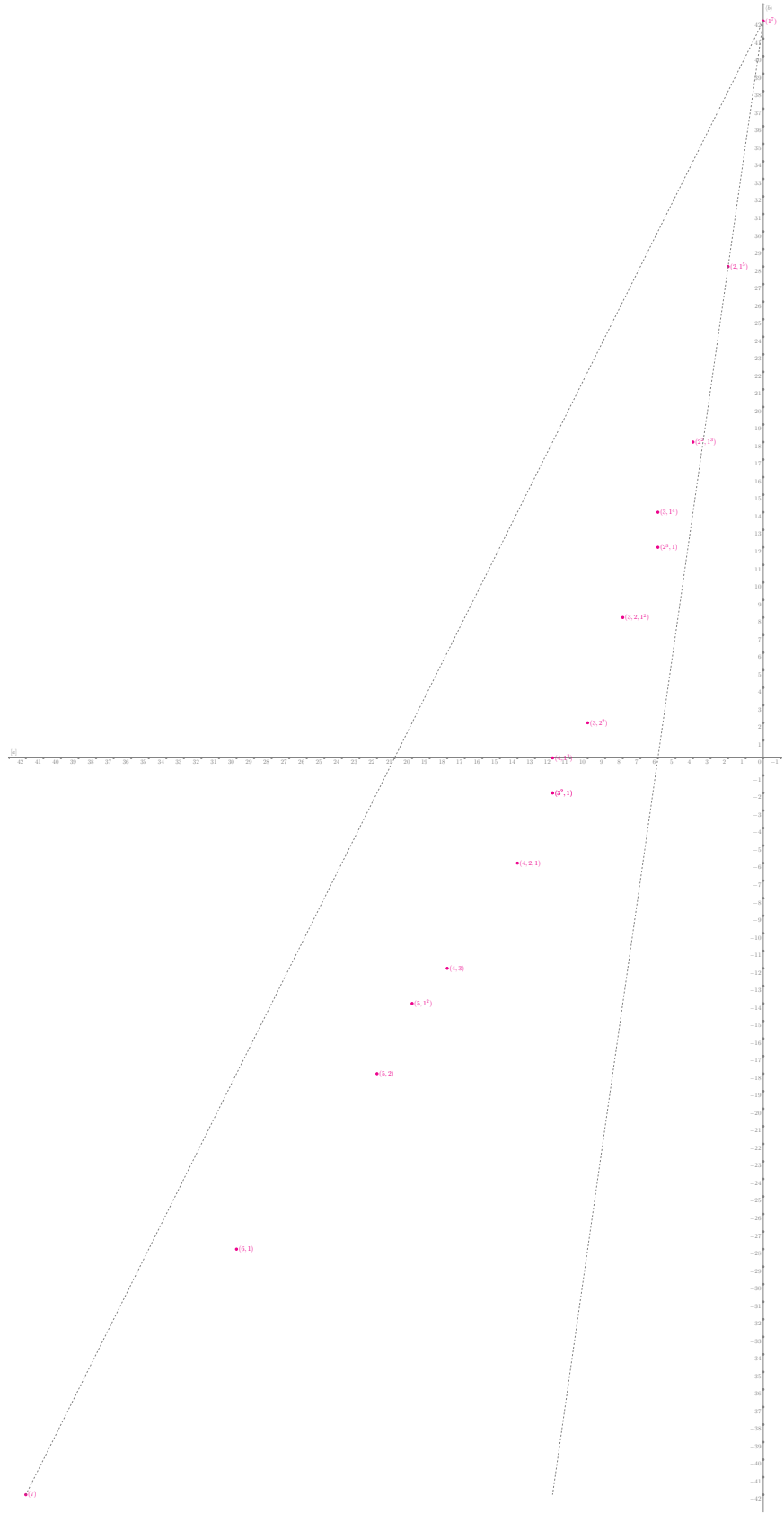

\subsection{Parabolic category of type $A_{n-1}/A_{n-2}$}\label{s7.4}

Consider $\mathfrak{g}=\mathfrak{sl}_n$, for $n\geq 2$, and a parabolic
subalgebra $\mathfrak{p}$ of $\mathfrak{g}$ containing $\mathfrak{b}$,
the semi-simple part of the Levi quotient of which is isomorphic to
$\mathfrak{sl}_{n-1}$, embedded into the upper left corner. 
Then $W\cong S_n$, the Weyl group of type
$A_{n-1}$, with Dynkin diagram
\begin{displaymath}
\xymatrix{s_1\ar@{-}[r]&s_2\ar@{-}[r]&s_3\ar@{-}[r]&\dots\ar@{-}[r]&s_{n-1}}, 
\end{displaymath}
and the group $W^\mathfrak{p}$ is the parabolic subgroup of $W$ generated by
the simple reflections $s_1$, $s_2,\dots$, $s_{n-2}$. We have 
$W^\mathfrak{p}\cong S_{n-1}$, the Weyl group of type
$A_{n-2}$, and
\begin{displaymath}
(W^\mathfrak{p}\setminus W)_{\mathrm{short}}=
\{e,s_{n-1},s_{n-1}s_{n-2},\dots, s_{n-1}s_{n-2}\dots s_2s_1\}.
\end{displaymath}
Set $\mathtt{1}:=s_{n-1}s_{n-2}\dots s_2s_1$,
$\mathtt{2}:=s_{n-1}s_{n-2}\dots s_2$, and so on
up to $\mathtt{n}:=e$. Then the category $\mathcal{O}_0^\mathfrak{p}$
is equivalent to the category of modules over the following
quiver:
\begin{displaymath}
\xymatrix{
\mathtt{1}\ar@/^2mm/[r]^{\alpha_1}&
\mathtt{2}\ar@/^2mm/[r]^{\alpha_2}\ar@/^2mm/[l]^{\beta_1}&
\dots\ar@/^2mm/[r]^{\alpha_{n-2}}\ar@/^2mm/[l]^{\beta_2}&
\mathtt{n}\text{-}\mathtt{1}\ar@/^2mm/[r]^{\alpha_{n-1}}\ar@/^2mm/[l]^{\beta_{n-2}}
&\mathtt{n}\ar@/^2mm/[l]^{\beta_{n-1}}} 
\end{displaymath}
with the following relations, for $i=1,2,\dots,n-2$:
\begin{displaymath}
\alpha_{i+1}\alpha_{i}=0,\quad
\beta_{i}\beta_{i+1}=0,\quad
\alpha_{i}\beta_{i}=\beta_{i+1}\alpha_{i+1},\quad
\alpha_{n-1}\beta_{n-1}=0.
\end{displaymath}
For further properties of this algebra, see \cite{PW}.
Here are the graded diagrams of projective objects in 
$\mathcal{O}_0^\mathfrak{p}$:
\begin{displaymath}
\xymatrix@C=5mm@R=5mm{
P^\mathfrak{p}_1&&P^\mathfrak{p}_2&&\dots&&P^\mathfrak{p}_{n-1}&&P^\mathfrak{p}_n\\
\mathtt{1}\ar@{-}[d]&&\mathtt{2}\ar@{-}[dl]\ar@{-}[dr]
&&&&\mathtt{n}\text{-}\mathtt{1}\ar@{-}[dl]\ar@{-}[dr]&&\mathtt{n}\ar@{-}[d]\\
\mathtt{2}\ar@{-}[d]&\mathtt{1}\ar@{-}[dr]&&\mathtt{3}\ar@{-}[dl]&\dots&
\mathtt{n}\text{-}\mathtt{2}\ar@{-}[dr]&&\mathtt{n}\ar@{-}[dl]&\mathtt{n}\text{-}\mathtt{1}\\
\mathtt{1}&&\mathtt{2}&&&&\mathtt{n}\text{-}\mathtt{1}&&
}
\end{displaymath}
In particular, the projectives $P^\mathfrak{p}_1$, $P^\mathfrak{p}_2$,\dots, $P^\mathfrak{p}_{n-1}$
are also injective. As all these projective-injective modules have
isomorphic top and socle, it follows that they all are 
Calabi-Yau objects of dimension $0$ for $\mathcal{O}^\mathfrak{p}$. 

The set $(W^\mathfrak{p}\setminus W)_{\mathrm{short}}$ splits into
two KL-right cells: $\mathcal{R}_0:=\{e\}$ and 
$\mathcal{R}_1$, containing all the remaining elements.
The only element $e$ of $\mathcal{R}_0$ is a Duflo element.
The Duflo element in $\mathcal{R}_1$ is $s_{n-1}$. We have
$P^\mathfrak{p}_{n-1}=\theta_{s_{n-1}}L_{s_{n-1}}$, 
$P^\mathfrak{p}_{n-2}=\theta_{s_{n-2}}\theta_{s_{n-1}}L_{s_{n-1}}$, and so on.
Note that $\theta_{s_{n-2}}\theta_{s_{n-1}}=
\theta_{s_{n-1}s_{n-2}}$,
$\theta_{s_{n-3}}\theta_{s_{n-2}}\theta_{s_{n-1}}=
\theta_{s_{n-1}s_{n-2}s_{n-3}}$ and so on.
We have $\mathbf{a}(e)=0$, $\mathbf{a}(s_{n-1})=1$,
$\mathbf{a}(w_0)=\frac{n(n-1)}{2}$ and 
$\mathbf{a}(w_0s_{n-1})=\frac{(n-1)(n-2)}{2}$.

The injective module ${\color{magenta}I^\mathfrak{p}_{n}}\cong 
(P^\mathfrak{p}_n)^\star$ has the following
projective resolution:
\begin{multline*}
0\to P^\mathfrak{p}_n\langle -2-2n\rangle\to P^\mathfrak{p}_{n-1}\langle -1-2n\rangle\dots \to 
P^\mathfrak{p}_2\langle -4-n\rangle\to P^\mathfrak{p}_1\langle -3-n\rangle\to 
\\ \to P^\mathfrak{p}_1\langle -1-n\rangle\to \dots \to 
P^\mathfrak{p}_{n-2}\langle 0\rangle\to P^\mathfrak{p}_{n-1}\langle 1\rangle\to {\color{magenta}I^\mathfrak{p}_n}\to0.
\end{multline*}
Note that the only non-linear step in this resolution is the 
step between the positions $P^\mathfrak{p}_1\langle -3-n\rangle$ 
and $P^\mathfrak{p}_1\langle -1-n\rangle$,
where the difference in degree shifts is $2$ instead of $1$ in all
other steps.

The unique, up to scalar, non-zero map $P^\mathfrak{p}_{n-1}\langle 1\rangle\to
P^\mathfrak{p}_n\langle 2\rangle$ gives rise to a natural transformation from
the Serre functor to the identity (shifted by $2=2\mathbf{a}(s_{n-1})$)
and the evaluation of this natural transformation at all 
projective-injective objects (i.e. Calabi-Yau objects of dimension $0$) 
is an isomorphism.

The remaining Calabi-Yau object is $L_n=\theta_e L_e$. 
The unique, up to scalar, non-zero map $P^\mathfrak{p}_{n}\langle -2-2n\rangle\to
P^\mathfrak{p}_n\langle -2-2n\rangle$ (both at the homological position $2-2n$)
gives rise to a natural transformation from
the Serre functor to the appropriately shifted identity.
The evaluation of this natural transformation at $L_n$
is an isomorphism. Note that 
\begin{displaymath}
2(n-1)=2\left(\frac{n(n-1)}{2}-\frac{(n-1)(n-2)}{2}\right)
=2(\mathbf{a}(w_0)-\mathbf{a}(w_0s_{n-1})).
\end{displaymath}

\subsection{A non-parabolic example of $\mathcal{O}_0^{\hat{\mathcal{R}}}$ 
in type $A_3$}\label{s7.5}

Let $\mathfrak{g}=\mathfrak{sl}_4$. Then $W=S_4$, the Weyl group 
of type $A_3$ with Dynkin diagram
\begin{displaymath}
\xymatrix{r\ar@{-}[r]&s\ar@{-}[r]&t}.
\end{displaymath}
Let $\mathcal{R}$ be the KL-right cell in $W$ containing the simple reflection 
$s$. Then $\mathcal{R}=\{s,sr,st\}$ and $\hat{\mathcal{R}}$ contains one
additional element, namely $e$. Consider $\mathcal{O}_0^{\hat{\mathcal{R}}}$.
The projective objects in this category have the following graded diagrams:
\begin{displaymath}
\xymatrix@C=5mm@R=5mm{
P^{\hat{\mathcal{R}}}_{st}&&P^{\hat{\mathcal{R}}}_{sr}
&&&P^{\hat{\mathcal{R}}}_{s}&&P_e^{\hat{\mathcal{R}}}\\
st\ar@{-}[d]&&sr\ar@{-}[d]&&&s\ar@{-}[dl]\ar@{-}[dr]\ar@{-}[d]&&e\ar@{-}[d]\\
s\ar@{-}[d]&&s\ar@{-}[d]&&sr\ar@{-}[dr]&e\ar@{-}[d]&st\ar@{-}[dl]&s\\
st&&sr&&&s&&
}
\end{displaymath}
Consequently, the category $\mathcal{O}_0^{\hat{\mathcal{R}}}$ is equivalent
to the category of modules over the following quiver:
\begin{displaymath}
\xymatrix{
sr\ar@/^2mm/[r]^{\alpha}&
s\ar@/^2mm/[r]^{\beta}\ar@/^2mm/[d]^{\gamma}\ar@/^2mm/[l]^{\varepsilon}
&st\ar@/^2mm/[l]^{\tau}\\
&e\ar@/^2mm/[u]^{\delta}&} 
\end{displaymath}
with the following relations:
\begin{displaymath}
\beta\alpha=\gamma\alpha=\varepsilon\tau=\gamma\tau=
\beta\delta=\varepsilon\delta=\gamma\delta=0,\quad
\alpha\varepsilon=\tau\beta=\delta\gamma.
\end{displaymath}
We see that $P^{\hat{\mathcal{R}}}_s$, $P^{\hat{\mathcal{R}}}_{st}$ 
and $P^{\hat{\mathcal{R}}}_{sr}$ are injective. The injective
${\color{teal}I^{\hat{\mathcal{R}}}_e}=(P^{\hat{\mathcal{R}}}_e)^\star$ 
has the following projective resolution:
\begin{displaymath}
0\to P^{\hat{\mathcal{R}}}_e\langle -2\rangle
\to P^{\hat{\mathcal{R}}}_s \langle -1\rangle \to 
P^{\hat{\mathcal{R}}}_{sr}\oplus P^{\hat{\mathcal{R}}}_{st}
\to P^{\hat{\mathcal{R}}}_s\langle 1\rangle\to 
{\color{teal}I^{\hat{\mathcal{R}}}_e}\to 0. 
\end{displaymath}
In particular, the projective dimension of $I^{\hat{\mathcal{R}}}_e$ is finite.
That is, the condition \eqref{cor-s9.3-3.3} of Corollary~\ref{cor-s9.3-3}
is satisfied.

Kostant's problem has positive solution for $L_s$ by 
\cite[Theorem~1]{Maz5}. That is, the condition \eqref{cor-s9.3-3.1} 
of Corollary~\ref{cor-s9.3-3} is satisfied.
Unfortunately, the condition \eqref{cor-s9.3-3.2} of 
Corollary~\ref{cor-s9.3-3} is 
{\bf not} satisfied in this example. Let us see why.

Just as in the previous example, we have  
$\mathbf{a}(e)=0$, $\mathbf{a}(s)=1$, $\mathbf{a}(w_0)=6$, 
$\mathbf{a}(w_0s)=3$. So,  the condition \eqref{cor-s9.3-3.2} of 
Corollary~\ref{cor-s9.3-3} requires that  the projective
dimension of $P^{\hat{\mathcal{R}}}_e$ in $\mathcal{O}$ equals $6$.
We claim that this is not the case and, in fact, the projective
dimension of $P^{\hat{\mathcal{R}}}_e$ in $\mathcal{O}$ equals $10$.

To see this, consider the short exact sequence
\begin{equation}\label{eq-ex4}
0\to L_s\langle 1\rangle\to P^{\hat{\mathcal{R}}}_e\to L_e\to 0.
\end{equation}
The module $L_e$ has a minimal projective resolution in $\mathcal{O}$ of 
length $12=2\mathbf{a}(w_0)$, see \cite[Proposition~6]{Maz1}. Under the Koszul duality,
this resolution is mapped to the indecomposable projective-injective
object $I_{w_0}$. The latter has a (dual) Verma flag in which each
Verma module appears with multiplicity $1$. 

The module $L_s$ has a minimal projective resolution of 
length $11$, see \cite[Proposition~6]{Maz1}. Under the Koszul duality,
this resolution is mapped to the indecomposable injective
object $I_{w_0s}$. Note that $[\Delta_e:L_{w_0s}]=2$, see
\cite[Page~344]{St7}. Therefore, by the BGG reciprocity, 
$I_{w_0s}$ has a dual Verma flag in which $\nabla_e$,
$\nabla_r$, $\nabla_t$ and $\nabla_{rt}$ appear twice 
and all other $\nabla_w$, where $w\leq w_0s$ with respect to the
Bruhat order, appear once.

We can use now the sequence in \eqref{eq-ex4} to construct a 
projective resolution of $P^{\hat{\mathcal{R}}}_e$: there is 
a unique, up to scalar, non-zero homomorphism $\varphi$ from 
$I_{w_0}$ to $I_{w_0s}\langle 1\rangle$, which corresponds,
via Koszul duality, to a non-zero homomorphism $\psi$ from
$L_e$ to $L_s\langle -1\rangle[1]$. The cone of the latter 
is isomorphic to $P^{\hat{\mathcal{R}}}_e$. 

As the functor $\mathbb{V}$ is full and faithful on 
injective modules, let us look at the images of 
$I_{w_0}$ and $I_{w_0s}\langle 1\rangle$ under this functor.
The graded dimension vectors of the corresponding images are:
\begin{displaymath}
v=(1,0,3,0,5,0,6,0,5,0,{\color{violet}3},0,1)\quad\text{ and }\quad
w=(0,0,1,0,4,0,7,0,7,0,{\color{violet}4},0,1).
\end{displaymath}
Here the leftmost degree is $0$. The vector $v$ just records the
number of elements in $S_4$ of a fixed length: $S_4$ has
$1$ element of length $0$, then $3$ elements of length $1$ and so on.
The vector
$w$ is the sum
\begin{displaymath}
w=(0,0,1,0,4,0,6,0,5,0,3,0,1)+
(0,0,0,0,0,0,1,0,2,0,1,0,0),
\end{displaymath}
where the first summand describes the contribution, similar  to $v$, of the 
Bruhat interval $[e,w_0s]$ while the second summand describes
the contribution of the additional four dual Verma modules $\nabla_e$,
$\nabla_r$, $\nabla_t$ and $\nabla_{rt}$.

Note the {\color{violet}violet} entries $\color{violet}3$ and $\color{violet}4$
in degree $10$. As $4>3$, the cokernel of $\varphi$ will be non-zero
in that degree. Via Koszul duality, this means that 
the minimal projective resolution of $P^{\hat{\mathcal{R}}}_e$ will
have a non-zero entry at homological position $-10$. This implies 
that the projective dimension of $P^{\hat{\mathcal{R}}}_e$ 
in $\mathcal{O}$ is at least $10$.

As the map $\varphi$ is bijective in both degrees $11$ and $12$,
it follows that  the projective dimension of $P^{\hat{\mathcal{R}}}_e$ 
in $\mathcal{O}$ is exactly $10$.

\begin{remark}\label{rem-eq4}
{\rm 
If one instead takes as $\mathcal{R}$ the right cell of $r$,
then this cell contains $\{r,sr,tsr\}$. Similarly to the above, we
can compute the the projective dimension of $P^{\hat{\mathcal{R}}}_e$ 
in $\mathcal{O}$ in this case. The difference will be that 
$[\Delta_e:L_{w_0r}]=1$ and, consequently, the new vector $w$
will be $(0,0,1,0,3,0,5,0,5,0,3,0,1)$, where we only have singleton 
contributions by each element in $[e,w_0r]$. Therefore the difference
between $v$ and $w$ is $(1,0,2,0,2,0,1,0,0,0,0,0,0)$. Here the maximal
non-zero degree is $6$. This implies that the 
projective dimension of $P^{\hat{\mathcal{R}}}_e$  
in $\mathcal{O}$ in this case equals $6$, as required by 
the condition \eqref{cor-s9.3-3.2} of 
Corollary~\ref{cor-s9.3-3}.
}
\end{remark}

\vspace{2mm}

\noindent
V.~M.: Department of Mathematics, Uppsala University, Box. 480,
SE-75106, Uppsala,\\ SWEDEN, email: {\tt mazor\symbol{64}math.uu.se}

\end{document}